\documentclass[12pt]{article}
\usepackage{amsmath,amsfonts,amssymb,amsthm,mathrsfs}
\usepackage[a4paper,vmargin={3.5cm,3.5cm},hmargin={2.5cm,2.5cm}]{geometry}
\usepackage[margin=1cm]{caption}
\usepackage{graphicx,graphics}
\usepackage{epsfig}
\usepackage{latexsym}
\usepackage[utf8]{inputenc}
\usepackage{ae,aecompl}
\usepackage[english]{babel}
\usepackage[colorlinks=true]{hyperref}
\usepackage{enumerate}
\usepackage{bbm}
\usepackage{pxfonts}
\usepackage{dsfont}
\usepackage{tikz}
\usepackage{tabularx}
\usepackage{multirow}

\usepackage[auto]{contour}
\usetikzlibrary{shapes, angles, decorations.text, patterns, fadings,decorations.pathreplacing,fit}
\pgfdeclarelayer{edgelayer}
\pgfdeclarelayer{nodelayer}
\pgfsetlayers{edgelayer,nodelayer,main}

\date{}

\newtheorem{theorem}{Theorem}[]
\newtheorem{proposition}{Proposition}[section]
\newtheorem{lemma}[proposition]{Lemma}
\newtheorem{cor}[proposition]{Corollary}
\theoremstyle{definition}
\newtheorem{df}{Definition}[section]
\newtheorem{remark}{Remark}[section]

\newcommand{\F}{\mathcal{F}} 
\newcommand{\rad}{\operatorname{r}}
\newcommand{\dgr}{d_{\mathrm{gr}}}

\renewcommand{\P}{\mathbb{P}} 
\newcommand{\E}{\mathbb{E}} 
\newcommand{\V}{\mathbb{V}} 
\newcommand{\sQ}{\mathsf{Q}} 
\newcommand{\LT}{\mathsf{LT}} 
\newcommand{\TR}{\mathsf{T}^{(r)}} 
\newcommand{\T}{\mathsf{T}} 
\renewcommand{\epsilon}{\varepsilon}
\newcommand{\CLR}{Y}
\newcommand{\CLT}{X}
\newcommand{\VLT}{{\widetilde{X}}}
\newcommand{\st}{: }

\DeclareSymbolFont{extraup}{U}{zavm}{m}{n}
\DeclareMathSymbol{\varheart}{\mathalpha}{extraup}{86}
\DeclareMathSymbol{\vardiamond}{\mathalpha}{extraup}{87}

\makeatletter
\renewcommand*{\@fnsymbol}[1]{\ensuremath{\ifcase#1\or  \vardiamond \or \clubsuit\or \spadesuit\or
   \mathsection\or \mathparagraph\or \|\or **\or \dagger\dagger
   \or \ddagger\ddagger \else\@ctrerr\fi}}
\makeatother

\title{\bf \textsc{Polynomial mixing time of edge flips on quadrangulations}}
\author{Alessandra Caraceni\thanks{Department of Mathematical Sciences, University of Bath, UK.\hfill  \texttt{A.Caraceni@bath.ac.uk}} \hspace{12pt} \& \hspace{4pt} Alexandre Stauffer\thanks{Department of Mathematical Sciences, University of Bath, UK.\hfill  \texttt{A.Stauffer@bath.ac.uk}\newline The authors wish to acknowledge the support of EPSRC via the grant entitled ``Mathematical Analysis of Strongly Correlated Processes on Discrete Dynamic Structures''.}}
\begin{document}
\maketitle

\tikzstyle{vertex}=[circle, fill=black, inner sep=3pt, draw=white, ultra thick]
\tikzstyle{tree}=[ultra thick]
\tikzstyle{P}=[ultra thick, blue!60!white]
\tikzstyle{P'}=[ultra thick, green!70!black]
\tikzstyle{map}=[thick, red]
\tikzstyle{b}=[line width=4pt, white]

\abstract{We establish the first polynomial upper bound for the mixing time of random edge flips on rooted quadrangulations: we show that the spectral gap of the edge flip Markov chain on quadrangulations with $n$ faces admits, up to constants, an upper bound of $n^{-5/4}$ and a lower bound of $n^{-11/2}$. In order to obtain the lower bound, we also consider a very natural Markov chain on plane trees -- or, equivalently, on Dyck paths -- and improve the previous lower bound for its spectral gap by Shor and Movassagh.}

\section{Introduction}

Our work on quadrangulation edge flips places itself in the midst of a developing area of research whose origin can be partly traced back to a question of Aldous about triangulations of the $n$-gon~\cite{Ald94b}. The question concerns a discrete time edge flip Markov chain analogous to the one we will introduce, defined on the state space of triangulations of the regular $n$-gon (i.e.~on the possible sets of diagonals which partition the $n$-gon into triangular regions). A single step of the Markov chain, given a triangulation, consists of picking a diagonal at random, deleting it and replacing it with the opposite diagonal in the quadrilateral created by its absence. One would wish to analyse the growth of the mixing time of this chain (which is sometimes referred to as the \emph{triangulation walk}) as a function of the size $n$ of the triangulation (or, equivalently, of the size of the state space, which is exponential in $n$). Aldous conjectures an upper bound of $n^{3/2}$ (up to logarithmic factors in $n$) for the order of the relaxation time of this chain. In connection to this problem, he conjectures the same upper bound for a chain defined on $n$-cladograms, a type of binary tree structure with labelled leaves whose relevance also lies in its role as a formalisation of \emph{phylogenetic trees} from systematic biology, which model evolutionary relationships between species~\cite{Ald00}.

An important feature of triangulations of the $n$-gon is the fact that they are counted by Catalan numbers: more precisely, there are $C_{n-2}$ triangulations of the $n$-gon, where $C_n=\frac{1}{n+1}{2n\choose n}$. In fact, there is an extreme abundance of combinatorial structures which are counted by Catalan numbers, from Dyck paths to strings of matched parentheses to plane trees and beyond, with a thriving net of explicit bijections weaved between them, which often highlight surprising connections between the geometric features of different objects.

It seems therefore natural to attempt a systematic study of Markov chains defined on Catalan structures, but this task has proved very hard. For one thing, the natural notion of adjacency for different Catalan structures does not always translate well via sensible bijections, which gives rise to a rich panorama of different chains one might consider. But even concentrating on a single Markov chain has proved challenging so far, as attested by the relative scarceness of tight bounds for their mixing times, one notable exception being Wilson's result \cite{Wilson04} for adjacent transpositions on Dyck paths. 

Twenty years after a serious effort was started problems of this kind, we still do not have tight bounds for the mixing of the triangulation walk proposed by Aldous. Molloy, Reed and Steiger showed an $\Omega(n^\frac32)$ lower bound for its mixing time~\cite{MRS99}, while the best upper bound to date is McShine and Tetali's $O(n^5\log n)$ obtained in \cite{Catalan97}, where they analyse Markov chains on a number of other Catalan structures. As for $n$-cladograms, while the conjecture Aldous made in conjunction to the triangulation walk remains open, some chains have proved easier to analyse: Aldous himself showed an upper bound of $O(n^3)$ for the relaxation time of a particular chain \cite{Ald00}, improved to $O(n^2)$, which is tight, by Schweinsberg \cite{S02}; also note L\"ohr, Mytnik and Winter's work on the chain in the diffusion limit \cite{LMW18} as well as Forman, Pal, Rizzolo and Winkel's in a similar vein \cite{2018arXiv180907756F}. Furthermore, recent results for the mixing of a very natural chain on Dyck paths were obtained by Cohen, Tetali and Yeliussizov \cite{Catalan15} by rephrasing it as a basis exchange walk on a balanced matroid.

On the other hand, for many natural chains on Catalan structures, triangulations and related objects not even a polynomial upper bound for the mixing time is known. One such example is that of lattice triangulations, where polynomial bounds are only known for biased versions of the chain~\cite{CMSS15,CMSS16,S17}; see also works on rectangular dissections, for which polynomial bounds were obtained very recently \cite{CMRdyadic, CLSdyadic}. 

One may also consider edge flip Markov chains on planar maps, and in particular on the set of $p$-angulations of the sphere of size $n$ (with $p\geq 3$), that is the set of spherically embedded connected planar multigraphs with $n$ faces of degree $p$ (considered equivalent under orientation-preserving homeomorphisms of the sphere). An edge flip Markov chain on this state space can be defined as follows: at each step, an edge is selected uniformly at random, erased and replaced with one of the edges that can be drawn within the face of perimeter $2p-2$ left behind in order to form two faces of degree $p$. The only result shown so far for this chain pertains to the case of triangulations of the sphere ($p=3$), where the mixing time is known to be of order at least $n^{\frac54}$~\cite{B17}. No polynomial upper bound on the mixing time was known, prior to this work, for any $p\geq 3$.

In this paper we consider the case of rooted quadrangulations with $n$ faces (i.e~the case $p=4$, where maps are endowed with a distinguished oriented edge) and derive the first polynomial upper bound on the mixing time.

Note that quadrangulations in particular occupy a privileged position within the panorama of planar maps, mainly thanks to the famous bijections first developed by Cori, Vanquelin and Schaeffer~\cite{CV81,CS04}, which encode them with (different classes of) labelled plane trees, thus placing them within the framework of (generalised) Catalan structures.

The relation with trees has been exploited to obtain both scaling and local limit results which have led to the definition and subsequent investigation of objects such as the Brownian map~\cite{LG11,Mie11} and the UIPQ \cite{Men08,CMMinfini}, providing very rich insights into the geometric structure of uniform random large quadrangulations. In fact, it has been shown that a number of classes of uniform random planar maps converge to the Brownian map, whose universality makes quadrangulations of the sphere a very useful model for a random surface. Quadrangulations and, in general, planar maps are also very much studied in physics, in the context of quantum gravity, where the edge flip Markov chain is extensively applied in simulations.

Our contribution within this paper will consist in estimating the mixing time of the edge flip Markov chain $\F_n$ on the set of rooted quadrangulation of the sphere with $n$ faces, as described above and much more thoroughly in Section~\ref{section: flips}. In particular, we shall  prove the following.

\begin{theorem}\label{main theorem}
Let $\nu_n$ be the spectral gap of the edge flip Markov chain $\F_n$ on the set $\sQ_n$ of rooted quadrangulations with $n$ faces. There are positive constants $C_1,C_2$ independent of $n$ such that 
$$C_1n^{-\frac{11}{2}}\leq \nu_n\leq C_2n^{-\frac54}.$$
Consequently, the mixing time for $\F_n$ is $O(n^{13/2})$.
\end{theorem}

The upper bound for the spectral gap is the same as Budzinski's lower bound for the mixing time of flips on triangulations (which indeed it implies for our case $p=4$); the strategy by which we obtain it is quite general and would apply in a much broader context (cf.~Remark~\ref{upper bound generality}).

As for the lower bound, we obtain it through a comparison (achieved with techniques developed by Diaconis and Saloff-Coste \cite{diaconis96}) to a chain on labelled trees which arises via the aforementioned Schaeffer bijection. This chain, which has very natural interpretations on a number of Catalan structures, is a coloured generalisation of a chain on plane trees with $n$ edges whose steps are as follows. Given a tree, pick an edge uniformly at random, and if it is a leaf, then choose one among the following three options with equal probabilities: leave the leaf intact, slide it one step to the left or slide it one step to the right (see Figure~\ref{fig:translations}). If the chosen edge is not a leaf, then do nothing. This chain is also natural in the context of Dyck paths: it is essentially equivalent to picking a vertex of the path uniformly at random and, if the vertex is a peak, translating the peak one position to the right or to the left, with equal probabilities (Figure~\ref{fig:Dyck paths}).

Though apparently not yet analysed within the scope of existing mathematical research about chains on Catalan structures, this `leaf translation' chain is mentioned in the physics literature under the name of Fredkin spin chain, and a first lower bound of order $n^{-\frac{11}{2}}$ for its spectral gap is given by Movassagh~\cite{Mov16}, based on work by himself and Shor \cite{MS16}. We shall partially follow their argument, which is based on the method of building canonical paths to estimate the conductance, to produce an improved lower bound of order $n^{-\frac{9}{2}}$ (see Theorem~\ref{leaf translation gap}), which will be instrumental to obtain our result for flips on quadrangulations.

The paper is organised as follows. Sections~\ref{section: flips} and \ref{section: Schaeffer} will provide the reader with all relevant definitions and recall some details of the Schaeffer bijection, since they will be relevant to our subsequent constructions. In section~\ref{section: upper bound} we give an upper bound of order $n^{-\frac54}$ for the spectral gap of $\F_n$ by considering the Dirichlet form evaluated at the function that gives the radius of a quadrangulation. Section~\ref{section: Markov chain on trees} will acquaint the reader with the leaf translation Markov chain on plane trees (and a ``leaf replanting'' variant) and prove our lower bound for its spectral gap. Finally, a large portion of the paper -- namely, Section~\ref{section: lower bound} -- will be devoted to showing our lower bound for the spectral gap of $\F_n$ via a comparison with a chain on \emph{pointed} rooted quadrangulations, which bridges the gap between $\F_n$ and the leaf translation Markov chain on labelled plane trees.

\section{Edge flips on quadrangulations}\label{section: flips}
\begin{figure}\centering
\begin{tikzpicture}[scale=.9,real/.style={draw=white, fill=black, very thick, inner sep=2pt, circle}, bijection/.style={red, very thick}, new/.style={thick}, b/.style={line width=5pt, white}, none/.style={}]
	\path[use as bounding box] (-4.5,-3) rectangle (3,3);
	\begin{pgfonlayer}{nodelayer}
		\node [style=vertex] (0) at (-4.25, 1.75) {};
		\node [style=vertex, fill=red] (1) at (-2.5, -3) {};
		\node [style=vertex] (2) at (2.5, -2.25) {};
		\node [style=vertex] (3) at (1.5, 3) {};
		\node [style=vertex] (4) at (-1.75, -1.5) {};
		\node [style=vertex] (5) at (-2.5, 0.5) {};
		\node [style=vertex] (6) at (-1.5, -0.25) {};
		\node [style=vertex] (7) at (0.75, 1.5) {};
		\node [style=vertex] (8) at (0.25, -0.25) {};
		\node [style=vertex] (9) at (0, -2.5) {};
	\end{pgfonlayer}
	\begin{pgfonlayer}{edgelayer}
	\begin{scope}
	\clip (1.center) to (9.center) to (2.center) [bend left=15] to (1.center);
	\fill[fill=red!10, draw=red] (1) circle (40pt);
	\end{scope}
	\draw[ultra thick, red] (-0.95,-2.8) arc (10:-10:30pt);
	
		\draw [style=b, bend left] (1) to (2);
		\draw [style=b] (2) to (3);
		\draw [style=b, in=12, out=-168] (3) to (0);
		\draw [style=b, bend left=75, looseness=1.75] (5) to (4);
		\draw [style=b] (4) to (6);
		\draw [style=b] (4) to (1);
		\draw [style=b] (0) to (7);
		\draw [style=b] (7) to (2);
		\draw [style=b, bend right=45, looseness=1.50] (4) to (7);
		\draw [style=b] (8) to (7);
		\draw [style=b, in=-150, out=30] (4) to (7);
		\draw [style=b] (1) to (9);
		\draw [style=b] (5) to (4);
		\draw [style=b] (0) to (5);
		\draw [style=b] (0) to (1);
		\draw [style=b, bend right=15] (1) to (2);
		
		\draw [style=new, bend left] (1) to (2);
		\draw [style=new] (2) to (3);
		\draw [style=new, in=12, out=-168] (3) to (0);
		\draw [style=new, bend left=75, looseness=1.75] (5) to (4);
		\draw [style=new] (4) to (6);
		\draw [style=new] (4) to (1);
		\draw [style=new] (0) to (7);
		\draw [style=new] (0) to (5);
		\draw [style=new] (7) to (2);
		\draw [style=new, bend right=45, looseness=1.50] (4) to (7);
		\draw [style=new] (8) to (7);
		\draw [style=new, in=-150, out=30] (4) to (7);
		\draw [style=new] (1) to (9);
		\draw [style=new] (5) to (4);
		\draw [style=new] (0) to (1);
			\draw [style=new, bend right=15, ->, ultra thick] (1) to (2);
	\end{pgfonlayer}
\end{tikzpicture}\hspace{1cm}
\begin{tikzpicture}[scale=.7]
	\begin{pgfonlayer}{nodelayer}
		\node [style=vertex, label=above right:\contour{white}{$c_1$}] (0) at (-1, 2) {};
		\node [style=vertex, label=above left:\contour{white}{$c_4$}] (1) at (1, 2) {};
		\node [style=vertex, label=below left:\contour{white}{$c_3$}] (2) at (1, 4) {};
		\node [style=vertex, label=below right:\contour{white}{$c_2$}] (3) at (-1, 4) {};
		\node [style=vertex] (4) at (0, -3) {};
		\node[yshift=16pt, xshift=-8pt] (c4) at (4) {\contour{white}{$c_4$}};
		\node[yshift=16pt, xshift=8pt] (c4) at (4) {\contour{white}{$c_2$}};
		\node [style=vertex] (5) at (0, 0.75) {};
		\node [style=vertex] (6) at (0, -1.25) {};
		\node[yshift=-15pt] (c1) at (5) {\contour{white}{$c_1$}};
		\node[yshift=15pt] (c3) at (6) {\contour{white}{$c_3$}};
		\node[] (7) at (-0.75, 2.5) {};
		\node[] (8) at (-0.75, 3.5) {};
		\node[] (9) at (-0.5, 3.75) {};
		\node[] (10) at (0.5, 3.75) {};
		\node[] (11) at (0.75, 3.5) {};
		\node[] (12) at (0.75, 2.5) {};
		\node[] (13) at (0.5, 2.25) {};
		\node[] (14) at (-0.5, 2.25) {};
		\node[] (15) at (0.25, 0.25) {};
		\node[] (16) at (0.25, -2.5) {};
		\node[] (17) at (0.25, -1.5) {};
		\node[] (18) at (-0.25, -1.5) {};
		\node[] (19) at (-0.25, -2.5) {};
		\node[] (20) at (0.5, -2.25) {};
		\node (21) at (-0.5, -2.25) {};
		\node[] (22) at (-0.25, 0.25) {};
	\end{pgfonlayer}
	\begin{pgfonlayer}{edgelayer}
	\begin{scope}
	\clip (0.center) to (1.center) to (2.center) to (3.center) to (0.center);
	\draw[fill=red!20, draw=red] (0) circle (15pt);
	\draw[fill=red!20, draw=red] (1) circle (15pt);
	\draw[fill=red!20, draw=red] (2) circle (15pt);
	\draw[fill=red!20, draw=red] (3) circle (15pt);
	\end{scope}
	\begin{scope}
	\clip [bend left=60, looseness=1.25] (4.center) to (5.center) to (4.center);
	\draw[fill=red!20, draw=red] (4) circle (15pt);
	\draw[fill=red!20, draw=red] (5) circle (13pt);
	\draw[fill=red!20, draw=red] (6) circle (13pt);
	\end{scope}
		\draw[b] (0) to (1);
		\draw[b] (1) to (2);
		\draw[b] (2) to (3);
		\draw[b] (3) to (0);
		\draw[thick] (0) to (1) to (2) to (3) to (0);
		\draw [b, bend left=60, looseness=1.25] (4) to (5);
		\draw [b, bend right=60, looseness=1.25] (4) to (5);
		\draw [b] (4) to (6);
		\draw [thick, bend left=60, looseness=1.25] (4) to (5);
		\draw [thick, bend right=60, looseness=1.25] (4) to (5);
		\draw [ultra thick] (4) to (6);
	\end{pgfonlayer}
\end{tikzpicture}
\caption{\label{fig:quadrangulation}On the left, a quadrangulation $q$ in $\sQ_8$; notice that we may choose to embed it in the plane (rather than the sphere) in a canonical way by having the external face be the one lying directly to the right of the root edge. The origin of $q$ is marked in red. To the right, the two kinds of faces in a quadrangulation -- non-degenerate and degenerate -- with marked corners in clockwise order. The double edge in the degenerate face is the one adjacent to $c_3$, which is drawn with a thicker line.}
\end{figure}
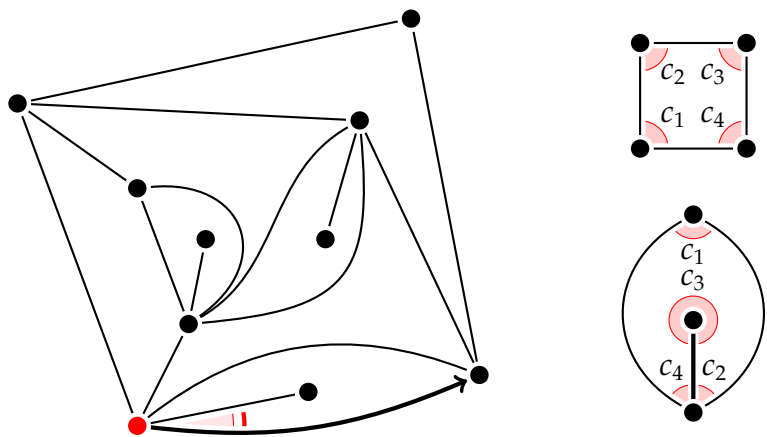

Throughout this paper we shall be dealing with certain Markov chains whose configuration space is the set of quadrangulations with a fixed number of faces; in order to introduce them, let us first discuss some notation.

First and foremost, we shall be adopting some of the language of \emph{planar maps}, with which we assume some familiarity: we will be referring to maps and their \emph{vertices}, \emph{edges}, \emph{faces}, as well as \emph{corners} and \emph{face contours}; we shall not review any definitions but the basic one, that is:

\begin{df}\label{def:planar map}A \emph{planar map} is a connected, locally finite planar multigraph endowed with a cellular embedding in the sphere $S^2$, considered up to orientation-preserving homeomorphisms of the sphere itself. We will call a \emph{rooted planar map of size $n$} a planar map with $n$ faces, endowed with one distinguished oriented edge.\end{df}

One can now define a \emph{quadrangulation of size $n$}, or \emph{of area $n$}, as a rooted planar map of size $n$ all of whose faces have four corners (see Figure~\ref{fig:quadrangulation}); we shall denote the set of all quadrangulations of size $n$ by $\sQ_n$. It follows from Euler's polyhedral formula that a quadrangulation $q\in\sQ_n$ has $2n$ edges and $n+2$ vertices; it is also worth noting that a quadrangulation is automatically bipartite, which implies that it has no loops (that is, it has no edges with only one endpoint). It may, however, have multiple edges between the same two endpoints, and edges which are adjacent to a single face; we call the latter \emph{double edges} of the face they belong to, and the face which contains them a \emph{degenerate face}. We shall often refer to the vertex that the root edge is issued from as \emph{the origin} of the quadrangulation.

Given a quadrangulation $q\in\sQ_n$ and an edge $e$ of $q$, we will denote by $q^{e,+}$ (resp.~$q^{e,-}$), \emph{the quadrangulation obtained from $q$ by flipping edge $e$ clockwise} (resp.~\emph{counterclockwise}); more formally, we mean the quadrangulation given by the following procedure:
\begin{itemize}
\item if $e$ is adjacent to two distinct faces of $q$, erase $e$ from $q$ (thus obtaining a new face with exactly 6 corners) and replace it with the edge obtained by rotating $e$ clockwise (resp. counterclockwise) by one corner (see Figure~\ref{fig:flips}).
\item if $e$ is a double edge within a degenerate face, let $v$ be the vertex of that face that is \emph{not} an endpoint of $e$ and let $w$ be the endpoint of $e$ having degree 1; erase $e$ and replace it with an edge within the same face having endpoints $v,w$. If $e$ is the root edge of $q$, let the newly drawn edge be the root of $q^{e,+}$ (resp.~$q^{e,-}$), oriented in the same way as before (with respect to $w$).\end{itemize}

\begin{figure}[t]\centering
\begin{minipage}{7cm}\begin{tikzpicture}
	\begin{pgfonlayer}{nodelayer}
		\node [style=vertex] (0) at (-4, 3) {};
		\node [style=vertex] (1) at (-4, 4) {};
		\node [style=vertex] (2) at (-3, 4) {};
		\node (x) at (-3,3.5) {\contour{white}{$e$}};
		\node [style=vertex] (3) at (-2, 4) {};
		\node [style=vertex] (4) at (-2, 3) {};
		\node [style=vertex] (5) at (-3, 3) {};
		\node [style=vertex] (6) at (-0.5, 4) {};
		\node [style=vertex] (7) at (0.5, 5) {};
		\node [style=vertex] (8) at (1.5, 4) {};
		\node [style=vertex] (9) at (0.5, 4) {};
		\node [style=vertex] (10) at (-0.5, 5) {};
		\node [style=vertex] (11) at (1.5, 5) {};
		\node [style=vertex] (12) at (0.5, 3) {};
		\node [style=vertex] (13) at (1.5, 2) {};
		\node [style=vertex] (14) at (0.5, 2) {};
		\node [style=vertex] (15) at (-0.5, 3) {};
		\node [style=vertex] (16) at (-0.5, 2) {};
		\node [style=vertex] (17) at (1.5, 3) {};
		\node [] (30) at (-1.5, 3.5) {$q$};
		\node [] (33) at (2.25, 4.5) {$q^{e,+}$};
		\node [] (34) at (2.25, 2.5) {$q^{e,-}$};
	\end{pgfonlayer}
	\begin{pgfonlayer}{edgelayer}
	\draw [thick, ->, bend left, black!30] (-1.4,3.85) to (-0.75,4.5);
	\draw [thick, ->, bend right, black!30] (-1.4,3) to (-0.75,2.5);
		\draw [very thick] (0) to (1);
		\draw [very thick] (1) to (2);
		\draw [very thick] (2) to (3);
		\draw [very thick] (3) to (4);
		\draw [very thick] (4) to (5);
		\draw [very thick] (5) to (0);
		\draw [very thick, red, ->] (5) to (2);
		\draw [very thick] (6) to (10);
		\draw [very thick] (10) to (7);
		\draw [very thick] (7) to (11);
		\draw [very thick] (11) to (8);
		\draw [very thick] (8) to (9);
		\draw [very thick] (9) to (6);
		\draw [blue, very thick, dashed] (9) to (7);
		\draw [very thick, red, ->] (6) to (11);
		\draw [very thick] (16) to (15);
		\draw [very thick] (15) to (12);
		\draw [very thick] (12) to (17);
		\draw [very thick] (17) to (13);
		\draw [very thick] (13) to (14);
		\draw [very thick] (14) to (16);
		\draw [very thick, red, ->] (13) to (15);
		\draw [blue, very thick, dashed] (14) to (12);
	\end{pgfonlayer}
\end{tikzpicture}\end{minipage}\quad
\begin{minipage}{6cm}
\begin{tikzpicture}
	\begin{pgfonlayer}{nodelayer}
		\node [style=vertex] (0) at (-1.5, -1) {};
		\node [style=vertex] (1) at (-1.5, -0.25) {};
		\node [style=vertex, label=right:$v$] (2) at (-1.5, 0.5) {};
		\node [style=vertex] (3) at (0.75, -1) {};
		\node [style=vertex, label=right:$v$] (4) at (0.75, 0.5) {};
		\node [style=vertex] (5) at (0.75, -0.25) {};
		\node [] (w) at (1.1, -0.25) {$w$};
		\node [] (w) at (-1.15, -0.25) {$w$};
		\node [] (6) at (-2.5, -0.25) {$q$};
		\node [] (7) at (2.5, -0.25) {$q^{e,+}=q^{e,-}$};
		\node (x) at (-1.5,-0.65) {\contour{white}{$e$}};
	\end{pgfonlayer}
	\begin{pgfonlayer}{edgelayer}
	\draw [thick, ->, bend left, black!30] (-0.75,-0.25) to (0,-0.25);
		\draw [very thick, bend left=60, looseness=1.25] (0) to (2);
		\draw [very thick, bend right=60, looseness=1.25] (0) to (2);
		\draw [very thick, red, ->] (0) to (1);
		\draw [very thick, bend left=60, looseness=1.25] (3) to (4);
		\draw [very thick, bend right=60, looseness=1.25] (3) to (4);
		\draw [very thick, red, <-] (5) to (4);
		\draw [very thick, blue, dashed] (5) to (3);
	\end{pgfonlayer}
\end{tikzpicture}\\[10pt]
\begin{tikzpicture}
	\begin{pgfonlayer}{nodelayer}
		\node [style=vertex] (0) at (-1.5, -1) {};
		\node [style=vertex] (1) at (-1.5, -0.25) {};
		\node [style=vertex, label=right:$v$] (2) at (-1.5, 0.5) {};
		\node [style=vertex] (3) at (0.75, -1) {};
		\node [style=vertex, label=right:$v$] (4) at (0.75, 0.5) {};
		\node [] (w) at (1.1, -0.25) {$w$};
		\node [] (w) at (-1.15, -0.25) {$w$};
		\node [style=vertex] (5) at (0.75, -0.25) {};
		\node [] (6) at (-2.5, -0.25) {$q$};
		\node [] (7) at (2.5, -0.25) {$q^{e,+}=q^{e,-}$};
		\node (x) at (-1.5,-0.55) {\contour{white}{$e$}};
	\end{pgfonlayer}
	\begin{pgfonlayer}{edgelayer}
	\draw [thick, ->, bend left, black!30] (-0.75,-0.25) to (0,-0.25);
		\draw [very thick, bend left=60, looseness=1.25] (0) to (2);
		\draw [very thick, bend right=60, looseness=1.25] (0) to (2);
		\draw [very thick, red, <-] (0) to (1);
		\draw [very thick, bend left=60, looseness=1.25] (3) to (4);
		\draw [very thick, bend right=60, looseness=1.25] (3) to (4);
		\draw [very thick, red, ->] (5) to (4);
		\draw [very thick, blue, dashed] (5) to (3);
	\end{pgfonlayer}
\end{tikzpicture}\end{minipage}
\caption{\label{fig:flips}Clockwise and counterclockwise flips for a simple and a double edge in a quadrangulation; if the edge is the root, its orientation is ``preserved'' as in the figure.}\end{figure}
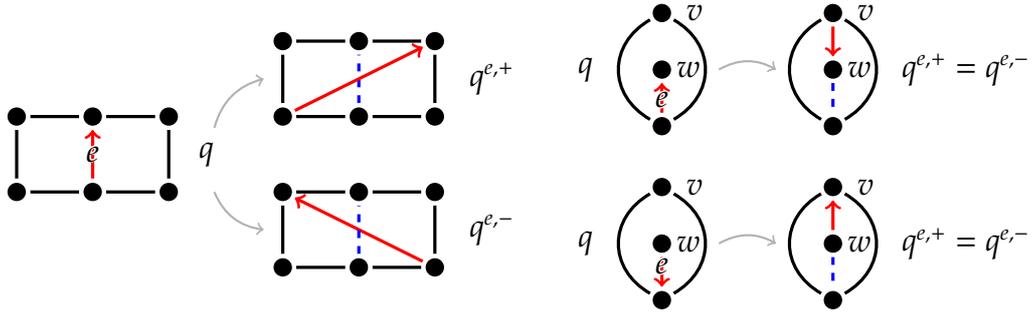

Throughout this paper, we will consider the Markov chain $\F^n$ of quadrangulation edge flips on the state space $\sQ_n$, whose transition probabilities are of the form $$p(q,q')=\frac{1}{6n}\sum_{e\in E(q)}\left(1_{q'=q^{e,+}}+1_{q'=q^{e,-}}+1_{q=q'}\right);$$
in other words, conditionally on $\F^n_k=q$, $\F^n_{k+1}$ can be determined by choosing an edge $e$ of $q$ uniformly at random and setting either $\F^n_{k+1}=q$, or $\F^n_{k+1}=q^{e,+}$, or $\F^n_{k+1}=q^{e,-}$, with equal probabilities.

Notice that, given a pair $(q,q')$ of distinct quadrangulations in $\sQ_n$, there are at most four distinct pairs $(e,s)$ in $E(q)\times \{+,-\}$ such that $q'=q^{e,s}$. In fact, assuming $e$ is not the root edge of $q$ then $e$ is uniquely determined by the pair $(q,q')$, and either $s$ is determined as well or, in the case where $e$ is a double edge, we have $q'=q^{e,+}=q^{e,-}$. In addition, if $e$ is not the root edge, it is possible that flipping the root edge might transform $q$ into $q'$; in other words, that $q'=q^{\eta,+}$ or $q'=q^{\eta,-}$, where $\eta\neq e$ is the root edge of $q$. Consequently, we have
$$\frac{1}{3 |E(q)|}=\frac{1}{6n}\leq p(q,q')\leq \frac{2}{3n}$$
for all $q,q'\in\sQ_n$ such that $p(q,q')\neq 0$ and $q\neq q'$.

Notice that, given $q\in\sQ_n$, $e\in E(q)$, $s\in\{+,-\}$, we can naturally identify vertices of $q$ with vertices of $q^{e,s}$, and edges of $q$ with edges of $q^{e,s}$ (where the edge $e$ corresponds to the edge redrawn by the flip procedure in $q^{e,s}$); we will therefore often refer to vertices or edges using the same notation in $q$ and $q^{e,s}$, when we wish to implicitly exploit such a correspondence. This, of course, will need to be done with some care, since the correspondence is not necessarily unique when the quadrangulations $q$ and $q^{e,s}$ are given, but $e$ and $s$ are not known.

\begin{remark}\label{reversibility}
The Markov chain $\F^n$ is reversible and aperiodic: indeed, we have $q'=q^{e,+}$ if and only if $q=(q')^{e,-}$, so $p(q,q')=p(q',q)$; furthermore, we have trivially that $p(q,q)\geq \frac{1}{3}$.	
\end{remark}

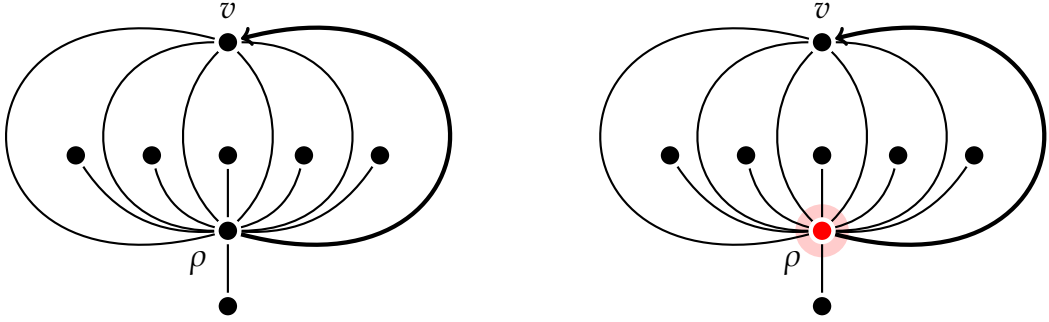
\begin{figure}\centering
\begin{tikzpicture}
	\begin{pgfonlayer}{nodelayer}
		\node [style=vertex, label=below left:$\rho$] (0) at (0, -1) {};
		\node [style=vertex, label=above:$v$] (1) at (0, 1.5) {};
		\node [style=vertex] (2) at (0, -0) {};
		\node [style=vertex] (3) at (1, -0) {};
		\node [style=vertex] (4) at (2, -0) {};
		\node [style=vertex] (5) at (0, -2) {};
		\node [style=vertex] (6) at (-1, -0) {};
		\node [style=vertex] (7) at (-2, -0) {};
	\end{pgfonlayer}
	\begin{pgfonlayer}{edgelayer}
		\draw [ultra thick, bend right=105, looseness=3.75, ->] (0) to (1);
		\draw [thick, bend right=105, looseness=3.75] (1) to (0);
		\draw [thick,bend left=90, looseness=2.00] (0) to (1);
		\draw [thick,bend right=90, looseness=2.00] (0) to (1);
		\draw [thick,bend right=45, looseness=1.00] (0) to (1);
		\draw [thick,bend left=45, looseness=1.00] (0) to (1);
		\draw [thick](2) to (0);
		\draw [thick,bend left, looseness=1.00] (3) to (0);
		\draw [thick,bend left, looseness=1.00] (4) to (0);
		\draw [thick,bend right, looseness=1.00] (6) to (0);
		\draw [thick,bend right, looseness=1.00] (7) to (0);
		\draw [thick] (0) to (5);
	\end{pgfonlayer}
\end{tikzpicture}
\begin{tikzpicture}
	\begin{pgfonlayer}{nodelayer}
		\node [style=vertex, label=below left:$\rho$, fill=red] (0) at (0, -1) {};
		\node [style=vertex, label=above:$v$] (1) at (0, 1.5) {};
		\node [style=vertex] (2) at (0, -0) {};
		\node [style=vertex] (3) at (1, -0) {};
		\node [style=vertex] (4) at (2, -0) {};
		\node [style=vertex] (5) at (0, -2) {};
		\node [style=vertex] (6) at (-1, -0) {};
		\node [style=vertex] (7) at (-2, -0) {};
	\end{pgfonlayer}
	\begin{pgfonlayer}{edgelayer}
		\fill[red!20] (0) circle (10pt);
		\draw [ultra thick, bend right=105, looseness=3.75, ->] (0) to (1);
		\draw [thick, bend right=105, looseness=3.75] (1) to (0);
		\draw [thick,bend left=90, looseness=2.00] (0) to (1);
		\draw [thick,bend right=90, looseness=2.00] (0) to (1);
		\draw [thick,bend right=45, looseness=1.00] (0) to (1);
		\draw [thick,bend left=45, looseness=1.00] (0) to (1);
		\draw [thick](2) to (0);
		\draw [thick,bend left, looseness=1.00] (3) to (0);
		\draw [thick,bend left, looseness=1.00] (4) to (0);
		\draw [thick,bend right, looseness=1.00] (6) to (0);
		\draw [thick,bend right, looseness=1.00] (7) to (0);
		\draw [thick] (0) to (5);
	\end{pgfonlayer}
\end{tikzpicture}
\caption{\label{fig:q_0}The quadrangulation $q_0$ in $\sQ_6$, with 6 degenerate faces arranged so that the degree of the origin is 12; to the right, the pointed version from Lemma~\ref{pointed irreducibility}.}
\end{figure}

\begin{lemma}\label{irreducibility}
The Markov chain $\F^n$ is irreducible.
\end{lemma}

\begin{proof}
Let $q_0$ be the quadrangulation with $n$ degenerate faces and such that the origin has the maximum possible degree (that is $2n$, all edges being incident to it) -- see Figure~\ref{fig:q_0}. We show that, given any quadrangulation $q\in\sQ_n$, one can obtain $q_0$ from $q$ with a sequence of edge flips. 

Indeed, given any quadrangulation $q$, unless the degree of the origin $\rho$ is $2n$, one can increase it via an edge flip. Suppose not all edges have $\rho$ as an endpoint and let $v$ be a neighbour of $\rho$ that has at least one neighbour different from $\rho$; then, if you consider edges issued from $v$ in clockwise order around $v$, there must be an edge $e$ with second endpoint $w\neq \rho$, followed by one with endpoints $v$ and $\rho$. Remark that $q^{e,-}$ has an origin with degree increased by one with respect to $q$.

We may therefore suppose that $q$ is a quadrangulation in $\sQ_n$ whose origin $\rho$ has degree $2n$. Let $v$ be the second endpoint of the root edge in $q$; we will show that, unless $\deg v=n$, there is an edge flip of $q$ increasing the degree of $v$ and not decreasing the degree of $\rho$. Indeed, notice that flipping any edge which is not a double edge inside a degenerate face does not change the degree of $\rho$; this is because every quadrangulation is bipartite, and in particular the bipartition of $q$'s vertices has one class consisting of $\rho$ only, and one consisting of $V(q)\setminus\{\rho\}$. The bipartition can be changed only by flips of degenerate edges, so any other flip will transform an edge having $\rho$ as an endpoint to another edge having $\rho$ as an endpoint. Consider now all edges adjacent to $v$; if $v$ has strictly less than $n$ adjacent edges, then it must be part of a face that is not degenerate (if it is only adjacent to degenerate faces, then the fact that all edges have $\rho$ as an endpoint implies $q=q_0$). Consider any edge $e$ of such a face not having $v$ as an endpoint: then either $q^{e,-}$ or $q^{e,+}$ has the degree of $v$ increased by one, and the degree of $\rho$ unchanged. 

Now, if $q\in\sQ_n$ has root edge $(\rho,v)$ with $\deg \rho=2n$ and $\deg v=n$, then $q=q_0$, as desired. Then reversibility (cf.~Remark~\ref{reversibility}) implies that $\F^n$ is irreducible.
\end{proof}

As a consequence of Lemma~\ref{irreducibility} and Remark~\ref{reversibility}, $\F^n$ admits the uniform measure on $\sQ_n$ as its (unique) stationary distribution.

We will see later how, rather than the set $\sQ_n$, it will be convenient to consider the set $\sQ_n^\bullet$ of all \emph{pointed} quadrangulations with $n$ faces, that is the set $\{(q,v) \st q\in\sQ_n, v\in V(q)\}$. The Markov chain $\F^n$ can be easily extended to a Markov chain $\F^{\bullet,n}$ with state space $\sQ_n^\bullet$, by redefining the (clockwise and counterclockwise) flips so that the distinguished vertex is preserved, thanks to the natural identification between $V(q^{e,s})$ and $V(q)$. Notice that, if $F:\sQ_n^\bullet\to\sQ_n$ is the forgetful map that rids quadrangulations of the pointing, for a quadrangulation $q_\bullet$ in $\sQ_n^\bullet$ we have $F(q_\bullet^{e,s})=F(q_\bullet)^{e,s}$, where we are treating $e$ both as an edge of $q_\bullet$ and as an edge of $F(q_\bullet)$, since $F$ does induce a natural identification for both vertices and edges.

Reversibility is of course still true, but one has to go a little further to prove irreducibility of $\F^{\bullet,n}$.

\begin{lemma}\label{pointed irreducibility}Let $q_0\in\sQ_n^{\bullet}$ be the quadrangulation with $n$ degenerate faces, rooted in an oriented edge $(\rho,v)$ such that $\rho$ has degree $2n$ and $v$ has degree $n$, pointed in $\rho$ (see Figure~\ref{fig:q_0}). Then any quadrangulation $q\in\sQ_n^\bullet$ can be turned into $q_0$ with a sequence of flips. In particular, $\F^{\bullet,n}$ is irreducible.
\end{lemma}
\begin{proof}
Turning $q$ into $q_0$ can be done with a very similar procedure to Lemma~\ref{irreducibility}. First, if $\delta$ is the one distinguished vertex of $q$, one can apply flips until they obtain a quadrangulation $q'$, similar to $q_0$ but where $\delta$ has degree $2n$. If $\delta$ turns out to be the origin, then we are done. Otherwise, if the root edge $e$ of $q'$ is not a double edge within a degenerate face, all we need to do is reverse its orientation by taking $(((q')^{e,+})^{e,+})^{e,+}$: this will make $\delta$ the origin and preserve its degree $\deg \delta=2n$. However, if $e$ is a double edge of $q'$, one only needs to flip clockwise the edge $e'$ that comes before $e$ in the clockwise contour of the degenerate face containing $e$. Then one can flip $e$ clockwise three times, then $e'$ counterclockwise, to have $\delta$ as the origin and preserve its degree. Note that, at this point, the root edge is a double edge. We can then proceed as in Lemma~\ref{irreducibility} to increase the degree of the second endpoint $v\neq\delta$ of the root edge until it is $n$. Notice that this only entails flipping edges that do not already have $v$ as an endpoint, so the root edge will not be flipped and the final quadrangulation will be correctly rooted in an edge issued from $\delta$.
\end{proof}

As a consequence of the lemma above, the stationary distribution for $\F^{\bullet,n}$ is the uniform measure on $\sQ_n^\bullet$.

Our aim in this paper will be to prove upper and lower bounds for the spectral gap $\nu_n$ of the Markov chain $\F^n$; we will rely on the Markov chain $\F^{\bullet,n}$ for the known bijections available between the set $\sQ_n^{\bullet}$ and certain sets of labelled trees, which we will briefly discuss in the next section. Dealing with $\F^{\bullet,n}$ will still provide information about $\F^n$: any lower bound for $\nu_n^\bullet$ will serve as a lower bound for $\nu_n$, as per the following lemma.

\begin{lemma}\label{pointed comparison}For the spectral gap $\nu_n^\bullet$ of $\F^{\bullet,n}$ and the spectral gap $\nu_n$ of $\F_n$, we have $\nu_n^\bullet\leq \nu_n$.\end{lemma}

\begin{proof}
The proof is quite immediate, since we can write
$$\nu_n=\mathcal{E}_{\F^n}(f,f)=\frac12\sum_{\substack{q\in\sQ_n\\e\in E(q)\\s\in\{+,-\}}}(f(q)-f(q^{e,s}))^2\frac{1}{|\sQ_n|6n}$$
for some function $f:\sQ_n\to\mathbb{R}$ such that $\E_\pi(f)=0$ and $\V_\pi(f)=1$ (where $\pi$ is the uniform measure on $\sQ_n$).

Now, setting $F:\sQ^\bullet_n\to\sQ_n$ to be the forgetful function which rids a quadrangulation of its distinguished vertex, consider the function $f\circ F:\sQ^\bullet_n\to\mathbb{R}$. We have $\E_{\pi^\bullet}(f\circ F)=0$ and $\V_{\pi^{\bullet}}(f\circ F)=1$, where $\pi^\bullet$ is the uniform measure on $\sQ_n^\bullet$. For the spectral gap $\nu_n^\bullet$ of $\F^{n,\bullet}$, we have
$$\nu^\bullet_n\leq\mathcal{E}_{\F^{n,\bullet}}(f\circ F,f\circ F)=\frac12\sum_{\substack{q\in\sQ_n^\bullet\\e\in E(q)\\s\in\{+,-\}}}\left(f(F(q))-f(F(q^{e,s}))\right)^2\cdot\frac{1}{(n+2)|\sQ_n|}\cdot\frac{1}{6n}$$
$$=\frac12\sum_{\substack{q\in\sQ_n^\bullet\\e\in E(F(q))\\s\in\{+,-\}}}\left(f(F(q))-f(F(q)^{e,s})\right)^2\cdot\frac{1}{(n+2)|\sQ_n|}\cdot\frac{1}{6n}=\mathcal{E}_{\F^n}(f,f)=\nu_n,$$
as claimed.\end{proof}

\section{The Schaeffer bijection}\label{section: Schaeffer}

	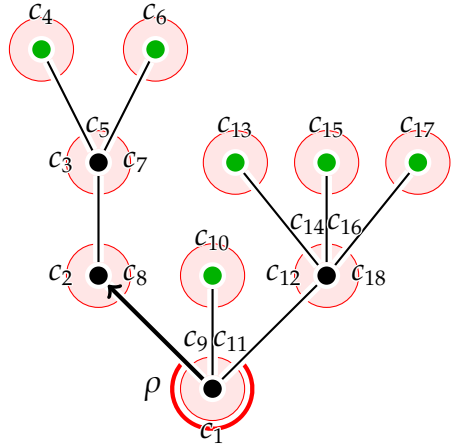
\begin{figure}\centering
	\begin{tikzpicture}[real/.style={circle,fill=black,draw=white, ultra thick, inner sep=3pt}, scale=1.5]
		\node[real, label={[yshift=-30pt]\contour{white}{$c_1$}}, label={[yshift=4.5pt, xshift=-6pt]\contour{white}{$c_9$}}, label={[yshift=4.5pt, xshift=7pt]\contour{white}{$c_{11}$}}] (0) at (0,0) {};
		\node[real, label=left:\contour{white}{$c_2$}, label=right:\contour{white}{$c_8$}] (01) at (-1,1) {};
		\node[real, label=left:\contour{white}{$c_{12}$}, label=right:\contour{white}{$c_{18}$}, label={[yshift=6pt, xshift=-7pt]\contour{white}{$c_{14}$}}, label={[yshift=6pt, xshift=7pt]\contour{white}{$c_{16}$}}] (02) at (1,1) {};
		\node[real, label=above:\contour{white}{$c_{10}$}, fill=green!70!black] (03) at (0,1) {};
		\node[real, label=above:\contour{white}{$c_{13}$}, fill=green!70!black] (021) at (0.2,2) {};
		\node[real, label=above:\contour{white}{$c_{15}$}, fill=green!70!black] (022) at (1,2) {};
		\node[real, label=above:\contour{white}{$c_{17}$}, fill=green!70!black] (023) at (1.8,2) {};
		\node[real, label=left:\contour{white}{$c_3$}, label=above:\contour{white}{$c_5$}, label=right:\contour{white}{$c_7$}] (011) at (-1,2) {};
		\node[real, label=above:\contour{white}{$c_4$}, fill=green!70!black] (0111) at (-1.5,3) {};
		\node[real, label=above:\contour{white}{$c_6$}, fill=green!70!black] (0112) at (-0.5,3) {};
		
		\node[] at (-15pt,0) {$\rho$};
		
		\pgfonlayer{edgelayer}

		\fill[fill=red!10, draw=red] (01) circle (8pt);
		\fill[fill=red!10, draw=red] (02) circle (8pt);
		\fill[fill=red!10, draw=red] (03) circle (8pt);
		\fill[fill=red!10, draw=red] (011) circle (8pt);
		\fill[fill=red!10, draw=red] (0111) circle (8pt);
		\fill[fill=red!10, draw=red] (0112) circle (8pt);
		\fill[fill=red!10, draw=red] (021) circle (8pt);
		\fill[fill=red!10, draw=red] (022) circle (8pt);
		\fill[fill=red!10, draw=red] (023) circle (8pt);
		\fill[fill=red!10, draw=red] (0) circle (8pt);
		\draw[ultra thick, red] (7pt,7pt) arc (45:-230:10pt);
		 
		\draw[white, line width=4pt] (0)--(01) (0)--(02) (0)--(03);
		\draw[white, line width=4pt] (01)--(011);
		\draw[white, line width=4pt] (011)--(0111) (011)--(0112);
		\draw[white, line width=4pt] (02)--(021) (02)--(022) (02)--(023);
		
		\draw[->, ultra thick] (0)--(01);
		\draw[thick] (0)--(02) (0)--(03);
		\draw[thick] (01)--(011);
		\draw[thick] (011)--(0111) (011)--(0112);
		\draw[thick] (02)--(021) (02)--(022) (02)--(023);
		
		\endpgfonlayer
	\end{tikzpicture}
	\caption{\label{fig:plane tree}A \emph{plane tree} with $9$ edges, whose 18 corners are labelled according to their order in the clockwise contour; the tree is rooted in the marked oriented edge, or equivalently has the corner labelled 1 as a distinguished corner. Leaves are marked in green and are defined as vertices other than the origin having degree 1, i.e.~only one corner.}
		\end{figure}
		
In order to obtain lower bounds for the spectral gap of $\F^n$, we will find it convenient to compare it to the spectral gap of a certain Markov chain on the state space of (labelled) trees. A key ingredient to set up this comparison will be a well-known bijection often referred to as the Schaeffer correspondence~\cite{CS04,CV81}.

Although this bijection and its variants have been described in a number of papers, we shall still give a very brief presentation of the construction of labelled trees from rooted, pointed quadrangulations and vice-versa, since part of it will be heavily relied upon in the rest of the paper.

\begin{df}A plane tree is a rooted planar map with a single face.	
\end{df}

We will often find it convenient to see the root of a plane tree as a distinguished corner rather than a distinguished oriented edge; in what follows, we shall refer to the \emph{clockwise contour} of a tree (see Figure~\ref{fig:plane tree}) as the cyclic sequence $(c_i)_{i=1}^{2n}$ of its corners (where $n$ is the number of its edges); we number the corner in such a way that $c_1$ is the root corner, that is the corner of the origin lying immediately to the left of the root edge. Given a vertex $v$ of a plane tree other than the origin, we shall write $p(v)$ for its parent; notice that each edge of a tree may be univocally written in the form $(v,p(v))$, where $v$ is a vertex of the tree other than the origin. Vertices of degree 1, with the exception of the origin, will be called \emph{leaves}. 

We shall call $\T_n$ the set of all plane trees with $n$ edges; trees with zero edges do not conform to the definition above, but we will still find it convenient to define $\T_0$ to be $\{\bullet\}$ by convention, where ``$\bullet$'' is the graph with one vertex and no edges.

\begin{df}
A \emph{labelled tree} is a plane tree $t$ endowed with a labelling $l:V(t)\to \mathbb{Z}$ such that
\begin{itemize}
\item if $\rho$ is the origin of $t$, $l(\rho)=0$;
\item for any vertex $v\in V(t)\setminus\{\rho\}$, $|l(v)-l(p(v))|\in \{1,-1,0\}$.	
\end{itemize}
We shall call $\LT_n$ the set of all labelled trees with $n$ edges, and set $\LT_0=\{\bullet\}$.
\end{df}

Notice that, equivalently, a labelled tree could simply be presented as a plane tree whose edges are three-coloured (the colours being $\{1,-1,0\}$); if $c(e)$ is the colour of the edge $e$, labels of vertices could be recovered by setting $l(v)=\sum_{e\in P(v)}c(e)$, where $P(v)$ is the one simple path leading from $v$ to the origin (or the empty path if $v$ is the origin itself). Throughout the paper, we will use both points of view; it will therefore be useful to introduce a more general notation for plane trees whose edges are $r$-coloured ($r$ being some fixed positive integer); we will write $\TR_n$ for the set
$$\left\{(t,C)\st t\in\T_n, C\in\{1,\ldots,r\}^{E(t)}\right\}.$$

For convenience, we will often refer to a labelled or $r$-coloured tree with a single symbol such as $t$, and consider the labelling or colouring to be implicit; in the case of labelled trees, we will usually call the labelling $l$ without further comment, and sometimes naturally extend it to corners, thus writing $l(c)$ when we mean $l(v)$, where $v$ is the vertex of $t$ that $c$ is adjacent to.

The reason for our definition of $\LT_n$ is the fact that the sets $\sQ_n^\bullet$ and $\LT_n\times\{-1,1\}$ have the same cardinality; moreover, pointed quadrangulations can be interpreted as pairs $(t,\epsilon)$, where $t$ is a labelled tree and $\epsilon\in\{-1,1\}$, in a rather natural way. As promised, we give here a description of how to construct an element of $\sQ_n^\bullet$ from an element of $\LT_n$ and a sign $\epsilon\in\{-1,1\}$ via the (unconstrained) Schaeffer correspondence; we include a brief description of the inverse construction for completeness and clarity, but this will not be explicitly used in the proofs to come.

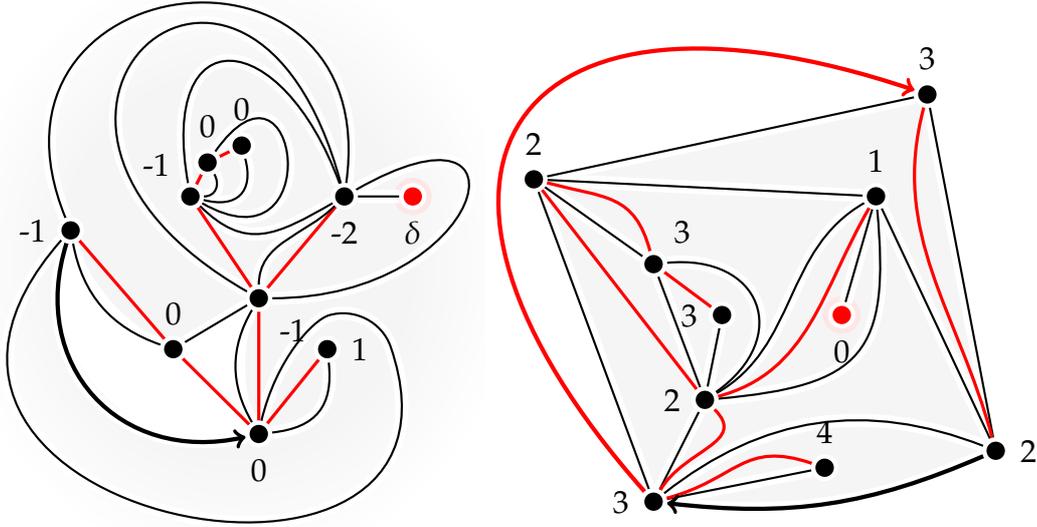
\begin{figure}\centering
\begin{tikzpicture}[scale=.9,real/.style={draw=white, fill=black, very thick, inner sep=2pt, circle}, bijection/.style={red, very thick}, new/.style={thick}, b/.style={line width=5pt, white}, none/.style={}]
	\path[use as bounding box] (-4,-4.5) rectangle (3.3,3.5);
		\begin{pgfonlayer}{nodelayer}
		\node [style=vertex, label=below:0] (0) at (0, -3) {};
		\node [style=vertex, label=0] (1) at (-1.25, -1.75) {};
		\node [style=vertex, label=below right:-1] (2) at (0, -1) {};
		\node [style=vertex, label=right:1] (3) at (1, -1.75) {};
		\node [style=vertex, label=left:-1] (4) at (-2.75, 0) {};
		\node [style=vertex, label=above left:-1] (5) at (-1, 0.5) {};
		\node [style=vertex, label=below:-2] (6) at (1.25, 0.5) {};
		\node [style=vertex, label=0] (7) at (-0.75, 1) {};
		\node [style=vertex, label=0] (8) at (-0.25, 1.25) {};
		\node [style=vertex, fill=red, label=below:$\delta$] (9) at (2.25, 0.5) {};
	\end{pgfonlayer}
	\begin{pgfonlayer}{edgelayer}
	\pgfdeclareradialshading{shaded}{\pgfpoint{0cm}{0cm}}%
{color(0cm)=(gray!8);
color(0.6cm)=(gray!8);
color(1cm)=(white)}
	\shade[shading=shaded, use as bounding box] (-4,-4.5) rectangle (3.3,3.5);
	\fill[white, bend left=60, looseness=1.25] (0.center) to (4.center) [bend right=20, looseness=1] to (1.center) -- (2.center) [bend right] to (0.center);
	\node [circle, inner sep=5pt, fill=red!10] at (2.25, 0.5) {};
	
		\draw [style=b, bend left=60, looseness=1.25] (0) to (4);
		\draw [style=b, bend right] (4) to (1);
		\draw [style=b, bend left=105, looseness=2.50] (4) to (6);
		\draw [style=b] (1) to (2);
		\draw [style=b, bend left] (0) to (2);
		\draw [style=b, in=105, out=150, looseness=6.25] (2) to (6);
		\draw [style=b, in=110, out=100, looseness=2.9] (5) to (6);
		\draw [style=b, in=0, out=-60, looseness=1.25] (7) to (5);
		\draw [style=b, bend left=60, looseness=1.50] (8) to (5);
		\draw [style=b, in=-30, out=55, looseness=7] (7) to (5);
		\draw [style=b, in=-150, out=-45, looseness=1.25] (5) to (6);
		\draw [style=b, in=-135, out=90] (2) to (6);
		\draw [style=b] (6) to (9);
		\draw [style=b, bend left=45, looseness=1.25] (3) to (0);
		\draw[style=b, bend left, rounded corners=18pt] (0) to (1.7,-1.05) [bend right=-120, looseness=3] to (4);
			\draw [style=b, in=30, out=0, looseness=4.00] (2) to (6);
			
		\draw [style=bijection] (0) to (1);
		\draw [style=bijection] (0) to (2);
		\draw [style=bijection] (0) to (3);
		\draw [style=bijection] (1) to (4);
		\draw [style=bijection] (2) to (5);
		\draw [style=bijection] (2) to (6);
		\draw [style=bijection] (5) to (7);
		\draw [style=bijection] (7) to (8);
		
		\draw [style=new, bend left=60, looseness=1.25, ultra thick, <-] (0) to (4);
		\draw [style=new, bend right] (4) to (1);
		\draw [style=new, bend left=105, looseness=2.50] (4) to (6);
		\draw [style=new] (1) to (2);
		\draw [style=new, bend left] (0) to (2);
		\draw [style=new, in=105, out=150, looseness=6.25] (2) to (6);
		\draw [style=new, in=110, out=100, looseness=2.9] (5) to (6);
		\draw [style=new, in=0, out=-60, looseness=1.25] (7) to (5);
		\draw [style=new, bend left=60, looseness=1.50] (8) to (5);
		\draw [style=new, in=-30, out=55, looseness=7] (7) to (5);
		\draw [style=new, in=-150, out=-45, looseness=1.25] (5) to (6);
		\draw [style=new, in=-135, out=90] (2) to (6);
		\draw [style=new] (6) to (9);
		\draw [style=new, bend left=45, looseness=1.25] (3) to (0);
		\draw[style=new, bend left, rounded corners=18pt] (0) to (1.7,-1.05) [bend right=-120, looseness=3] to (4);
			\draw [style=new, in=30, out=0, looseness=4.00] (2) to (6);
	\end{pgfonlayer}\end{tikzpicture}\quad
	\begin{tikzpicture}[scale=.9,real/.style={draw=white, fill=black, very thick, inner sep=2pt, circle}, bijection/.style={red, very thick}, new/.style={thick}, b/.style={line width=5pt, white}, none/.style={}]
\path[use as bounding box] (-4.5,-3.5) rectangle (3,3);
	\begin{pgfonlayer}{nodelayer}
		\node [style=vertex, label=2] (0) at (-4.25, 1.75) {};
		\node [style=vertex, label=left:3] (1) at (-2.5, -3) {};
		\node [style=vertex, label=right:2] (2) at (2.5, -2.25) {};
		\node [style=vertex, label=3] (3) at (1.5, 3) {};
		\node [style=vertex, label=left:2] (4) at (-1.75, -1.5) {};
		\node [style=vertex, label=above right:3] (5) at (-2.5, 0.5) {};
		\node [style=vertex, label=left:3] (6) at (-1.5, -0.25) {};
		\node [style=vertex, label=1] (7) at (0.75, 1.5) {};
		\node [style=vertex, label=below:0, fill=red] (8) at (0.25, -0.25) {};
		\node [style=vertex, label=4] (9) at (0, -2.5) {};
	\end{pgfonlayer}
	\begin{pgfonlayer}{edgelayer}
		\fill[gray!8] (0.center)--(1.center) [bend right=15] to (2.center)--(3.center)--(0.center);
				\node [circle, inner sep=5pt, fill=red!10] at (0.25, -0.25) {};
		\draw [style=b, bend left] (1) to (2);
		\draw [style=b] (2) to (3);
		\draw [style=b] (3) to (0);
		\draw [style=b, bend left=75, looseness=1.75] (5) to (4);
		\draw [style=b] (4) to (6);
		\draw [style=b] (4) to (1);
		\draw [style=b] (0) to (7);
		\draw [style=b] (7) to (2);
		\draw [style=b, bend right=45, looseness=1.50] (4) to (7);
		\draw [style=b] (8) to (7);
		\draw [style=b, in=-150, out=30] (4) to (7);
		\draw [style=b] (1) to (9);
		\draw [style=b, bend right=15] (1) to (2);
		\draw [style=b] (5) to (4);
		\draw [style=b] (0) to (5);
		\draw [style=b] (0) to (1);
		\draw [style=new, bend left] (1) to (2);
		\draw [style=new] (2) to (3);
		\draw [style=new] (3) to (0);
		\draw [style=new, bend left=75, looseness=1.75] (5) to (4);
		\draw [style=new] (4) to (6);
		\draw [style=new] (4) to (1);
		\draw [style=new] (0) to (7);
		\draw [style=new] (7) to (2);
		\draw [style=new, bend right=45, looseness=1.50] (4) to (7);
		\draw [style=new] (8) to (7);
		\draw [style=new, in=-150, out=30] (4) to (7);
		\draw [style=new] (1) to (9);
		\draw [style=new] (5) to (4);
		\draw [style=new] (0) to (5);
		\draw [style=new] (0) to (1);
		
		\draw [style=bijection, in=15, out=-120] (7) to (4);
		\draw [style=bijection, in=60, out=-45, looseness=1.25] (4) to (1);
		\draw [style=bijection] (0) to (4);
		\draw [style=bijection] (5) to (6);
		\draw [style=bijection, in=105, out=-15, looseness=1.25] (0) to (5);
		\draw [style=bijection, in=105, out=-105] (3) to (2);
		\draw [style=bijection, in=15, out=161, looseness=1.25] (9) to (1);
		\draw [style=bijection, bend left=75, looseness=2.25, ultra thick, ->] (1) to (3);
				\draw [style=new, bend right=15, <-, ultra thick] (1) to (2);
	\end{pgfonlayer}
\end{tikzpicture}
	\caption{\label{Schaeffer+-}The (unconstrained) Schaeffer bijection. On the left, the map from the labelled tree (in red) to the pointed quadrangulation (in black); the distinguished vertex is marked in red, and the numbers represent the labels on the tree. On the right, the map from the pointed quadrangulation (in black) to the labelled tree (in red); the distinguished vertex is marked in red, and numbers represent distances to the distinguished vertex in the quadrangulation. The two quadrangulations above are the same, although the unbounded face in the embedding on the right corresponds to the white inner face on the left.}
	\end{figure}
\bigskip
\noindent{\bf Construction of a mapping $\phi$ from $\LT_n$ to $\sQ_n^\bullet$ ($\phi:(\tau,\epsilon)\mapsto q$)}
\begin{itemize}\item consider the clockwise (cyclic) contour $(c_i)_{i=1}^{2n}$ of $\tau$, started at the distinguished corner, and let $\ell$ be the minimal label appearing on vertices of $\tau$;
\item for each corner $c_i$ labelled at least $\ell+1$, set $k=\min\{j>0\st l(c_{i+j})=l(c_i)-1\}$; join $c_i$ to $c_{i+k}$ with an edge (so that edges being drawn do not cross, see Figure~\ref{Schaeffer+-});
\item draw a new vertex $\delta$ within the unbounded face of the tree and join each corner labelled $\ell$ to $\delta$ with a new edge (again, so as not to cross any previously drawn edges);
\item root the map thus obtained in the newly drawn edge issued from the distinguished corner of $\tau$, oriented away from the origin of $\tau$ if $\epsilon=-1$, towards it if $\epsilon=1$; make $\delta$ the distinguished vertex;
\item erase all edges of $\tau$ and forget all labels.
\end{itemize}
\bigskip
\noindent{\bf Construction of a mapping $\phi^{-1}$ from $\sQ_n^\bullet$ to $\LT_n$ ($\phi^{-1}:q\mapsto (\tau,\epsilon)$)}
\begin{itemize}
\item label all vertices in $q$ with their graph distance to the distinguished vertex $\delta$, thus defining a labelling $l:V(q)\to\mathbb{N}$; for each face of $q$, read the labels of the vertices adjacent to its four corners cyclically according to a  \emph{clockwise} contour. Given two successive corners $c_i$ and $c_{i+1}$ in a clockwise contour of a face $f$, we say $c_i$ is a down-step corner of $f$ if the label of $c_{i+1}$ is strictly smaller than that of $c_i$ (notice that, since the map is bipartite, the label of $c_{i+1}$ is either one more or one less than that of $c_i$, hence each face has exactly 2 down-step corners);
\item draw a new edge within each face of $q$, joining its two down-step corners;
\item consider the root edge $(e_{-},e_+)$ of $q$, and let $f_l$ and $f_r$ be the faces lying \emph{left} and \emph{right} of $(e_{-},e_+)$ respectively (of course, the two may coincide); if $l(e_-)<l(e_+)$, set $\epsilon=1$ and choose as new root the edge being drawn between a corner adjacent to $e_+$ and the other down-step corner of $f_l$, oriented away from $e_+$; if $l(e_-)>l(e_+)$, set $\epsilon=-1$ and root in the edge drawn between a corner of $e_-$ and a down-step corner of $f_r$, oriented away from $e_-$; 
\item subtract $l(e_-)$ (if $\epsilon=-1$) or $l(e_+)$ (if $\epsilon=1$) from all labels: this way the label of the new origin is~$0$;
\item erase all original edges and the distinguished vertex $\delta$.\end{itemize}

\begin{theorem}[Schaeffer correspondence]\label{Schaeffer}
	The construction $\phi:\LT_n\times\{-1,1\}\to\sQ_n^\bullet$ described above is a bijection and $\phi^{-1}$ is its inverse; given $t\in\LT_n$ and $\epsilon\in\{-1,1\}$, the mapping $\phi$ naturally induces an identification between vertices of $t$ and vertices of $\phi(t,\epsilon)$ such that, if $l$ is the labelling of $t$, we have $l(v)=\dgr(v,\delta)-\dgr(\delta,\rho)$, where $v$ is interpreted as a vertex of $t$ in the left hand side of the equation and as a vertex of $\phi(t,\epsilon)$ in the right hand side, $\rho$ is the origin of $\phi(t,\epsilon)$ and $\delta$ its distinguished vertex. 
\end{theorem}

\section{An upper bound for the spectral gap of $\F^n$}\label{section: upper bound}

We will first show our upper bound for the spectral gap of $\F_n$, which will be achieved by evaluating the Dirichlet form for $\F_n$ in a function related to the radius of a quadrangulation. The same bound arises by considering many other natural functions relating to the metric structure of quadrangulations, constructed from graph distances, volumes of balls, lengths of separating cycles, etc.

Note that our proof will essentially rely on the fact that edge flips change distances by at most a constant and that the scaling limit of the radius of random quadrangulations is a known random variable (i.e.~the radius of the Brownian map). The same upper bound would thus extend to analogous edge flip chains for other classes of random planar maps which converge to the Brownian Map when rescaled by $n^{1/4}$; in particular, it implies the lower bound given by Budzinski in~\cite{B17} for the mixing time of random triangulations.

\begin{proposition}\label{upper bound}
	For the spectral gap $\nu_n$ of the Markov chain $\F^n$ of flips on quadrangulations of size $n$ we have
	$$\nu_n\leq Cn^{-\frac{5}{4}},$$
	where $C$ is some positive constant independent of $n$.
\end{proposition}

\begin{proof}
Let $\rad:\sQ_n\to\mathbb{N}$ be the mapping sending a quadrangulation $q$ to its radius, that is the maximum possible distance of a vertex of $q$ to the origin. Consider the function $f_n:\sQ_n\to\mathbb{R}$ defined as
$$f_n(q)=\frac{\rad(q)}{n^{\frac14}}.$$

Notice that, given $q\in\sQ_n$, $e\in E(q)$, $v\in V(q)$ and $s\in\{+,-\}$, we have 
\begin{equation}\label{eq:distance change}\left|\dgr^q(v,\rho_q)-\dgr^{q^{e,s}}(v,\rho_{q^{e,s}})\right|\leq 3,\end{equation}
where $\dgr^q(v,\rho)$ is the distance of $v$ to the origin of $q$ and $\dgr^{q^{e,s}}(v,\rho_{q^{e,s}})$ is the distance to the origin of $q^{e,s}$ of the vertex that corresponds to $v$ via the natural identification induced by flipping the edge $e$. Indeed, removing $e$ can only increase the distance of $v$ to $\rho$ by at most 2, while reintroducing a rotated edge can only decrease it by at most 2; if $e$ is not the root edge of $q$, then $\rho$ is still the origin in $q^{e,s}$; otherwise, $\rho_{q^{e,s}}$ is a vertex adjacent to the previous origin $\rho$. As a consequence, we have $|\rad(q)-\rad(q^{e,s})|\leq 3$.

Let us now evaluate the Dirichlet form $\mathcal{E}_{\F^n}(f_n,f_n)$; we have
$$\mathcal{E}_{\F^n}(f_n,f_n)=\frac12\sum_{\substack{q\in\sQ_n\\e\in E(q)\\s\in\{+,-\}}}(f_n(q)-f_n(q^{e,s}))^2\frac{1}{6n|\sQ_n|}=
\sum_{\substack{q\in\sQ_n\\e\in E(q),\ s\in\{+,-\}\\f_n(q)<f_n(q^{e,s})}}(f_n(q)-f_n(q^{e,s}))^2\frac{1}{6n|\sQ_n|},$$
and therefore
$$\mathcal{E}_{\F^n}(f_n,f_n)\leq \frac{3}{2n^{\frac12+1}}\sum_{\substack{q\in\sQ_n\\e\in E(q),\ s\in\{+,-\}\\f_n(q)<f_n(q^{e,s})}}\frac{1}{|\sQ_n|}=\frac{3}{2n^{\frac54}}\E_\pi(n^{-\frac14}2X),$$
where $X:\sQ_n\to\mathbb{N}$ maps $q$ to the number of edges $e$ in $E(q)$ such that $\rad(q)<\rad(q^{e,s})$ for some $s\in\{+,-\}$, and $\pi$ is the uniform probability measure on $\sQ_n$.

We intend to show that $\E_\pi(n^{-\frac14}X)$ is bounded above by a constant independent of $n$.

Given $q\in\sQ_n$, consider the set $S(q)=(B_{\rad(q)-2})^c$ of all vertices $v$ of $q$ such that $\dgr(v,\rho)\geq\rad(q)-1$, where $\rho$ is the origin of $q$. Also, for each $v$ in $S(q)$, consider a simple path $P_v$ in $q$ with endpoints $\rho$ and $v$ and length $\dgr(\rho, v)$.

Flipping an edge $e$ that is not the root of $q$ and does not belong to $\bigcup_{v\in S(q)}P_v$ cannot increase the radius of the quadrangulation; in fact, since all paths $P_v$ and the origin are preserved, the distance to the root of vertices in $S(q)$ cannot increase, and the distance to the root of any vertex outside of $S(q)$ becomes at most $\rad(q)-2+2=\rad(q)$. We thus have, for all $q\in\sQ_n$, $X(q)\leq |S(q)|\rad(q)+1$.

Thanks to the Cauchy-Schwarz inequality, we can write
$$\mathcal{E}(f_n,f_n)\leq \frac{3}{n^{\frac54}}\left(\E_\pi(|S(q)|^2)\E_\pi(f_n(q)^2)\right)^{1/2}+\frac{3}{n^{\frac32}}\leq \frac{C'}{n^{\frac54}}\left(\E_\pi(|S(q)|^2)\E_\pi(f_n(q)^2)\right)^{1/2}.$$

We claim that $\E_\pi(|S(q)|^2)$ has a finite limit as $n\to\infty$; in fact, the random variable $|S(q)|$, where $q$ is distributed according to $\pi$, has exponential tails, hence the claim: we postpone the proof of this fact to Lemma~\ref{exp tails of S} at the end of this section. Furthermore, the random variable $f_n(q)$ (considered under $\pi$) converges weakly to the range of a Brownian snake driven by a Brownian excursion, whose variance is positive, and all of its moments converge (see \cite[Corollary 3]{CS04}); thus the right hand side of
$$n^{\frac54}{\nu_n}\leq C'\frac{\left(\E_\pi(|S(q)|^2)\E_\pi(f_n(q)^2)\right)^{1/2}}{\V_\pi[f_n(q)]}$$
is bounded by a constant independent of $n$, which proves the proposition.
\end{proof}

\begin{remark}\label{upper bound generality}
Note that the above proof essentially relies on the fact that each edge flip changes the radius of the quadrangulation by a constant, as in~\eqref{eq:distance change}, and that $n^{-\frac14}X$ converges to a non-trivial random variable (for which we needed both the convergence properties of the radius in the scaling limit and some kind of control over the quantity $S(q)$). The proof above then yields a lower bound of $n^{\frac54}$ on the mixing time of the edge flip chains on any $p$-angulations provided the above two properties hold.	
\end{remark}

\begin{lemma}\label{exp tails of S}Let $q$ be random quadrangulation distributed according to the uniform probability measure $\pi$ on $\sQ_n$ and let $\rho_q$ be the origin of $q$; define $\rad(q)=\max_{v\in V(q)}\dgr(\rho_q,v)$ and $S(q)=\{v\in V(q)\mid \dgr(v,\rho_q)\geq \rad(q)-1\}$.

The random variable $|S(q)|$ has exponential tails.
\end{lemma}

\begin{proof}
The statement about $|S(q)|$ follows from the fact that the random variable $|B_2(q)|=|\{v\in V(q)\mid \dgr(\rho_q,v)\leq 2\}|$ has exponential tails (see for example the proof of Proposition 9 in \cite{BCsubdiffusive}), combined with a rerooting argument.

Consider the labelled tree $t_q=\phi^{-1}(P(q))$, where $P(q)$ is the quadrangulation $q$, pointed in its origin $\rho$. The mapping $\phi^{-1}\circ P$ is a well-known variant of the Schaeffer construction, and is a bijection between the set $\sQ_n$ and the set $\LT_n^+\times\{1\}$, where $\LT_n^+$ is the set of all labelled trees with $n$ edges such that no negative labels appear on them and that the root vertex is labelled $0$ (in particular, $t_q$ is a uniform element of $\LT_n^+$). When $\rad(q)-1> 0$, the quantity $|S(q)|$ represents the number of vertices labelled $\rad(q)-1$ or $\rad(q)-2$ in $t_q$ ($\rad(q)-1$ being the maximum label appearing on vertices of $t_q$). When $\rad(q)=1$, all $n+1$ vertices of $t_q$ are labelled 0, and $|S(q)|=n+2$.

Given a tree $t\in\LT_n^+$, let $M(t)$ be the number of vertices of $t$ whose label is either maximal or maximal minus one (so that $M(t_q)=|S(q)|$ or $M(t_q)=|S(q)|-1$); let $B(t)$ be the number of vertices of $t$ labelled 0 or 1, which represent vertices of the corresponding quadrangulation having distance $1$ or $2$ from the origin (so that $B(t_q)=|B_2(q)|-1$).

\newcommand{\Reroot}{\operatorname{Reroot}}
Now consider the map $\Reroot$ from the set $\LT^+_n$ to itself defined as follows: given a tree $t\in\LT^+_n$, obtain $\Reroot(t)$ by rerooting it in the leftmost corner (according to the clockwise contour) bearing maximal label (equal to, say, $M=\max_vl(v)$), then relabel each vertex $v$ of $t$ with the label $M-l(v)$.   Note that, given a tree $t'\in\LT^+_n$, the number of trees $t\in \LT^+_n$ such that $\Reroot(t)=t'$ is equal to the number $Z(t')$ of corners bearing maximal label that one meets in a \emph{counterclockwise} contour after the initial root corner before meeting  a corner labelled 0. Moreover, $M(t)=B(\Reroot(t))$.

It follows that, for all real numbers $\theta>0$, we have
$$\E(e^{\theta M(t_q)})=\frac{1}{|\LT_n^+|}\sum_{t\in\LT^+_n}e^{\theta M(t)}=\frac{1}{|\LT_n^+|}\sum_{t'\in\LT^+_n}\sum_{t\in\LT^+_n:\Reroot(t)=t'}e^{\theta B(t')}=\frac{1}{|\LT_n^+|}\sum_{t'\in\LT^+_n}Z(t')\cdot e^{\theta B(t')}.$$

We now wish to bound $Z(t')$ in terms of $M(t')$. Note that a corner bearing maximal label $l_{max}$ corresponds to an oriented edge whose endpoints are both counted by $M(t')$ (either both are labelled $l_{max}$ or one is labelled $l_{max}$ and the other $l_{max}-1$). For each corner of maximal label, mark the endpoint of its corresponding oriented edge that is further from the origin of the tree. Each vertex labelled $l_{max}$ or $l_{max}-1$ gets marked zero, one or two times.  It follows that $Z(t')\leq 2M(t')$, and thus the expression above is upper bounded by
$$\frac{2}{|\LT_n^+|}\sum_{t'\in\LT^+_n}M(t')\cdot e^{\theta B(t')}\leq2\left(\E\left(e^{2\theta |B_2(q)|}\right)\E\left(M(t_q)^2\right)\right)^{\frac12},$$
where the last inequality holds by Cauchy-Schwarz.

We can apply a similar argument to $M(t_q)^2$ (instead of $e^{\theta M(t_q)}$) to obtain that
$$\E\left(M(t_q)^2\right)\leq 2\E\left(|B_2(q)|^2M(t_q)\right)\leq 2\left(\E\left(|B_2(q)|^4\right)\E\left(M(t_q)^2\right)\right)^{\frac12},$$
hence $\E\left(M(t_q)^2\right)\leq 4\E\left(|B_2(q)|^4\right)$.

It follows from the above that $\E\left(e^{\theta M(t_q)}\right)\leq 4\left(\E\left(e^{2\theta |B_2(q)|}\right)\E\left(|B_2(q)|^4\right)\right)^{\frac12}$; since $B_2(q)$ has exponential tails, the upper bound is finite for small enough $\theta$, hence $M(t_q)$, and also $|S(q)|$, which differs from $M(t_q)$ by at most 1, also have exponential tails.
\end{proof}

\section{A Markov chain on labelled trees}\label{section: Markov chain on trees}
\begin{figure}\centering
\begin{tikzpicture}
	\begin{pgfonlayer}{nodelayer}
		\node [style=vertex] (0) at (-0.5, -0) {};
		\node [style=vertex] (1) at (-0.5, 1) {};
		\node [style=vertex] (2) at (0, 2) {};
		\node [style=vertex] (3) at (-1, 2) {};
		\node [style=vertex] (4) at (-1.25, 1) {};
		\node [style=vertex] (5) at (-0.5, 3) {};
		\node [style=vertex, label=above:$v$, fill=red] (6) at (0.5, 3) {};
		\node [style=vertex] (7) at (0.25, 1) {};
		\node [style=vertex] (8) at (2.75, 3) {};
		\node [style=vertex] (9) at (3.25, 2) {};
		\node [style=vertex] (10) at (2.25, 2) {};
		\node [style=vertex, fill=black!20] (11) at (3.75, 3) {};
		\node [style=vertex] (12) at (3.5, 1) {};
		\node [style=vertex] (13) at (2.75, 1) {};
		\node [style=vertex] (14) at (2, 1) {};
		\node [style=vertex] (15) at (2.75, -0) {};
		\node [style=vertex, fill=red] (16) at (3.75, 1.5) {};
		\node [style=vertex] (17) at (-3.75, 3) {};
		\node [style=vertex] (18) at (-3.25, 2) {};
		\node [style=vertex] (19) at (-4.25, 2) {};
		\node [style=vertex, fill=black!20] (20) at (-2.75, 3) {};
		\node [style=vertex] (21) at (-3, 1) {};
		\node [style=vertex] (22) at (-3.75, 1) {};
		\node [style=vertex] (23) at (-3.75, -0) {};
		\node [style=vertex] (24) at (-4.5, 1) {};
		\node [style=vertex, fill=red] (25) at (-3.75, 4) {};
		\node (26) at (-3.75, -0.8) {$t^{v,\leftarrow}$};
		\node (27) at (-0.5, -0.8) {$t$};
		\node (28) at (2.75, -0.8) {$t^{v,\rightarrow}$};
	\end{pgfonlayer}
	\begin{pgfonlayer}{edgelayer}
	\draw [->, bend left, black!20] (11) to (16);
	\draw [->, bend right, black!20] (20) to (25);
		\draw [ultra thick,->] (0) to (4);
		\draw [very thick] (0) to (1);
		\draw [very thick](1) to (3);
		\draw [very thick](1) to (2);
		\draw [very thick](2) to (5);
		\draw [very thick, red] (2) to (6);
		\draw [very thick](0) to (7);
		\draw [ultra thick,->] (15) to (14);
		\draw [very thick](15) to (13);
		\draw [very thick](13) to (10);
		\draw [very thick](13) to (9);
		\draw [very thick](9) to (8);
		\draw [dashed, black!20, very thick] (9) to (11);
		\draw [very thick](15) to (12);
		\draw [very thick, red](13) to (16);
		\draw [ultra thick,->] (23) to (24);
		\draw [very thick](23) to (22);
		\draw [very thick](22) to (19);
		\draw [very thick](22) to (18);
		\draw [very thick](18) to (17);
		\draw [dashed, very thick, black!20] (18) to (20);
		\draw [very thick](23) to (21);
		\draw [very thick, red](17) to (25);
	\end{pgfonlayer}
\end{tikzpicture}
\caption{\label{fig:translations}Left and right leaf translation.}
\end{figure}
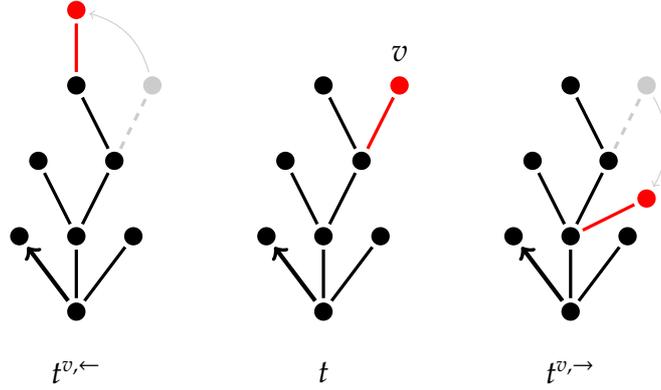

Our main results about the Markov chain $\F^n$ will be achieved via a comparison to a very natural Markov chain on labelled plane trees, which we will introduce presently.

Given a plane tree $t\in\T_n$ with contour $c_1,\ldots,c_{2n}$ and a leaf $v\in V(t)$, suppose the corner of $v$ is $c_l$ (i.e.~$c_l$ is the one corner adjacent to the vertex $v$), with $l<2n$; build a new tree $t^{v,\rightarrow}\in\T_n$ as follows: draw a new leaf $v'$ attached to the tree via the corner $c_{l+2}$ (the corner $c_1$ if $l=2n-1$); then erase $v$; if $(p(v),v)$ is the root edge of $t$ (i.e.~if $l=2$), then root $t^{v,\rightarrow}$ in $(p(p(v')),p(v'))$. Notice that, given a pair of trees $t,t'\in\T_n$, there is at most one leaf $v$ of $t$ such that $t'=t^{v,\rightarrow}$; this induces (when there is such a leaf) a natural identification between vertices of $t$ and vertices of $t'$ which sends $v$ to the ``shifted'' leaf $v'$ in $t'$ and is a tree isomorphism between the trees $\tau$ and $\tau'$ obtained from $t$ and $t'$ by erasing $v$ and $v'$. This is why, given two trees of the form $t,t^{v,x}$, we will automatically identify their vertices and denote them in the same way, including vertices $v$ and $v'$, thus taking ``vertex $v$ in $t^{v,x}$'' to mean the newly drawn leaf $v'$.

Given $t\in\T_n$, we can define analogously a tree $t^{v,\leftarrow}\in\T_n$ as the one tree such that $(t^{v,\leftarrow})^{v,\rightarrow}$ is $t$, if it exists. Additionally, we set $t^{v,\leftarrow}=t$ if $(v,p(v))$ is the root edge of $t$ (that is, in the one case where $t$ is not of the form $t'^{v,\rightarrow}$) and $t^{v\rightarrow}=t$ if the corner of $v$ is number $2n$ in the contour (so that now $t^{v,\rightarrow}$ is defined for all leaves $v$ of $t$).

When $t^{v,\rightarrow}\neq t$, we say that the tree $t^{v,\rightarrow}$ has been constructed from $t$ by \emph{translating the leaf $v$ to the right} (and $t^{v,\leftarrow}$ differs from $t$ by a \emph{leaf translation to the left}); notice that, given two trees $t,t'$ which differ by a leaf translation, there is a unique leaf $v$ of $t$ and a unique direction $d$, either $\rightarrow$ or $\leftarrow$, such that $t'$ can be expressed as $t^{v,d}$.

One could define a Markov chain $X$ on the set of plane trees with $n$ edges so that, given $X_k=t$, $X_{k+1}$ is determined by selecting an edge $(v,p(v))$ of $t$ uniformly at random and, if $v$ is a leaf of $t$, setting $X_{k+1}=t$ or $X_{k+1}=t^{v,\rightarrow}$ or $X_{k+1}=t^{v,\leftarrow}$ with equal probabilities, while $X_{k+1}=t$ if $v$ is not a leaf. We shall need a coloured variant of this chain, which can easily be defined on the set of plane trees with coloured edges $\TR_n$, where $C=\{1,\ldots,r\}$ is the set of possible edge colours. The trees $t^{v,\rightarrow}$ and $t^{v,\leftarrow}$ are defined from $t\in \TR_n$ exactly as before, by additionally ensuring that all edge colours are preserved. We can also introduce appropriate ``recolouring'' moves: given $x\in C$ and a leaf $v$ in $V(t)$, we set $t^{v,x}$ to be the tree $t$, where the edge $(v,p(v))$ is recoloured with colour $x$. One can now define $X_{k+1}$, given $X_k=t$, by selecting an edge $(v,p(v))$ of $t$ uniformly at random: if $v$ is a leaf, we set $X_{k+1}$ to be one of $t^{v,\rightarrow},t^{v,\leftarrow},t^{v,1},t^{v,2},\ldots,t^{v,r}$, each with probability $\frac{1}{r+2}$; if $v$ is not a leaf, $X_{k+1}=t$. In other words, given $t\neq t^{v,x}$, where $v$ is a leaf of $t$ and $x\in\{1,\ldots,r,\rightarrow,\leftarrow\}$, the Markov chain $X$ on $\TR_n$ has transition probability
$$p_X(t,t^{v,x})=\frac{1}{n(r+2)};$$
if $t,t'$ do not differ by a leaf translation or recolouring we have $p_X(t,t')=0$, and $p_X(t,t) \geq \frac{1}{r+2}$ since for each leaf $v$, if $c$ is the color of $v$ in $t$ we have that $t^{v,c}=t$.

Notice that the case $r=3$ corresponds to a Markov chain on the state space $\LT_n$ of labelled trees with $n$ vertices. We shall call this the \emph{leaf translation} Markov chain and will be estimating its spectral gap as well as comparing it to the spectral gap of $\F^{\bullet,n}$.

\begin{figure}[t]\centering
\begin{tikzpicture}
\node[style=vertex] (0) at (3.5,3) {};
\node[style=vertex] (1) at (3,3.75) {};
\node[style=vertex] (2) at (4,3.75) {};
\node[style=vertex] (3) at (3.5,4.5) {};
\node[style=vertex] (4) at (4.5,4.5) {};
\node[style=vertex, label=above:$v$,fill=red] (5) at (3,5.25) {};
\node[style=vertex] (6) at (4,5.25) {};
\node[style=vertex] (7) at (4,6) {};
\foreach \x in {0,...,14}
{
\node (x) at (\x*0.5,-0.25) {\scriptsize$\x$};
}
\draw[very thick] (1)--(0);
\draw[very thick] (3)--(2)--(0);
\draw[very thick] (7)--(6)--(3);
\draw[very thick] (4)--(2);
\draw[very thick,red] (5)--(3);
\draw[->] (0,0)--(7.5,0);	
\draw[->] (0,0)--(0,2.5);
\fill[red!10] (2,0)--(2,1)--(2.5,1.5)--(3,1)--(3,0)--cycle;
\draw[red,dashed] (2,0)--(2,1);
\draw[red,dashed] (3,1)--(3,0);
\draw[very thick] (0,0)--(.5,.5)--(1,0)--(1.5,0.5)--(2,1)--(2.5,1.5)--(3,1)--(3.5,1.5)--(4,2)--(4.5,1.5)--(5,1)--(5.5,.5)--(6,1)--(6.5,.5)--(7,0);
\draw[very thick, red] (2,1)--(2.5,1.5)--(3,1);
\fill[fill=red!10, rounded corners] (2,-1.55) rectangle (3,-0.5);
\node[] at (.25,-1) {\large$($};
\node[] at (.75,-1) {\large$)$};
\node[] at (1.25,-1) {\large$($};
\node[] at (1.75,-1) {\large$($};
\node[red] at (2.25,-1) {\large$($};
\node[red] at (2.75,-1) {\large$)$};
\node[] at (3.25,-1) {\large$($};
\node[] at (3.75,-1) {\large$($};
\node[] at (4.25,-1) {\large$)$};
\node[] at (4.75,-1) {\large$)$};
\node[] at (5.25,-1) {\large$)$};
\node[] at (5.75,-1) {\large$($};
\node[] at (6.25,-1) {\large$)$};
\node[] at (6.75,-1) {\large$)$};
\end{tikzpicture}
\begin{tikzpicture}
\node[style=vertex] (0) at (3.5,3) {};
\node[style=vertex] (1) at (3,3.75) {};
\node[style=vertex] (2) at (4,3.75) {};
\node[style=vertex] (3) at (3.5,4.5) {};
\node[style=vertex] (4) at (4.5,4.5) {};
\node[style=vertex, label=above:$v$,fill=red] (5) at (3.5,6) {};
\node[style=vertex] (6) at (4,5.25) {};
\node[style=vertex] (7) at (4.5,6) {};
\foreach \x in {0,...,14}
{
\node (x) at (\x*0.5,-0.25) {\scriptsize$\x$};
}
\draw[very thick] (1)--(0);
\draw[very thick] (3)--(2)--(0);
\draw[very thick] (7)--(6)--(3);
\draw[very thick] (4)--(2);
\draw[very thick,red] (5)--(6);
\draw[->] (0,0)--(7.5,0);	
\draw[->] (0,0)--(0,2.5);
\fill[red!10] (2.5,0)--(2.5,1.5)--(3,2)--(3.5,1.5)--(3.5,0)--cycle;
\draw[red,dashed] (2.5,0)--(2.5,1.5);
\draw[red,dashed] (3.5,1.5)--(3.5,0);
\draw[very thick] (0,0)--(.5,.5)--(1,0)--(1.5,0.5)--(2,1)--(2.5,1.5)--(3,2)--(3.5,1.5)--(4,2)--(4.5,1.5)--(5,1)--(5.5,.5)--(6,1)--(6.5,.5)--(7,0);
\fill[fill=red!10, rounded corners] (2.5,-1.55) rectangle (3.5,-0.5);
\draw[very thick, red] (2.5,1.5)--(3,2)--(3.5,1.5);
\node[] at (.25,-1) {\large$($};
\node[] at (.75,-1) {\large$)$};
\node[] at (1.25,-1) {\large$($};
\node[] at (1.75,-1) {\large$($};
\node[] at (2.25,-1) {\large$($};
\node[red] at (2.75,-1) {\large$($};
\node[red] at (3.25,-1) {\large$)$};
\node[] at (3.75,-1) {\large$($};
\node[] at (4.25,-1) {\large$)$};
\node[] at (4.75,-1) {\large$)$};
\node[] at (5.25,-1) {\large$)$};
\node[] at (5.75,-1) {\large$($};
\node[] at (6.25,-1) {\large$)$};
\node[] at (6.75,-1) {\large$)$};
\end{tikzpicture}
\caption{\label{fig:Dyck paths}\small Some of the natural correspondences between the set $\T_n$ of plane trees with $n$ edges and other Catalan structures result in interesting alternative interpretations of the leaf translation Markov chain. For example, given a Dyck path $D\colon \{0,\ldots,2n\}\to\mathbb{N}$ of length $2n$, say that $i\in\{1,\cdots,2n-1\}$ is an \emph{upward point} for $D$ if $D(i)=D(i-1)+1$ and say that $i$ is a \emph{peak} if it is a local maximum for $D$. Then, the leaf translation chain corresponds to selecting an upward point $i$ uniformly at random (there are exactly $n$ of them), and, if $i$ is a peak, then to either leaving it untouched or shifting it to the right (if possible) or shifting it to the left (if possible), each with probability 1/3. Shifting a peak $i$ of $D$ to the right, for example, can be done if $i\leq2n-2$ and consists in constructing a Dyck path $D'$ such that $D'(i)=D'(i+2)=D(i+2)$, $D'(i+1)=D(i+2)+1$ and $D'(j)=D(j)$ for all $j\in\{1,\ldots,2n\}\setminus\{i,i+1,i+2\}$; the above picture shows an example with $i=5$. An analogous interpretation for the leaf translations can be given on strings of balanced parentheses, where one selects an open parenthesis at random and, if it is immediately followed by a closed parenthesis, then the pair ``$()$'' may remain untouched or move one place to the right or to the left.}
\end{figure}

The leaf translation Markov chain on $\TR_n$ (including the simpler variant introduced at the beginning of this section for $r=1$) is a very natural chain, worthy in fact of study independently of our efforts with regards to $\F^n$. We have chosen to present it as a chain on $\TR_n$, but its transitions appear very natural for a number of different interpretations of the state space via classical bijections between Catalan structures (see Figure~\ref{fig:Dyck paths}). Indeed, variants of this chain have been discussed in the physics literature under the name of \emph{Fredkin spin models}, and have been investigated by Movassagh and Shor~\cite{MS16, Mov16}, relying on work by Bravyi et al.~\cite{Bravyi12}.

In particular, Movassagh and Shor prove a lower bound of $Cn^{-\frac{11}{2}}$ for the spectral gap of a chain on $\TR_n$ closely related to $X$, from which a bound for the spectral gap of $X$ can be gleaned; we shall partially follow their argument for estimating the spectral gap, but will improve their results and will therefore provide a complete proof of our lower bounds in the next section. In order to do this, we will now introduce a variant of the leaf translation Markov chain which is closer to the one originally considered by Movassagh and Shor in their proofs; even though bounding its spectral gap is not strictly speaking necessary for achieving our results for $X$ and therefore $\F^n$, we believe our improved bound to be of independent interest, and the proof -- which is somewhat simpler than the one for $X$ -- to provide a handy way to more naturally introduce some of the necessary notation and showcase the basic argument.

We shall introduce this chain on the set $\TR_n$ and refer to \cite{Mov16} for a presentation as a chain on the set of (coloured) Dyck paths. Given $t\in \TR_n$, a leaf $v\in V(t)$, an integer $k\in \{1,\ldots,2n-1\}$ and a colour $c\in\{1,\ldots,r\}$, we define $t^{v,k,c}\in \TR_n$ by the following procedure (Figure~\ref{fig:replanting}):
\begin{itemize}
\item erase $(v,p(v))$ from $t$, thus obtaining $t'\in\TR_{n-1}$;
\item consider the clockwise contour $c_1,\ldots,c_{2n-2}$ of $t'$. If $1<k<2n-1$, add a leaf $v'$ to $t'$	via its corner $c_k$; if $k=1$ or $k=2n-1$, add a leaf $v'$ to $t'$ via the root corner $c_1$: if $k=1$, let the new root corner be the one right before the added leaf, i.e.~reroot $t'$ so that the corner of $v'$ becomes the second corner of the contour; if $k=2n-1$, let the root corner be the one right after $v'$, so that the corner of $v'$ is the last one in the clockwise contour of the new tree;
\item colour the edge $(v',p(v'))$ with the colour $c$.\end{itemize}

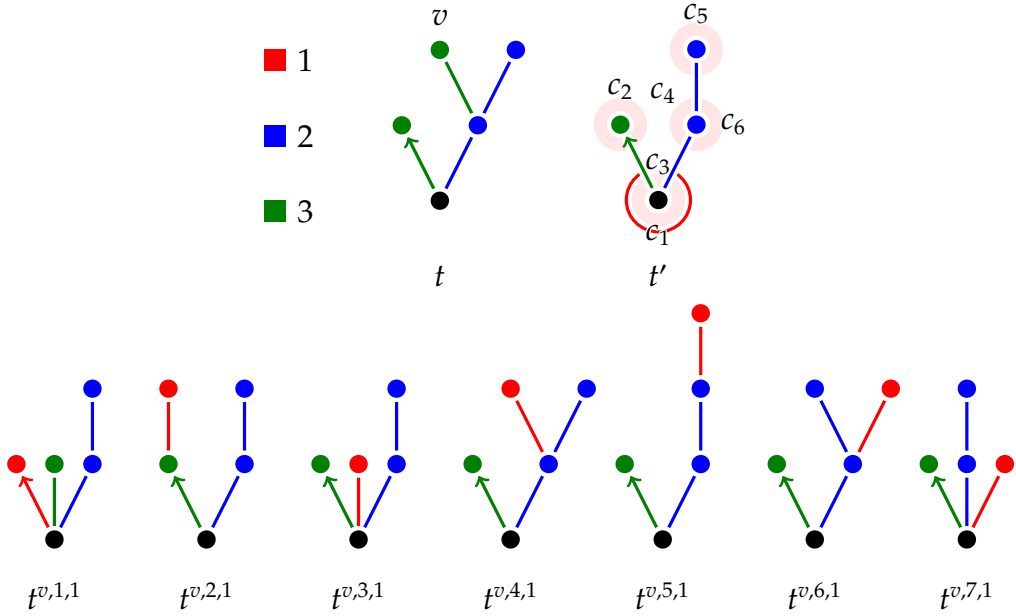
\begin{figure}
\centering
\begin{tikzpicture}
		\node [fill=red, label=right:1] (0) at (0, 1) {};
		\node [fill=blue, label=right:2] (1) at (0, 0) {};
		\node [fill=green!50!black, label=right:3] (1) at (0, -1) {};
		\node (x) at (0, -2) {};
\end{tikzpicture}\qquad
\begin{tikzpicture}
	\begin{pgfonlayer}{nodelayer}
		\node [style=vertex] (0) at (0.5, -0) {};
		\node [style=vertex, fill=blue] (1) at (1.5, 2) {};
		\node [style=vertex, fill=green!50!black] (2) at (0, 1) {};
		\node [style=vertex, fill=blue] (3) at (1, 1) {};
		\node [style=vertex, fill=green!50!black, label=above:$v$] (4) at (0.5, 2) {};
		\node (x) at (0.5, -1) {$t$};
	\end{pgfonlayer}
	\begin{pgfonlayer}{edgelayer}
		\draw [very thick, <-, green!50!black] (2) to (0);
		\draw [very thick, blue] (0) to (3);
		\draw [very thick, blue] (3) to (1);
		\draw [very thick, green!50!black] (4) to (3);
	\end{pgfonlayer}
\end{tikzpicture}\qquad
\begin{tikzpicture}
	\begin{pgfonlayer}{nodelayer}
		\node [style=vertex, label=below:\contour{white}{$c_1$}, label=above:\contour{white}{$c_3$}] (0) at (0.5, -0) {};
		\node [style=vertex, label=above:\contour{white}{$c_5$}, fill=blue] (1) at (1, 2) {};
		\node [style=vertex, label=above:\contour{white}{$c_2$}, fill=green!50!black] (2) at (0, 1) {};
		\node [style=vertex, label=above left:\contour{white}{$c_4$}, label=right:\contour{white}{$c_6$}, fill=blue] (3) at (1, 1) {};
				\node (x) at (0.5, -1) {$t'$};
	\end{pgfonlayer}
	\begin{pgfonlayer}{edgelayer}
		\draw[red, very thick] (0) circle (12pt);
	\fill[white] (2.center)--(0.center)--(3.center)--cycle;
	\fill[red!10] (0) circle (10pt);
	\fill[red!10] (1) circle (10pt);
	\fill[red!10] (2) circle (10pt);
	\fill[red!10] (3) circle (10pt);
		\draw[b] (2) to (0);
		\draw[b] (0) to (3);
		\draw[b] (3) to (1);
		\draw[very thick,<-,green!50!black] (2) to (0);
		\draw[very thick,blue] (0) to (3);
		\draw[very thick,blue] (3) to (1);
	\end{pgfonlayer}
\end{tikzpicture}\\
\begin{tikzpicture}
	\begin{pgfonlayer}{nodelayer}
		\node [style=vertex,fill=blue] (0) at (-5.25, -0) {};
		\node [style=vertex,fill=blue] (1) at (-5.25, 1) {};
		\node [style=vertex,fill=green!50!black] (2) at (-5.75, -0) {};
		\node [style=vertex] (3) at (-5.75, -1) {};
		\node [style=vertex, fill=red] (4) at (-6.25, -0) {};
		\node [style=vertex,fill=blue] (5) at (-3.25, -0) {};
		\node [style=vertex,fill=blue] (6) at (-3.25, 1) {};
		\node [style=vertex,fill=green!50!black] (7) at (-4.25, -0) {};
		\node [style=vertex] (8) at (-3.75, -1) {};
		\node [style=vertex, fill=red] (9) at (-4.25, 1) {};
		\node [style=vertex,fill=blue] (10) at (-1.25, -0) {};
		\node [style=vertex,fill=blue] (11) at (-1.25, 1) {};
		\node [style=vertex,fill=green!50!black] (12) at (-2.25, -0) {};
		\node [style=vertex] (13) at (-1.75, -1) {};
		\node [style=vertex, fill=red] (14) at (-1.75, -0) {};
		\node [style=vertex,fill=blue] (15) at (0.75, -0) {};
		\node [style=vertex,fill=blue] (16) at (1.25, 1) {};
		\node [style=vertex,fill=green!50!black] (17) at (-0.25, -0) {};
		\node [style=vertex] (18) at (0.25, -1) {};
		\node [style=vertex, fill=red] (19) at (0.25, 1) {};
		\node [style=vertex,fill=blue] (20) at (4.75, -0) {};
		\node [style=vertex,fill=blue] (21) at (4.25, 1) {};
		\node [style=vertex,fill=green!50!black] (22) at (3.75, -0) {};
		\node [style=vertex] (23) at (4.25, -1) {};
		\node [style=vertex, fill=red] (24) at (5.25, 1) {};
		\node [style=vertex,fill=blue] (25) at (6.25, -0) {};
		\node [style=vertex,fill=blue] (26) at (6.25, 1) {};
		\node [style=vertex,fill=green!50!black] (27) at (5.75, -0) {};
		\node [style=vertex] (28) at (6.25, -1) {};
		\node [style=vertex, fill=red] (29) at (6.75, -0) {};
		\node [style=vertex,fill=blue] (30) at (2.75, -0) {};
		\node [style=vertex,fill=blue] (31) at (2.75, 1) {};
		\node [style=vertex,fill=green!50!black] (32) at (1.75, -0) {};
		\node [style=vertex] (33) at (2.25, -1) {};
		\node [style=vertex, fill=red] (34) at (2.75, 2) {};]
		\node (40) at (-5.75, -1.75) {$t^{v,1,1}$};
		\node (41) at (-3.75, -1.75) {$t^{v,2,1}$};
		\node (42) at (-1.75, -1.75) {$t^{v,3,1}$};
		\node (43) at (0.25, -1.75) {$t^{v,4,1}$};
		\node (44) at (2.25, -1.75) {$t^{v,5,1}$};
		\node (45) at (4.25, -1.75) {$t^{v,6,1}$};
		\node (46) at (6.25, -1.75) {$t^{v,7,1}$};
	\end{pgfonlayer}
	\begin{pgfonlayer}{edgelayer}
		\draw [very thick,green!50!black] (2) to (3);
		\draw [very thick,blue] (3) to (0);
		\draw [very thick,blue] (0) to (1);
		\draw [very thick, red, ->](3) to (4);
		\draw [very thick, <-,green!50!black](7) to (8);
		\draw [very thick,blue](8) to (5);
		\draw [very thick,blue](5) to (6);
		\draw [very thick, red](7) to (9);
		\draw [very thick, <-,green!50!black](12) to (13);
		\draw [very thick,blue](13) to (10);
		\draw [very thick,blue](10) to (11);
		\draw [very thick, red](13) to (14);
		\draw [very thick, <-,green!50!black](17) to (18);
		\draw [very thick,blue](18) to (15);
		\draw [very thick,blue](15) to (16);
		\draw [very thick, red](15) to (19);
		\draw [very thick, <-,green!50!black](22) to (23);
		\draw [very thick,blue](23) to (20);
		\draw [very thick,blue](20) to (21);
		\draw [very thick, red](20) to (24);
		\draw [very thick, <-,green!50!black](27) to (28);
		\draw [very thick,blue](28) to (25);
		\draw [very thick,blue](25) to (26);
		\draw [very thick, red](28) to (29);
		\draw [very thick, <-,green!50!black](32) to (33);
		\draw [very thick,blue](33) to (30);
		\draw [very thick,blue](30) to (31);
		\draw [very thick, red](31) to (34);
	\end{pgfonlayer}
\end{tikzpicture}
\caption{\label{fig:replanting}The leaf replanting move of a leaf $v$ performed on a tree $t\in\T^{(3)}_4$: above, the tree $t'	\in\T^{(3)}_3$ and its contour; below, the trees $t^{v,k,1}$ for $k=1,\ldots,7$.}
\end{figure}

Notice that, if the corner of $v$ is the $k$-th corner in the clockwise contour of $t$ with $1<k\leq 2n-1$ and $c$ is the colour of $(v,p(v))$ in $t$, then $t^{v,\rightarrow}=t^{v,k,c}$. Similarly, if $k>2$, we have $t^{v,{k-2},c}=t^{v,\leftarrow}$, and if $1 < k \leq 2n$ we have $t^{v,k-1,c}=t$.

We define the \emph{leaf replanting} Markov chain $Y$ on the state space $\TR_n$ by choosing, if $Y_k=t$, a uniformly random edge $(v,p(v))$ of $t$; if $v$ is not a leaf, then we set $Y_{k+1}=t$; if $v$ is a leaf, we set $Y_{k+1}$ to be $t^{v,k,c}$, where $k$ and $c$ are chosen independently and uniformly at random in $\{1,\ldots,2n-1\}$ and $\{1,\ldots,r\}$ respectively.

While we were able to identify the leaf being moved between two trees that differ by a leaf translation, notice that this is not the case when we're dealing with a leaf replanting. In general, we have
$$p_Y(t,t')=\sum_{v,k,c}\frac{1}{n(2n-1)r}1_{t'=t^{v,k,c}}.$$

\begin{remark}\label{basic tree chains}
Notice that both $X$ and $Y$ are reversible, irreducible and aperiodic. Reversibility and aperiodicity are clear from the definition; irreducibility is also clear: given any tree in $\TR_n$, one can turn it into the tree of height 1 whose edges are all coloured 1 with at most $n$ transitions from $Y$ (indeed, it suffices to apply the replanting $\cdot\mapsto \cdot^{v,1,1}$ on the rightmost leaf $v$ of the tree several times). Since each leaf replanting can actually be obtained by concatenating at most $2n-1$ transitions from $X$ ($2n-2$ translations and one recolouring), the same height one tree can be obtained with at most $n(2n-1)$ transitions from $X$. The two Markov chains $X$ and $Y$ therefore both admit the uniform measure on $\TR_n$ as their unique stationary distribution.
\end{remark}

\subsection{A lower bound for the spectral gaps of leaf replanting and leaf translation Markov chains}

In order to prove the desired lower bounds, we first need to set up some machinery; the first part of this section will be devoted to constructing a family of probability measures on sequences of transitions for the Markov chain $\CLR$. This will be done via a family of functions $f_n:\TR_n\times\TR_{n-1}\to \mathbb{R}$ and a function $F:(\TR_n)^2\to\TR_{n-1}$ with some specific properties, which we now state. Throughout this section, we will make extensive use of the fact that plane trees are counted by Catalan numbers, hence
\begin{equation}\label{catalan}\left|\TR_n\right|=\frac{r^n}{n+1}{2n \choose n}\sim \frac{(4r)^n}{\sqrt{\pi}n^{\frac32}}\end{equation}
and in particular $|\TR_{n+1}|< 4 r\cdot|\TR_n|$.

\begin{figure}\centering
\begin{tikzpicture}
	\begin{pgfonlayer}{nodelayer}
		\node [style=vertex, fill=blue] (0) at (0.5, -0) {};
		\node [style=vertex, fill=red] (1) at (-0.75, 1) {};
		\node [style=vertex, fill=red] (2) at (-1.25, 2) {};
		\node [style=vertex, fill=red] (3) at (-0.25, 2) {};
		\node [style=vertex, fill=red] (4) at (-1.25, 3) {};
		\node [style=vertex, fill=red] (5) at (-1.75, 4) {};
		\node [style=vertex, fill=red] (6) at (-0.75, 4) {};
		\node [style=vertex, fill=blue] (7) at (0.5, 1) {};
		\node [style=vertex, fill=blue] (8) at (1.5, 1) {};
		\node [style=vertex, fill=blue] (9) at (1.5, 2) {};
		\node [style=vertex, fill=blue] (10) at (1, 3) {};
		\node [style=vertex, fill=blue] (11) at (2, 3) {};
		\node [style=vertex, fill=blue] (24) at (0.5, 2) {};
		\node [style=vertex, fill=blue] (25) at (2, 4) {};
	\end{pgfonlayer}
	\begin{pgfonlayer}{edgelayer}
	\node[fit=(1)(2)(3)(4)(5)(6), fill=red!10, rounded corners,label=below left:$L(t)$] {};
	\node[fit=(0)(7)(8)(9)(10)(11)(25), rounded corners, fill=blue!10,label=above right:$R(t)$] {};
		\draw[ultra thick,->] (0) to (1);
		\draw[red,thick,->] (1) to (2);
		\draw[red,thick] (1) to (3);
		\draw[red,thick] (2) to (4);
		\draw[red,thick] (4) to (5);
		\draw[red,thick] (4) to (6);
		\draw[blue,thick,->] (0) to (7);
		\draw[blue,thick] (0) to (8);
		\draw[blue, thick] (8) to (9);
		\draw[blue, thick] (10) to (9);
		\draw[blue, thick] (9) to (11);
		\draw[blue, thick] (7) to (24);
		\draw[blue, thick] (11) to (25);
	\end{pgfonlayer}
\end{tikzpicture}
\caption{\label{fig:L,R}The decomposition of a tree $t\in\T_{13}$ into its left and right components $L(t)\in\T_5$ and $R(t)\in\T_7$.}
\end{figure}
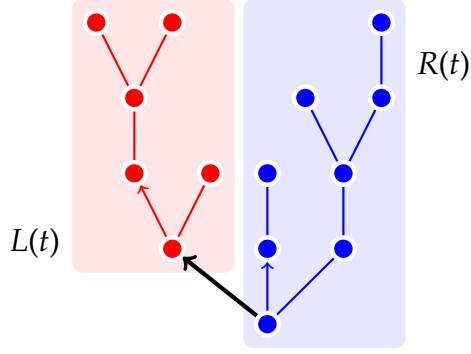

\begin{proposition}\label{hierarchy}For $n\geq 1$, there exists a mapping $f_n:\TR_n\times\TR_{n-1}\to\mathbb{R}$ such that
\begin{itemize}
	\item[(i)] $f_n(t,t')=0$ if $t'$ cannot be obtained from $t$ by deleting a leaf;
	\item[(ii)] $\displaystyle\sum_{t'\in\TR_{n-1}}f_n(t,t')=1$ for all $t$ in $\TR_n$;
	\item[(iii)] $\displaystyle\sum_{t\in\TR_n}f_n(t,t')=\frac{|\TR_n|}{|\TR_{n-1}|}$ for all $t'$ in $\TR_{n-1}$.
\end{itemize}
\end{proposition}

\begin{proof}
	We can recursively construct a mapping $f_n$ with the required properties. 
	
	Indeed, $\TR_1\times\TR_0=\{P_1^{(1)},\ldots,P_1^{(r)}\}\times\{\bullet\}$, where $P_1^{(i)}$ is the tree with one edge which is coloured $i$ and $\bullet$ is the single vertex, and we can set $f_1(P_1^{(i)},\bullet)=1$.
	
	Notice that, for $n>1$, we can define two functions $L,R:\TR_n\to\bigcup_{k=0}^{n-1}\TR_k$ (see Figure~\ref{fig:L,R}) by setting $L(t)$ to be the tree of descendants of $u$ in $t$, where $(\rho,u)$ is the root edge of $t$ (including $u$ and with the natural rooting induced by that of $t$, unless $u$ is a leaf, in which case $L(t)=\bullet$) and $R(t)$ to be the tree obtained from $t$ by erasing $L(t)$ (and the edge $(\rho,u)$), rooted in the corner that contains the original root corner of $t$ (unless $\rho$ has degree 1, in which case $R(t)=\bullet$). We have $0\leq|L(t)|\leq n-1$ and $|L(t)|+|R(t)|=n-1$, and $t \mapsto (L(t),R(t),c)$, where $c\in\{1,\ldots,r\}$ is the colour of the root edge in $t$, is a bijection between $\TR_n$ and $\bigcup_{k=0}^{n-1}(\TR_k\times\TR_{n-1-k})\times\{1,\ldots,r\}$.
	
	We will now construct $f_n:\TR_n\times\TR_{n-1}$, for $n>1$, from the mappings $f_1,\ldots,f_{n-1}$, by using the projections $L,R$ and $n$ constants $C^{(n)}_0,\ldots,C^{(n)}_{n-1}$ which will be explicitly worked out below. Consider $(t,t')\in\TR_n\times\TR_{n-1}$. If $|L(t)|>0$ and $t'$ can be obtained from $t$ by deleting a leaf contained in $L(t)$, set $f_n(t,t')=C^{(n)}_{|L(t)|}f_{|L(t)|}(L(t),L(t'))$; if $t'$ can be obtained from $t$ by deleting a leaf contained in $R(t)$, set $f_n(t,t')=C^{(n)}_{|R(t)|}f_{|R(t)|}(R(t),R(t'))$. Notice that the two conditions are mutually exclusive; if neither is satisfied, set $f_n(t,t')=0$.
	
	We now have, for all $t$ in $\TR_n$ such that $|L(t)|=k$ (where $1\leq k\leq n-2$),
	$$\sum_{t'\in\TR_{n-1}}f_n(t,t')=C^{(n)}_{k}\sum_{t_l\in\TR_{k-1}}f_{k}(L(t),t_l)+C^{(n)}_{n-k-1}\sum_{t_r\in\TR_{n-k-2}}f_{n-k-1}(R(t),t_r)=C^{(n)}_{k}+C^{(n)}_{n-k-1},$$
	as well as $\sum_{t'\in\TR_{n-1}}f_n(t,t')=C^{(n)}_{n-1}$ if $|L(t)|=0$ or $|L(t)|=n-1$.
	
	Furthermore, for all $t'\in\TR_{n-1}$ such that $|L(t')|=k$ (where $0\leq k\leq n-2$),
	$$\sum_{t\in\TR_n}f_n(t,t')=C^{(n)}_{k+1}\sum_{t_l\in\TR_{k+1}}f_{k+1}(t_l,L(t'))+C^{(n)}_{n-k-1}\sum_{t_r\in\TR_{n-k-1}}f_{n-k-1}(t_r,R(t'))=C^{(n)}_{k+1}\frac{|\TR_{k+1}|}{|\TR_k|}+C^{(n)}_{n-k-1}\frac{|\TR_{n-k-1}|}{|\TR_{n-k-2}|}.
	$$
	
	Notice that $f_n$ has property (i) by construction; to enforce properties (ii) and (iii), it is sufficient to choose
	$$C^{(n)}_i=\frac{i(i+1)(3n-2i-1)}{(n-1)n(n+1)}$$
	for $i=0,\ldots,n-1$, since one then has $C^{(n)}_0=0$ and $C^{(n)}_i+C^{(n)}_{n-i-1}=1$ for $0\leq i\leq n-1$, as well as 
	$$C^{(n)}_i\frac{|\TR_{i}|}{|\TR_{i-1}|}+C^{(n)}_{n-i}\frac{|\TR_{n-i}|}{|\TR_{n-i-1}|}=C^{(n)}_i\frac{r|2(2i-1)|}{i+1}+C^{(n)}_{n-i}\frac{2r(2n-2i-1)}{n-i+1}=\frac{2r(2n-1)}{n+1}=\frac{|\TR_{n}|}{|\TR_{n-1}|}$$
	for all $1\leq i\leq n-1$, by~\eqref{catalan}.
\end{proof}

\begin{lemma}\label{small fibers}
There is a mapping $F:(\TR_n)^2\to\TR_{n-1}$ such that for all $t\in\TR_n$ and $\tau\in\TR_{n-1}$ we have $|\{t'|F(t,t')=\tau\}|\leq 8r$ and $|\{t'|F(t',t)=\tau\}|\leq 8r$.
\end{lemma}
\begin{proof}
	Enumerate the elements of $\TR_n$ as $t_1,\ldots,t_{|\TR_n|}$ and the elements of $\TR_{n-1}$ as $\tau_1,\ldots,\tau_{|\TR_{n-1}|}$, in some order. The function that sends $(t_a,t_b)$ to $\tau_S$, where $S=(a+b)\mod |\TR_{n-1}|$, satisfies the requirements. Indeed, we have $|\TR_n|=2r\frac{2n-1}{n+1}|\TR_{n-1}|<4r|\TR_{n-1}|$ and therefore $a+b\leq 2|\TR_n|<8r|\TR_{n-1}|$. Given $\tau=\tau_S$ and $t=t_a$, there are then at most $8r$ possibilities for $a+b$ and therefore at most $8r$ possibilities for $b$. 
\end{proof}

We intend to prove lower bounds for the spectral gap of the leaf replanting Markov chain $\CLR$ by assigning each pair of trees $x,y\in\TR_n$ a canonical path of leaf replanting moves turning $x$ into $y$ -- or rather a probability measure on the set of possible paths from $x$ to $y$. Such a probability measure will be constructed by using a set of functions $f_i:\TR_i\times\TR_{i-1}\to\mathbb{R}$ with the requirements of Proposition~\ref{hierarchy} and a function $F:(\TR_n)^2\to\TR_{n-1}$ as in Lemma~\ref{small fibers}.

Let $L,R:\TR_n\to\bigcup_{i=0}^{n-1}\TR_i$ be the mappings defined within the proof of Proposition~\ref{hierarchy} and depicted in Figure~\ref{fig:L,R}.

First, given a tree $x\in\TR_n$ and a tree $x'\in\TR_n$ such that $L(x')\in\TR_{n-1}$, we will define a probability measure on paths from $x$ to $x'$. Then, given generic trees $x,y\in\TR_n$, we will construct random paths from $x$ to $y$ by concatenating paths from $x$ to $z\in\TR_n$, where $L(z)=F(x,y)\in\TR_{n-1}$, and from $z$ to $y$.

Given $t\in\TR_n$, consider the set $\Gamma^t$ of sequences $t_0,\ldots,t_n$ such that $t_0=t$, $t_i\in \TR_{n-i}$ and $t_{i+1}$ is obtained from $t_i$ by erasing a leaf; define the probability measure $Q^t$ on $\Gamma^t$ as

$$Q^t(t_0,\ldots,t_n)=\prod_{i=0}^{n-1}f_{n-i}(t_i,t_{i+1}),$$
that is the law of a sequence of random trees $\theta_0,\ldots,\theta_n$ such that $\theta_0=t$ and that, given $\theta_i=t_i\in\TR_{n-i}$, the tree $\theta_{i+1}$ is chosen in $\TR_{n-i-1}$ according to the probability measure $f_{n-i}(t_i,\cdot)$.

Given $x,y\in\TR_n$, consider now the set $\Gamma_{x\to y}$ of all paths $(t_i,t_{i+1})_{i=0}^{2n-1}$ such that
\begin{itemize}
	\item for all $i$ between $0$ and $2n$, the tree $t_i$ belongs to $\TR_n$;
	\item $t_0=x$, $t_{2n}=y$ and $L(t_n)=F(x,y)$ (which, since $|F(x,y)|=n-1$, determines $t_n$);
	\item the tree $t_1$ is obtained by replanting a leaf of $t_0$ onto corner 1; for $0<i<n$, the tree $t_{i+1}$ is obtained by removing a leaf from $R(t_i)$ and replanting it onto a corner of $L(t_i)$;
	\item similarly, the tree $t_{2n-1}$ can be obtained by replanting a leaf of $t_{2n}$ onto corner 1; for $n<i<2n$, the tree $t_{i-1}$ can be obtained by removing a leaf from $R(t_i)$ and replanting it onto a corner of $L(t_i)$.
\end{itemize}
In other words, for $\gamma=(t_i,t_{i+1})_{i=0}^{2n-1}\in\Gamma_{x\to y}$, we have that the two sequences $L_1(\gamma)=(L(t_n),\ldots,L(t_1))$ and $L_2(\gamma)=(L(t_n),\ldots,L(t_{2n-1}))$ belong to $\Gamma^{F(x,y)}$, while $R_1(\gamma)=(t_0, R(t_1),\ldots,R(t_n))\in\Gamma^{x}$ and $R_2(\gamma)=(t_{2n},R(t_{2n-1}),\ldots,R(t_n))\in\Gamma^{y}$. Vice-versa, any quadruple of sequences $L_1,L_2\in \Gamma^{F(x,y)}, R_1\in\Gamma^x,R_2\in\Gamma^y$ can be assembled into a path $\gamma\in\Gamma_{x\to y}$.

We can thus construct a probability measure $P_{x\to y}$ on $\Gamma_{x\to y}$ by setting
$$P_{x\to y}(\gamma)=Q^x(R_1(\gamma))Q^y(R_2(\gamma))Q^{F(x,y)}(L_1(\gamma))Q^{F(x,y)}(L_2(\gamma)).$$

\begin{figure}
\centering\begin{tabularx}{.9\textwidth}{>{\centering\arraybackslash\hsize=.62\hsize}X>{\centering\arraybackslash}X>{\centering\arraybackslash}X>{\centering\arraybackslash}X>{\centering\arraybackslash}X>{\raggedright\arraybackslash}X}
$t_0$ & $t_1$ & $t_2$ & $t_3$ & $t_4$ & \\[.2cm]
\begin{tikzpicture}
	\begin{pgfonlayer}{nodelayer}
		\node [style=vertex] (0) at (0, -0) {};
		\node [style=vertex] (1) at (0, 0.75) {};
		\node [style=vertex] (2) at (0.5, 0.75) {};
		\node [style=vertex] (3) at (0.5, 1.5) {};
		\node [style=vertex] (4) at (0.25, 2.25) {};
		\node [style=vertex] (5) at (0.75, 2.25) {};
	\end{pgfonlayer}
	\begin{pgfonlayer}{edgelayer}
	\draw[blue] (-0.25,-0.25) rectangle (1,2.5);
	\node (x) at (0.5,-0.55) {$R_1^0$};
		\draw[very thick] (1) to (0);
		\draw[very thick] (0) to (2);
		\draw[very thick] (2) to (3);
		\draw[very thick] (3) to (4);
		\draw[very thick] (3) to (5);
	\end{pgfonlayer}
\end{tikzpicture}& 
\begin{tikzpicture}
	\begin{pgfonlayer}{nodelayer}
		\node [style=vertex] (0) at (0, -0) {};
		\node [style=vertex,fill=red] (4x) at (-0.5, 0.75) {};
		\node [style=vertex] (1) at (0, 0.75) {};
		\node [style=vertex] (2) at (0.5, 0.75) {};
		\node [style=vertex] (3) at (0.5, 1.5) {};
		\node [style=vertex,fill=black!10] (4) at (0.25, 2.25) {};
		\node [style=vertex] (5) at (0.75, 2.25) {};
	\end{pgfonlayer}
	\begin{pgfonlayer}{edgelayer}
	\draw[blue] (-0.2,-0.25) rectangle (1,2.5);
	\draw[red] (-0.3,-0.25) rectangle (-1.5,2.5);
	\node (x) at (-1,-0.55) {$L_1^5$};
	\node (x) at (0.5,-0.55) {$R_1^1$};
		\draw[very thick] (1) to (0);
		\draw[very thick] (0) to (2);
		\draw[very thick] (2) to (3);
		\draw[very thick, dashed, black!10] (3) to (4);
		\draw[very thick] (3) to (5);
		\draw[very thick,red] (4x) to (0);
	\end{pgfonlayer}
\end{tikzpicture}&
\begin{tikzpicture}
	\begin{pgfonlayer}{nodelayer}
		\node [style=vertex] (0) at (0, -0) {};
		\node [style=vertex] (4) at (-0.5, 0.75) {};
		\node [style=vertex] (1) at (0, 0.75) {};
		\node [style=vertex] (2) at (0.5, 0.75) {};
		\node [style=vertex] (3) at (0.5, 1.5) {};
		\node [style=vertex,fill=black!10] (5) at (0.75, 2.25) {};
		\node [style=vertex,fill=red] (5x) at (-0.5, 1.5) {};	
	\end{pgfonlayer}
	\begin{pgfonlayer}{edgelayer}
		\draw[blue] (-0.2,-0.25) rectangle (1,2.5);
	\draw[red] (-0.3,-0.25) rectangle (-1.5,2.5);
	\node (x) at (-1,-0.55) {$L_1^4$};
	\node (x) at (0.5,-0.55) {$R_1^2$};
		\draw[very thick] (1) to (0);
		\draw[very thick] (0) to (2);
		\draw[very thick] (2) to (3);
		\draw[very thick, black!10,dashed] (3) to (5);
		\draw[very thick] (4) to (0);
		\draw[very thick,red] (5x) to (4);
	\end{pgfonlayer}
\end{tikzpicture}&
\begin{tikzpicture}
	\begin{pgfonlayer}{nodelayer}
		\node [style=vertex] (0) at (0, -0) {};
		\node [style=vertex] (4) at (-0.5, 0.75) {};
		\node [style=vertex,fill=black!10] (1) at (0, 0.75) {};
		\node [style=vertex] (2) at (0.5, 0.75) {};
		\node [style=vertex] (3) at (0.5, 1.5) {};
		\node [style=vertex] (5) at (-0.5, 1.5) {};
		\node [style=vertex,fill=red] (1x) at (-1, 1.5) {};
	\end{pgfonlayer}
	\begin{pgfonlayer}{edgelayer}
		\draw[blue] (-0.2,-0.25) rectangle (1,2.5);
	\draw[red] (-0.3,-0.25) rectangle (-1.5,2.5);
	\node (x) at (-1,-0.55) {$L_1^3$};
	\node (x) at (0.5,-0.55) {$R_1^3$};
		\draw[very thick,red] (1x) to (4);
		\draw[very thick] (0) to (2);
		\draw[very thick] (2) to (3);
		\draw[very thick, black!10,dashed] (1) to (0);
		\draw[very thick] (4) to (0);
		\draw[very thick] (5) to (4);
	\end{pgfonlayer}
\end{tikzpicture}&
\begin{tikzpicture}
	\begin{pgfonlayer}{nodelayer}
		\node [style=vertex] (0) at (0, -0) {};
		\node [style=vertex] (4) at (-0.5, 0.75) {};
		\node [style=vertex] (2) at (0.5, 0.75) {};
		\node [style=vertex,fill=black!10] (3) at (0.5, 1.5) {};
		\node [style=vertex] (5) at (-0.5, 1.5) {};
		\node [style=vertex] (1) at (-1, 1.5) {};
		\node [style=vertex,fill=red] (3x) at (-0.5, 2.25) {};
	\end{pgfonlayer}
	\begin{pgfonlayer}{edgelayer}
		\draw[blue] (-0.2,-0.25) rectangle (1,2.5);
	\draw[red] (-0.3,-0.25) rectangle (-1.5,2.5);
	\node (x) at (-1,-0.55) {$L_1^2$};
	\node (x) at (0.5,-0.55) {$R_1^4$};
		\draw[very thick] (1) to (4);
		\draw[very thick] (0) to (2);
		\draw[very thick, black!10,dashed] (2) to (3);
		\draw[very thick] (4) to (0);
		\draw[very thick] (5) to (4);
		\draw[very thick,red] (3x) to (5);
	\end{pgfonlayer}
\end{tikzpicture}&
\multirow{2}{*}[2cm]{\begin{tikzpicture}
	\begin{pgfonlayer}{nodelayer}
		\node [style=vertex] (0) at (0, -0) {};
		\node [style=vertex] (4) at (-0.5, 0.75) {};
		\node [style=vertex, fill=black!10] (2) at (0.5, 0.75) {};
		\node [style=vertex] (5) at (-0.5, 1.5) {};
		\node [style=vertex] (1) at (-1, 1.5) {};
		\node [style=vertex] (3) at (-0.5, 2.25) {};
		\node [style=vertex, fill=red] (2x) at (-1, 2.25) {};
		\node (x) at (-0.25,3) {$t_5$};
	\end{pgfonlayer}
	\begin{pgfonlayer}{edgelayer}
		\draw[blue] (-0.2,-0.25) rectangle (1,2.5);
	\draw[red] (-0.3,-0.25) rectangle (-1.5,2.5);
	\node (x) at (-1,-0.55) {$L_1^1=L_2^1$};
	\node (x) at (0.5,-0.55) {$R_1^5=R_2^5$};
		\draw[very thick] (1) to (4);
		\draw[very thick, black!10, dashed] (0) to (2);
		\draw[very thick] (4) to (0);
		\draw[very thick] (5) to (4);
		\draw[very thick] (3) to (5);
		\draw[very thick,red] (2x) to (1);
	\end{pgfonlayer}
\end{tikzpicture}}\\
\begin{tikzpicture}
\path[use as bounding box] (-0.25,-0.95) rectangle (1,2.5);
	\begin{pgfonlayer}{nodelayer}
		\node [style=vertex] (0) at (0, -0) {};
		\node [style=vertex] (1) at (0.5, 0.75) {};
		\node [style=vertex] (2) at (0, 0.75) {};
		\node [style=vertex] (3) at (0, 1.5) {};
		\node [style=vertex] (4) at (0, 2.25) {};
		\node [style=vertex,fill=blue] (5) at (0.75, 0.35) {};
		\node [style=vertex,fill=black!10] (5x) at (-0.5, 0.75) {};
	\end{pgfonlayer}
	\begin{pgfonlayer}{edgelayer}
	\draw[blue] (-0.25,-0.25) rectangle (1,2.5);
	\node (x) at (0.5,-0.55) {$R_2^0$};
		\draw[very thick] (1) to (0);
		\draw[very thick] (0) to (2);
		\draw[very thick] (2) to (3);
		\draw[very thick] (3) to (4);
		\draw[very thick,blue] (0) to (5);
		\draw[very thick,black!10,dashed] (0) to (5x);
	\end{pgfonlayer}
\end{tikzpicture}& 
\begin{tikzpicture}
	\begin{pgfonlayer}{nodelayer}
		\node [style=vertex] (0) at (0, -0) {};
		\node [style=vertex] (4) at (-0.5, 0.75) {};
		\node [style=vertex] (2) at (0, 1.5) {};
		\node [style=vertex] (5) at (0, 0.75) {};
		\node [style=vertex,fill=black!10] (1) at (-1, 1.5) {};
		\node [style=vertex,fill=blue] (1x) at (0, 2.25) {};
		\node [style=vertex] (3) at (0.5, 0.75) {};
	\end{pgfonlayer}
	\begin{pgfonlayer}{edgelayer}
	\draw[blue] (-0.2,-0.25) rectangle (1,2.5);
	\draw[red] (-0.3,-0.25) rectangle (-1.5,2.5);
	\node (x) at (-1,-0.55) {$L_2^5$};
	\node (x) at (0.5,-0.55) {$R_2^1$};
		\draw[very thick,black!10,dashed] (1) to (4);
		\draw[very thick, blue] (1x) to (2);
		\draw[very thick] (2) to (5);
		\draw[very thick] (4) to (0);
		\draw[very thick] (5) to (0);
		\draw[very thick] (3) to (0);
	\end{pgfonlayer}
\end{tikzpicture}&
\begin{tikzpicture}
	\begin{pgfonlayer}{nodelayer}
		\node [style=vertex] (0) at (0, -0) {};
		\node [style=vertex] (4) at (-0.5, 0.75) {};
		\node [style=vertex, fill=black!10] (2) at (-1, 2.25) {};
		\node [style=vertex, fill=blue] (2x) at (0, 1.5) {};
		\node [style=vertex] (5) at (0, 0.75) {};
		\node [style=vertex] (1) at (-1, 1.5) {};
		\node [style=vertex] (3) at (0.5, 0.75) {};
	\end{pgfonlayer}
	\begin{pgfonlayer}{edgelayer}
		\draw[blue] (-0.2,-0.25) rectangle (1,2.5);
	\draw[red] (-0.3,-0.25) rectangle (-1.5,2.5);
	\node (x) at (-1,-0.55) {$L_2^4$};
	\node (x) at (0.5,-0.55) {$R_2^2$};
		\draw[very thick] (1) to (4);
		\draw[very thick,black!10,dashed] (1) to (2);
		\draw[very thick, blue] (2x) to (5);
		\draw[very thick] (4) to (0);
		\draw[very thick] (5) to (0);
		\draw[very thick] (3) to (0);
	\end{pgfonlayer}
\end{tikzpicture}&
\begin{tikzpicture}
	\begin{pgfonlayer}{nodelayer}
		\node [style=vertex] (0) at (0, -0) {};
		\node [style=vertex] (4) at (-0.5, 0.75) {};
		\node [style=vertex] (2) at (-1, 2.25) {};
		\node [style=vertex, fill=black!10] (5) at (-0.5, 1.5) {};
		\node [style=vertex, fill=blue] (5x) at (0, 0.75) {};
		\node [style=vertex] (1) at (-1, 1.5) {};
		\node [style=vertex] (3) at (0.5, 0.75) {};
	\end{pgfonlayer}
	\begin{pgfonlayer}{edgelayer}
		\draw[blue] (-0.2,-0.25) rectangle (1,2.5);
	\draw[red] (-0.3,-0.25) rectangle (-1.5,2.5);
	\node (x) at (-1,-0.55) {$L_2^3$};
	\node (x) at (0.5,-0.55) {$R_2^3$};
		\draw[very thick] (1) to (4);
		\draw[very thick] (1) to (2);
		\draw[very thick] (4) to (0);
		\draw[very thick,black!10,dashed] (5) to (4);
		\draw[very thick,blue] (5x) to (0);
		\draw[very thick] (3) to (0);
	\end{pgfonlayer}
\end{tikzpicture}&
\begin{tikzpicture}
	\begin{pgfonlayer}{nodelayer}
		\node [style=vertex] (0) at (0, -0) {};
		\node [style=vertex] (4) at (-0.5, 0.75) {};
		\node [style=vertex] (2) at (-1, 2.25) {};
		\node [style=vertex] (5) at (-0.5, 1.5) {};
		\node [style=vertex] (1) at (-1, 1.5) {};
		\node [style=vertex,fill=black!10] (3) at (-0.5, 2.25) {};
		\node [style=vertex,fill=blue] (3x) at (0.5, 0.75) {};
	\end{pgfonlayer}
	\begin{pgfonlayer}{edgelayer}
		\draw[blue] (-0.2,-0.25) rectangle (1,2.5);
	\draw[red] (-0.3,-0.25) rectangle (-1.5,2.5);
	\node (x) at (-1,-0.55) {$L_2^2$};
	\node (x) at (0.5,-0.55) {$R_2^4$};
		\draw[very thick] (1) to (4);
		\draw[very thick] (1) to (2);
		\draw[very thick] (4) to (0);
		\draw[very thick] (5) to (4);
		\draw[very thick,black!10,dashed] (3) to (5);
		\draw[very thick,blue] (3x) to (0);
	\end{pgfonlayer}
\end{tikzpicture}& \\[-0.1cm]
$t_{10}$ & $t_9$ & $t_8$ & $t_7$ & $t_6$ & \\
\end{tabularx}
\caption{The form of a path in $\Gamma_{t_0\to t_{10}}$, where $t_0,t_{10}\in\T_5$; notice that $|L(t_5)|=4$.}
\end{figure}
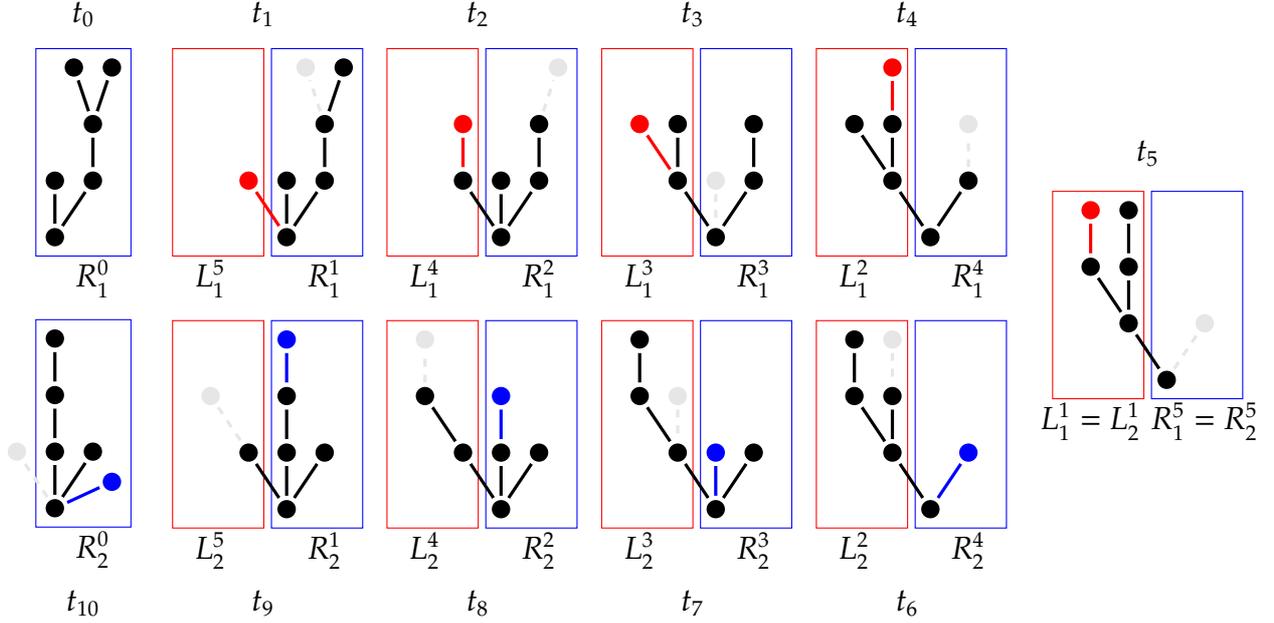

Before we prove a lower bound for the spectral gap of the leaf replanting Markov chain, it is useful to establish the following estimate:
\begin{lemma}\label{double count me}
	For all $t\in\TR_n$ and $i\in\{0,\ldots ,2n\}$ we have
	$$\sum_{x,y\in\TR_n}P_{x\to y}(\{\gamma\in\Gamma_{x\to y}\st \gamma(i)=t\})\leq 2(4r)^{n+1},$$
	where, if $\gamma=(t_i,t_{i+1})_{i=0}^{2n}$, we write $\gamma(i)$ to indicate $t_i$.
\end{lemma}
\begin{proof}

	As above, given $\gamma\in\Gamma_{x\to y}$, let us define sequences
	$R_i(\gamma)=(R_i^0,\ldots,R_i^n)$ and $L_i(\gamma)=(L_i^1,\ldots,L_i^n)$, for $i=1,2$.
	
	If $0<i\leq n$, we have $\gamma(i)=t$ if and only if $R_1^i=R(t)$ and $L_1^{n+1-i}=L(t)$; therefore, we have 
	
	$$\sum_{x,y\in\TR_n}P_{x\to y}(\{\gamma\in\Gamma_{x\to y}\st \gamma(i)=t\})$$
	$$=
	\sum_{x\in\TR_n}Q^x(\{(R_j)_{j=0}^n\in\Gamma^x\st R_{i}=R(t)\})
	\sum_{z\in\TR_{n-1}}Q^z(\{(L_j)_{j=1}^n\in\Gamma^z\st L_{n+1-i}=L(t)\})
	\sum_{\substack{y\in\TR_n\\F(x,y)=z}}1
	$$
Since we have chosen $F$ as in Lemma~\ref{small fibers}, the internal sum (having fixed $x$ and $z$) is at most $8r$; as for the other sums, we wish to show that, for any given $t\in\TR_k$ and $i\in\{0,\ldots,n\}$, we have
\begin{equation}\label{4r^{i-1}}\sum_{x\in\TR_n}Q^x(\{(R_j)_{j=0}^n\in\Gamma^x\st R_i=t\})\leq (4r)^i;\end{equation}
but, indeed,
$$\sum_{x\in\TR_n}Q^x(\{(R_j)_{j=0}^n\in\Gamma^x\st R_i=t\})=\sum_{\substack{R_0\in\TR_n,\ldots,R_{i-1}\in\TR_{n-i+1}\\R_i=t}}\prod_{j=0}^{i-1}f_{n-j}(R_j,R_{j+1})=
\prod_{j=0}^{i-1}\frac{|\TR_{n-j}|}{|\TR_{n-j-1}|}\leq (4r)^i,
$$
which we obtain by separating $R_{i-1},R_{i-2},\ldots,R_0$ from the sum, one after the other, and using the fact that $f_{n-i+1},\ldots,f_n$ satisfy requirement (iii) of Proposition~\ref{hierarchy}.

Using \eqref{4r^{i-1}} (where one needs to be weary of the fact that -- in order to keep notation consistent with previous definitions -- $R_1$ is indexed from 0 and $L_1$ is indexed from 1) we get
$$\sum_{x,y\in\TR_n}P_{x\to y}(\{\gamma\in\Gamma_{x\to y}\st \gamma(i)=t\})\leq 8r\cdot (4r)^i\cdot (4r)^{n-i}\leq 2 (4r)^{n+1}$$
for $0<i\leq n$; the same estimate is true for $i=0$, since we have
$$\sum_{y\in\TR_n}P_{t\to y}(\{\gamma\in\Gamma_{t\to y}\})=
	\sum_{z\in\TR_{n-1}}
	\sum_{\substack{y\in\TR_n\\F(x,y)=z}}1\leq 
	8r\cdot|\TR_{n-1}|\leq
	8r\cdot (4r)^n.
	$$

The case of $n<i\leq 2n$ is perfectly symmetric.
\end{proof}

All necessary notation is now in place to prove lower bounds for the spectral gap of both the leaf replanting and leaf translation Markov chains.

\begin{theorem}
If $\gamma_{\CLR}$ is the spectral gap of the leaf replanting Markov chain $\CLR$ on $\TR_n$, we have $\gamma_{\CLR}\geq C_r n^{-\frac{9}{2}}$ for an appropriate constant $C_r$ independent of $n$ .
\end{theorem}

\begin{proof}
	By the canonical paths method (see for example \cite[Section 13.4]{LevinPeres}), we have
	$$\frac{1}{\gamma_{\CLR}}\leq\max_{t,t'\in\TR_n:p_\CLR(t,t')>0}\frac{1}{\pi(t)p_{\CLR}(t,t')}\sum_{x,y\in\TR_n}\sum_{\substack{\gamma\in\Gamma_{x\to y}:\\(t,t')\in\gamma}}|\gamma|P_{x\to y}(\gamma)\pi(x)\pi(y),$$	
where, if $\gamma=(t_i,t_{i+1})_{i=0}^{N-1}$, we are writing $|\gamma|$ to mean the length $N$ of the sequence, and we say that $(t,t')\in\gamma$ if $t=t_i$ and $t'=t_{i+1}$ for some $i\in\{0,\ldots,N-1\}$.
	
	By the description of the leaf replanting Markov chain, we know that (assuming $t,t'$ differ by the replanting of a leaf) $p_{\CLR}(t,t')\geq \frac{1}{2rn^2}$; furthermore, every path $\gamma\in\Gamma_{x\to y}$ (for $x,y\in\TR_n$) has length exactly $2n$. Using the fact that $\pi$ is the uniform measure on $\TR_n$ and setting $(t,t')$ to be a pair of trees achieving the maximum above, one obtains
	$$\frac{1}{\gamma_{\CLR}}\leq\frac{2n\cdot 2rn^2}{|\TR_n|}\sum_{x,y\in\TR_n}\sum_{(t,t')\in\gamma\in\Gamma_{x\to y}}P_{x\to y}(\gamma).$$
	
	All that is left to do is to estimate the sum of all probabilities $P_{x\to y}(\gamma)$, where $\gamma$ is a path in some $\Gamma_{x\to y}$ involving the transition $(t,t')$. Notice that, if $(t,t')$ appears in $\gamma$, then the fact that for $0<i\leq n$ we have $|L(\gamma(i))|=i-1$ and for $n\leq i<2n$ we have $|L(\gamma(i))|=2n-1-i$ implies that either $t=\gamma(0)$, or $t=\gamma(|L(t)|+1)$, or $t=\gamma(2n-1-|L(t)|)$; we therefore have, if we set $S=\{0,|L(t)|+1,2n-1-|L(t)|\}$,
$$\frac{1}{\gamma_{\CLR}}\leq\frac{2n\cdot 2rn^2}{|\TR_n|}\sum_{\substack{i\in S}}\sum_{x,y\in\TR_n}
P_{x\to y}(\{\gamma\in\Gamma_{x\to y}\st\gamma(i)=t\});$$
by Lemma~\ref{double count me}, this yields
$$\frac{1}{\gamma_{\CLR}}\leq\frac{2n\cdot 2rn^2}{|\TR_n|}3\cdot 2(4r)^{n+1}\leq C_r\cdot n^{3+\frac{3}{2}}\leq C_r\cdot n^{\frac{9}{2}},$$
as wanted.
\end{proof}
 
The proof of a lower bound for the spectral gap of the leaf translation Markov chain is analogous, if a little more fiddly.

\begin{theorem}\label{leaf translation gap}
If $\CLT$ is the leaf translation Markov chain	on the state space of $r$-coloured plane trees with $n$ edges $\TR_n$ and $\gamma_{\CLT}$ is its spectral gap, we have
$$\gamma_\CLT \geq C_r n^{-\frac{9}{2}}$$
for some constant $C_r$ independent of $n$.
\end{theorem}

\begin{proof}
Suppose $t,t'\in\TR_n$ differ by the replanting and recolouring of a leaf (i.e.~are such that $p_\CLR(t,t')>0$); the leaf being replanted and recoloured may not be uniquely determined, but let $v$ be the leftmost leaf and $k$ be the minimum integer, given $v$, such that $t'=t^{v,k,c}$. If $k$ is greater than or equal to the number of the corner of $v$ in $t$, we construct a ``leaf translation path'' $\gamma(t,t')=(t_i,t_{i+1})_{i=0}^{N-1}$ from $t$ to $t'$ by setting $t_0=t$, $t_{i+1}=t_i^{v,\rightarrow}$ and choosing $N$ to be as small as possible and such that $t_{N-1}$ is $t'$ up to the recolouring of the replanted leaf; finally, we set $t_N=t_{N-1}^{v,c}$. Similarly, if the corner of $v$ in $t$ is indexed by a number strictly greater than $k$, we construct $\gamma(t,t')$ as a (minimal) sequence of leftward translations of $v$, followed by a recolouring.

Given $t_0,t_{2n}\in\TR_n$ and $\gamma=(t_i,t_{i+1})_{i=0}^{2n-1}\in\Gamma_{t_0\to t_{2n}}$, we can now define a ``leaf translation'' path $\gamma^\CLT$ by concatenating $\gamma_0,\ldots,\gamma_{2n-1}$, where $\gamma_i=\gamma(t_i,t_{i+1})$. We call $\Gamma^\CLT_{t_0\to t_{2n}}$ the set $\{\gamma^\CLT \st \gamma\in\Gamma_{t_0\to t_{2n}}\}$. Notice that $\gamma\mapsto \gamma^\CLT$ is a bijection between $\Gamma_{t_0\to t_{2n}}$ and $\Gamma^\CLT_{t_0\to t_{2n}}$, since the sequence $\gamma$ can be reconstructed from $\gamma^{\CLT}$ by setting $\gamma(0)=\gamma^\CLT(0)$ and $\gamma(i+1)=\gamma^\CLT(x)$, where 
$$x=\min\{j>i \st\gamma^\CLT(j)\mbox{ is of the form }\gamma^\CLT(j-1)^{v,c}\mbox{ with }c\notin\{\leftarrow,\rightarrow\}\}.$$
We can therefore define a probability measure on $\Gamma^\CLT_{t_0\to t_{2n}}$, which we still call $P_{t_0\to t_{2n}}$, by simply setting $P_{t_0\to t_{2n}}(\gamma^\CLT)=P_{t_0\to t_{2n}}(\gamma)$.

By the canonical paths method, we have
$$\frac{1}{\gamma_{\CLT}}\leq\max_{t,t'\in\TR_n:p_\CLT(t,t')>0}\frac{1}{\pi(t)p_{\CLT}(t,t')}\sum_{x,y\in\TR_n}\sum_{\substack{\gamma\in\Gamma_{x\to y}:\\(t,t')\in\gamma^\CLT}}|\gamma^\CLT|P_{x\to y}(\gamma)\pi(x)\pi(y).$$

Notice that for all $\gamma$ we have $|\gamma^\CLT|\leq 4n^2$ (by Remark~\ref{basic tree chains}) and, for all $t,t'\in\TR_n$ such that $p_\CLT(t,t')>0$, we have $p_\CLT(t,t')\geq \frac{1}{n(r+2)}$. We therefore have, if we also replace occurrences of $\pi(\cdot)$ with $|\TR_n|^{-1}$,
$$\frac{1}{\gamma_{\CLT}}\leq\frac{4n^3(r+2)}{|\TR_n|}\max_{(t,t^{v,c})}\sum_{x,y\in\TR_n}\sum_{\substack{\gamma\in\Gamma_{x\to y}:\\(t,t^{v,c})\in\gamma^\CLT}}P_{x\to y}(\gamma).$$

Given a leaf translation or recolouring $(t,t^{v,c})$, we now wish to determine which paths $\gamma=(t_i,t_{i+1})_{i=0}^{2n-1}\in\Gamma_{t_0\to t_{2n}}$ are such that $(t,t^{v,c})$ appears in $\gamma^\CLT$. 

If $c\notin\{\rightarrow,\leftarrow\}$, then $t^{v,c}=t_i$ for some $i$ which is almost univocally determined by $|L(t^{v,c})|$: we necessarily have $i=2n$ or $i=|L(t^{v,c})|+1$ or $i=2n-|L(t^{v,c})|-1$. We therefore find that
$$\sum_{x,y\in\TR_n}\sum_{\substack{\gamma\in\Gamma_{x\to y}:\\(t,t^{v,c})\in\gamma^\CLT}}P_{x\to y}(\gamma)\leq \sum_{x,y\in\TR_n}\sum_{\substack{i\in\{2n,|L(t)|+1,\\2n-1-|L(t)|\}}}P_{x\to y}(\{\gamma\in\Gamma_{x\to y}\st\gamma(i)=t\})\leq 6\cdot (4r)^{n+1}$$

If $c=\leftarrow$, consider the tree $\tau$ obtained by removing $v$ from $t^{v,c}$. We must have either $\tau=R(t_1)$, 
or $R(\tau)=R(t_i)$ and $L(\tau)=L(t_{i-1})$, with $i=|L(\tau)|+1$. Similarly for $c=\rightarrow$: either 
$\tau=R(t_{2n-1})$, or $R(\tau)=R(t_{i-1})$ and $L(\tau)=L(t_i)$, with $i=2n-1-|L(\tau)|$. The proof of Lemma~\ref{double count me} can be modified slightly to yield that, for any fixed $i\in\{1,\ldots,n\}$ and $a,b\in\bigcup_{j=0}^{n}\TR_j$, we have
$$\sum_{x,y\in\TR_n}P_{x\to y}(\{\gamma\in\Gamma_{x\to y}\st R(\gamma(i))=a, L(\gamma(i-1))=b\})\leq 2(4r)^{n+1}.$$

Applying both the original inequality from Lemma~\ref{double count me} and this variant yields that, when $c\in\{\rightarrow,\leftarrow\}$,
$$\sum_{x,y\in\TR_n}\sum_{\substack{\gamma\in\Gamma_{x\to y}:\\(t,t^{v,c})\in\gamma^\CLT}}P_{x\to y}(\gamma)\leq 6\cdot (4r)^{n+1}.$$

Finally, this entails
$$\frac{1}{\gamma_{\CLT}}\leq\frac{24n^3(r+2)}{|\TR_n|}(4r)^{n+1}\leq C_r\cdot n^{\frac{9}{2}}$$
for some constant $C_r$ independent of $n$.
\end{proof}

\begin{remark}
In order to obtain results for the flip chain $\F_n$, we are only interested in lower bounds for the spectral gap of the chain $\CLT$. It is not difficult, however, to obtain upper bounds proportional to $n^{-2}$ for the spectral gaps $\gamma_\CLT$ and $\gamma_\CLR$	of $\CLT$ and $\CLR$ by an argument even simpler than that of Section~\ref{section: upper bound}.

Indeed, let $H_n:\TR_n\to\mathbb{N}$ be the function giving the height of a tree, and consider the Dirichlet forms $\mathcal{E}_\CLT(n^{-\frac12}H_n,n^{-\frac12}H_n)$ and $\mathcal{E}_\CLR(n^{-\frac12}H_n,n^{-\frac12}H_n)$. Both can be bounded above by (a constant times) $n^{-2}$, using the fact that $H_n$ changes by at most 1 when a leaf replanting/translation/recolouring is performed, and that moreover, given $t\in\TR_n$, its height decreases with probability at most $\frac{1}{n}$ when taking a step of either chain (since one has to remove the ``top leaf'', which even needs to be unique). The bound on the spectral gap is then established thanks to the fact that the random variable $n^{-\frac12}H_n(t)$, where $t$ is a uniform random element of $\TR_n$, converges to a nontrivial random variable as $n\to\infty$, and in fact its variance converges to a positive constant (cf.~\cite[Section 3.1]{Ald91}).
\end{remark}

\section{A lower bound for the spectral gap of $\F^n$}\label{section: lower bound}

\subsection{Edge flips and the leaf translation Markov chain}\label{chain comparison}

We will now set up a comparison à la Diaconis--Saloff-Coste~\cite{DSC} between the Markov chain $\F^{n,\bullet}$ and a variant of the leaf translation Markov chain $\CLT$ on $\LT_n$; as per Theorem~\ref{Schaeffer}, we have an explicit bijection $\phi$ between the state space $\sQ_n^\bullet$ of $\F^{n,\bullet}$ and the set $\LT_n\times\{-1,1\}$.

A (reversible, irreducible and aperiodic) variant $\VLT$ of the leaf translation Markov chain can be defined on $\LT_n\times\{-1,1\}=\T^{(3)}_n\times \{-1,1\}$, where we consider the set of edge colours to be $\{+,-,=\}$ rather than $\{1,2,3\}$, as follows: conditionally on $\VLT_k=(t,\epsilon)$, where $t\in\LT_n$ and $\epsilon\in\{-1,1\}$, we set
\begin{itemize}
\item $\VLT_{k+1}=(t,-\epsilon)$ with probability $\frac{1}{n+1}$
\item with probability $\frac{n}{n+1}$, we select a random edge $(v,p(v))$ of $t$; if $v$ is not a leaf, we set $\VLT_{k+1}=(t,\epsilon)$; if $v$ is a leaf, we set $\VLT_{k+1}$ to be one of $(t^{v,\rightarrow},\epsilon)$, $(t^{v,\leftarrow},\epsilon)$, $(t^{v,+},\epsilon)$, $(t^{v,-},\epsilon)$, $(t^{v,=},\epsilon)$, each with probability $\frac{1}{5}$ given the choice of $v$.
\end{itemize}

From Theorem~\ref{leaf translation gap} we can deduce the following analogous estimate for the spectral gap of this chain.

\begin{cor}\label{leaf translation variant gap}
If $\widetilde{\gamma}$ is the spectral gap of the Markov chain $\VLT$ on the state space $\LT_n\times\{-1,1\}$ as defined above, we have $\widetilde{\gamma}\geq C n^{-\frac{9}{2}}$ for some constant $C$.
\end{cor}

\begin{proof}
Let $f:\LT_n\times\{-1,1\}\to\mathbb{R}$ be a function such that $\E_\pi(f)=0$, $\V_\pi (f)=1$ and $\widetilde{\gamma}=\mathcal{E}_\VLT(f,f)$, where $\pi$ is the uniform measure on $\LT_n\times\{-1,1\}$ and $\mathcal{E}_\VLT$ is the Dirichlet form for the Markov chain $\VLT$. Then
$$\widetilde{\gamma}=\frac12\sum_{\substack{t\in\LT_n\\v\mbox{ leaf of }t\\x\in\{\rightarrow,\leftarrow,+,-,=\}\\\epsilon\in\{-1,1\}}}(f(t,\epsilon)-f(t^{v,x},\epsilon))^2\frac{1}{2|\LT_n|}\frac{1}{5(n+1)}+\frac12
\sum_{\substack{t\in\LT_n\\\epsilon\in\{-1,1\}}}(f(t,\epsilon)-f(t,-\epsilon))^2\frac{1}{2|\LT_n|}\frac{1}{n+1}.$$

Consider now the maps $f_1,f_{-1}:\LT_n\to\mathbb{R}$ defined so that $f_\epsilon(t)=f(t,\epsilon)$ and the leaf translation Markov chain $\CLT$ on $\LT_n$. We can immediately identify the first of the two sums above as $$\frac{n}{2(n+1)}\left(\mathcal{E}_\CLT(f_1,f_1)+\mathcal{E}_\CLT(f_{-1},f_{-1})\right).$$ On the other hand, a lower bound for the second sum is given by $\frac{2}{n+1}\V(\E(f|\epsilon))$, where $\epsilon:\LT\times\{-1,1\}\to\{-1,1\}$ is the projection on the second component, since a simple application of the Cauchy-Schwarz inequality gives
$$\sum_{\substack{t\in\LT_n\\\epsilon\in\{-1,1\}}}(f(t,\epsilon)-f(t,-\epsilon))^2\frac{1}{2|\LT_n|}\frac{1}{n+1}=
\frac{1}{(n+1)4|\LT_n|^2}\left(\sum_{\substack{t\in\LT_n\\\epsilon\in\{-1,1\}}}(f(t,\epsilon)-f(t,-\epsilon))^2\right)\left(\sum_{\substack{t\in\LT_n\\\epsilon\in\{-1,1\}}}1^2\right)
$$
$$\geq\frac{1}{4(n+1)}\left(\frac{2}{|\LT_n|}\sum_{\substack{t\in\LT_n}}f_1(t)-\frac{2}{|\LT_n|}\sum_{\substack{t\in\LT_n}}f_{-1}(t)\right)^2=\frac{4}{n+1}\V(\E(f|\epsilon)).
$$

We thus have
$$\widetilde{\gamma}\geq \frac{n}{2(n+1)}\left(\mathcal{E}_\CLT(f_1,f_1)+\mathcal{E}_\CLT(f_{-1},f_{-1})\right)+\frac{2}{n+1}\V(\E(f|\epsilon))$$
$$
\geq\frac{n\gamma}{2(n+1)}(\V(f|\epsilon=1)+\V(f|\epsilon=-1))+\frac{2}{n+1}\V(\E(f|\epsilon))=
\frac{n\gamma}{n+1}\E(\V(f|\epsilon))+\frac{2}{n+1}\V(\E(f|\epsilon)),
$$
where $\gamma$ is the spectral gap of the leaf translation Markov chain $\CLT$ on $\LT_n$. Finally, using Proposition~\ref{leaf translation gap}, the variance decomposition formula and the fact that $\V_\pi(f)=1$, we obtain that, for some constants $C'$ and $C$
$$\tilde{\gamma}\geq\frac{n}{(n+1)n^{\frac{9}{2}}}C'(\E(\V_\pi(f|\epsilon))+\V_\pi(\E_\pi(f|\epsilon)))\geq \frac{C}{n^{\frac{9}{2}}}.$$
\end{proof}

We are now ready to set up a comparison between the chains $\VLT$ and $\F^{n,\bullet}$. In order to do this, we will devote the next section to explicitly building sequences of quadrangulation edge flips that turn $\phi(t,\epsilon)$ into $\phi(t',\epsilon')$, where $(t,\epsilon)$ and $(t',\epsilon')$ are elements of $\LT_n\times\{-1,1\}$ that are adjacent according to the graph of the Markov chain $\VLT$.

In particular, for each $t,v,x,\epsilon$ such that $t\in\LT_n$, $v$ is a leaf of $t$, $x\in\{\rightarrow,\leftarrow,+,=,-\}$ and $\epsilon=\pm1$, we shall build a sequence of quadrangulation edge flips $P_\epsilon(t,t^{v,x})=(q_i,e_i,s_i)_{i=1}^N$, such that
\begin{itemize}
\item $q_i\in\sQ_n^\bullet$, $e_i\in E(q_i)$, $s_i\in\{+,-\}$;
\item $q_1=\phi(t,\epsilon)$;
\item $q_{i+1}=q_i^{e_i,s_i}$, for $i=1,2,\ldots,N$;
\item $q_{N+1}=\phi(t^{v,x},\epsilon)$.
\end{itemize}

Notice that (as we remarked in Section~\ref{section: flips}) we can naturally identify vertices of $q_i$ with vertices of $q_{i+1}$ and edges of $q_i$ with edges of $q_{i+1}$ by building $q_{i+1}$ via the procedure described for flipping $e_i$. We will therefore often define edges $e_i,\ldots,e_N$ as edges of $q_1$, since edges in $E(q_1)$ have a natural interpretation in $E(q_2),\ldots,E(q_N)$.

Similarly, we will also build sequences $P(t)=(q_i,e_i,s_i)_{i=1}^N$ such that $q_1=\phi(t,1)$, $q_{N+1}=\phi(t,-1)$ and $q_{i+1}=q_i^{e_i,s_i}$.

Having constructed these in an appropriate way, a comparison of the Markov chains $\F^{n,\bullet}$ and $\VLT$ (cf.~\cite{DSC}) will yield Theorem~\ref{main theorem}, provided that we can bound the maximum length of a flip sequence with (a constant times) $n$ and show that each triple $(q,e,s)$ (where $q\in\sQ^\bullet_n$, $e\in E(q)$, $s=\pm1$) only appears in at most a constant number of sequences $P(t)$ and $P_\epsilon(t,t^{v,x})$, independent of $n$. Constructing our flip sequences and proving such bounds will be the aim of the next three subsections; Section~\ref{section: final comparison} will conclude by deriving Theorem~\ref{main theorem}. 

\subsection{The sequence $P(t)$}\label{section: rerooting}

As a matter of fact, we have already discussed the sequence $P(t)$ in Lemma~\ref{pointed irreducibility}.

From the Schaeffer construction within Section~\ref{section: Schaeffer} one can immediately see that $\phi(t,-1)$ can be obtained from $\phi(t,1)$ by simply giving the root edge the opposite orientation. If the root edge is not a double edge, this can be achieved via flips by just flipping it three times in the same direction, so we can set $P(t)=(q_i,e_i,s_i)_{i=1}^3$, with $q_1=\phi(t,1)$, $s_1=s_2=s_3=+$ and $e_i$ being the root edge of $q_i$, for $i=1,2,3$.

If the root edge of $\phi(t,1)$ is a double edge within a degenerate face, one need only perform an extra flip on one of the boundary edges of the degenerate face before and after flipping the root edge three times. Setting $q_1=\phi(t,1)$ and assuming $e'$ is the edge before the root edge in the clockwise contour of its degenerate face, we set $P(t)=(q_i,e_i,s_i)_{i=1}^5$, with $e_1=e'$, $s_1=s_2=s_3=s_4=+$, $e_i$ is the root edge of $q_i$ for $i=2,3,4$, and $e_5=e'$, $s_5=-$ (see Figure~\ref{fig:root reversal}).

\begin{lemma}\label{root reversal estimates}
For all $t\in\LT_n$, we have $|P(t)|\leq 5$. Moreover, given a triple $(q,e,s)$ where $q\in\sQ_n^\bullet$, $e\in E(q)$, $s\in\{+,-\}$, we have
$$\left|\{t\st(q,e,s)\mbox{ appears in }P(t)\}\right|\leq 9.$$
\end{lemma}

\begin{proof}
	The first part of the statement is clear by definition.
	
	As for the second part, if $t$ is a tree such that $(q,e,s)$ appears in $P(t)$ and the root of $\phi(t,1)$ is not a double edge, then $\phi(t,-1)$ is obtained by flipping the root edge of $q$ one, two or three times, so there are at most three possibilities for $t$.
	
	If $(q,e,s)$ appears in $P(t)$, where $e$ is the root edge of $q$ and the root edge of $\phi(t,1)$ is a double edge, then one or more among $q$, $q^{e,-}$ and $(q^{e,-})^{e,-}$ have one or both endpoints of the root edge of degree $2$; setting $e'$ to be such that $e,e'$ share an endpoint of degree 2 in $q'\in\{q,q^{e,-},(q^{e,-})^{e,-}\}$, $\phi(t,1)$ must be of the form $q'^{e',-}$, and therefore $t$ must be one of at most 6 possible labelled trees.
	
	If $e$ is \emph{not} the root edge of $q$ and $s=-$, then $\phi(t,-1)=q^{e,s}$; if $s=+$, $\phi(t,1)=q$.
\end{proof}

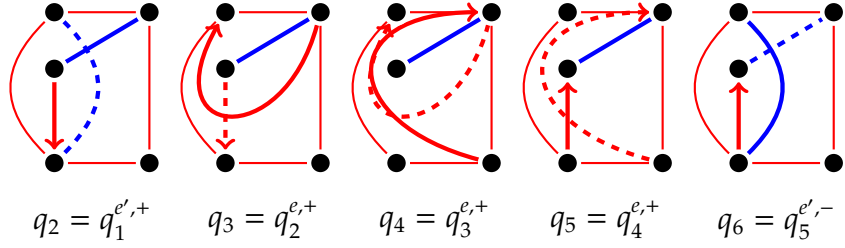
\begin{figure}\centering
\begin{tikzpicture}
	\begin{pgfonlayer}{nodelayer}
		\node [style=vertex] (0) at (-5, -1) {};
		\node [style=vertex] (1) at (-5, 0.25) {};
		\node [style=vertex] (2) at (-5, 1) {};
		\node [style=vertex] (3) at (-3.75, 1) {};
		\node [style=vertex] (4) at (-3.75, -1) {};
		\node [style=vertex] (5) at (-2.75, -1) {};
		\node [style=vertex] (6) at (-1.5, 1) {};
		\node [style=vertex] (7) at (-1.5, -1) {};
		\node [style=vertex] (8) at (-2.75, 0.25) {};
		\node [style=vertex] (9) at (-2.75, 1) {};
		\node [style=vertex] (10) at (-0.5, 0.25) {};
		\node [style=vertex] (11) at (-0.5, -1) {};
		\node [style=vertex] (12) at (0.75, 1) {};
		\node [style=vertex] (13) at (0.75, -1) {};
		\node [style=vertex] (14) at (-0.5, 1) {};
		\node [style=vertex] (15) at (1.75, 0.25) {};
		\node [style=vertex] (16) at (3, -1) {};
		\node [style=vertex] (17) at (3, 1) {};
		\node [style=vertex] (18) at (1.75, 1) {};
		\node [style=vertex] (19) at (1.75, -1) {};
		\node [style=vertex] (20) at (5.25, 1) {};
		\node [style=vertex] (21) at (4, 1) {};
		\node [style=vertex] (22) at (4, -1) {};
		\node [style=vertex] (23) at (4, 0.25) {};
		\node [style=vertex] (24) at (5.25, -1) {};
		
		\node (25) at (-4.5, -1.75) {$q_2=q_1^{e',+}$};
		\node (26) at (-2.25, -1.75) {$q_3=q_2^{e,+}$};
		\node (27) at (0, -1.75) {$q_4=q_3^{e,+}$};
		\node (28) at (2.25, -1.75) {$q_5=q_4^{e,+}$};
		\node (29) at (4.5, -1.75) {$q_6=q_5^{e',-}$};
	\end{pgfonlayer}
	\begin{pgfonlayer}{edgelayer}
		\draw [red, ultra thick, ->] (1) to (0) ;
		\draw [ultra thick, blue, dashed, bend left=45, looseness=1.25] (2) to (0);
		\draw [style=map] (2) to (3);
		\draw [style=map] (3) to (4);
		\draw [style=map] (4) to (0);
		\draw [style=map, bend right=45, looseness=1.25] (2) to (0);
		\draw [ultra thick, blue] (1) to (3); 
		\draw [ultra thick, red, dashed,->] (8) to (5);
		\draw [ultra thick, red, <-, in=-105, out=-120, looseness=3.50] (9) to (6); 
		\draw [style=map] (6) to (7);
		\draw [style=map] (7) to (5);
		\draw [style=map, bend right=45, looseness=1.25] (9) to (5);
		\draw [ultra thick, blue] (8) to (6); 
		\draw [style=map] (9) to (6);
		\draw [ultra thick, red, dashed, in=-105, out=-120, looseness=3.50, <-] (14) to (12);
		\draw [red, ultra thick, <-, in=165, out=180, looseness=2.50] (12) to (13); 
		\draw [style=map] (13) to (11);
		\draw [style=map, bend right=45, looseness=1.25] (14) to (11);
		\draw [ultra thick, blue] (10) to (12);
		\draw [style=map] (14) to (12);
		\draw [style=map](12) to (13);
		\draw [ultra thick, red, dashed, in=165, out=180, looseness=2.50, <-] (17) to (16);
		\draw [style=map] (16) to (19);
		\draw [style=map, bend right=45, looseness=1.25] (18) to (19);
		\draw [blue, ultra thick] (15) to (17);
		\draw [style=map] (18) to (17);
		\draw [style=map] (17) to (16);
		\draw [ultra thick, red, ->] (19) to (15); 
		\draw [style=map] (24) to (22);
		\draw [style=map, bend right=45, looseness=1.25] (21) to (22);
		\draw [ultra thick, blue, dashed] (23) to (20);
		\draw [style=map] (21) to (20);
		\draw [style=map] (20) to (24);
		\draw [ultra thick, red, ->] (22) to (23); 
		\draw [ultra thick, blue, bend right=45, looseness=1.25] (22) to (21); 
	\end{pgfonlayer}
\end{tikzpicture}
\caption{\label{fig:root reversal}A sequence of flips $P(t)=(q_i,e_i,s_i)_{i=1}^5$, where the root edge of $q_1=\phi(t,1)$ is a double edge within a degenerate face. Notice that either the sequence $(q_2,e_2,s_2),(q_3,e_3,s_3),(q_4,e_4,s_4)$ 
is an example of $P(t')$, for some $t'$ such that the root of $\phi(t',1)$ is not a double edge within a degenerate face.}
\end{figure}

\subsection{The colour change sequences $P_\epsilon(t,t^{v,-})$, $P_\epsilon(t,t^{v,=})$, $P_\epsilon(t,t^{v,+})$}\label{section: colour change}

This section will be devoted to constructing the sequences $P_\epsilon(t, t^{v,c})$, where $t\in\LT_n$, $\epsilon\in\{-1,1\}$, $v$ is a leaf of $t$ and $c\in\{-,=,+\}$, i.e.~sequences of quadrangulation edge flips whose aim is to achieve a ``colour change'', or equivalently a ``label change'', from the leaf translation Markov chain $\VLT$ on $\LT_n\times\{-1,1\}$.

Given a leaf label change $(t, t^{v,c})$, where $c\in \{=,+,-\}$, we need to construct a sequence $P_\epsilon(t, t^{v,c})=(q_i,e_i,s_i)_{i=1,\ldots,N}$ such that $q_1=\phi(t,\epsilon)$, $q_{i+1}=q_i^{e_i,s_i}$ and $q_{N+1}=\phi(t^{v,c},\epsilon)$. Our aim will then consist in estimating the maximum length $N$ of such sequences in terms of $n$, as well as the number of quadruples $(t,v,c,\epsilon)$, where $t\in\LT_n$, $v$ is a leaf in $t$, $c\in\{=,+,-\}$ and $\epsilon\in\{-1,1\}$, such that a fixed triple $(q, e, s)$ appears in the sequence $P_\epsilon(t, t^{v,c})$.

We shall describe explicitly all sequences of the form $P_\epsilon(t, t^{v,+})$, where $l(v)=l(p(v))$ in $t$, and $P_\epsilon(t, t^{v,-})$, where $l(v)-l(p(v))=1$ in $t$, that ``is colour changes'' from $=$ to $+$ and from $+$ to $-$. All other sequences will be built from these in the natural way, by concatenating them and/or running them backwards. If $t^{v,x}=t$, we set $P_\epsilon(t,t^{v,x})$ to be empty.

Let us first consider the case of a colour change from $=$ to $+$, which is easily dealt with.

\begin{lemma}\label{=->+ construction}
Given $(t,\epsilon)\in\LT_n\times\{-1,1\}$ and a leaf $v$ of $t$ such that $l(v)=l(p(v))$, the quadrangulations $q=\phi(t,\epsilon)$ and $q'=\phi(t^{v,+},\epsilon)$ differ by an edge flip.\end{lemma}

\begin{proof}
Consider all corners of $t$ but the one corner around the leaf $v$; their target corners as determined by the Schaeffer bijection are unaffected by increasing the label of $v$, by definition. In particular, the two corners immediately before and after the corner of $v$ in the contour, which are corners of $p(v)$, share the same target corner before and after the label change, making the quadrangulation face which corresponds to the tree edge $(v,p(v))$ a degenerate face. The only effect of the label increase is that of changing the target of the $v$ corner to the appropriate corner of $p(v)$, i.e.~flipping the double edge $e$ within the aforementioned degenerate face of $\phi(t,\epsilon)$. Also notice that the edge issued from the root corner of $t$ (which cannot be the corner of the leaf $v$) is unaffected both by the label change and by the flip, so we do have $\phi(t^{v,+},\epsilon)=\phi(t,\epsilon)^{e,+}$.
\end{proof}

It is therefore natural, when $l(v)=l(p(v))$, to set $P_\epsilon(t,t^{v,+})$ to $(q,e,+)$, where $e$ is the double edge of $\phi(t,\epsilon)$ incident to the vertex $v$. Similarly, we will set $P_\epsilon(t^{v,+},(t^{v,+})^{v,=})=P_\epsilon(t^{v,+},t)=(q^{e,+},e,-)$, thus covering all the cases of a colour change from $+$ to $=$ and vice-versa with flip sequences of length~1.

The construction of sequences of the type $P_\epsilon(t,t^{v,-})$ is less immediate. First, let us give a ``static'' description of how the label change affects the corresponding quadrangulation. Recall the mapping $\phi$ sending $(t,\epsilon)\in \LT_n \times \{-1,1\}$ to $q\in \sQ_n^\bullet$ as described in Section~\ref{section: Schaeffer}. The quadrangulation $q$ is constructed via the map $\phi$ by considering each corner $c$ of $t$ and drawing an edge from $c$ to another corner which we refer to as the target corner of $c$, henceforth denoted by $t(c)$. Recall also that $\delta$ denotes the distinguished vertex of $q$. 

\begin{lemma}\label{+->- static}
Consider a pair $(t,\epsilon)\in\LT_n\times\{-1,1\}$ and a leaf $v$ of $t$ such that $l(v)=l(p(v))+1$; let $c$ be the corner of $p(v)$ right after the corner of $v$ in the clockwise contour of $t$. Suppose the target corner $t(c)$ of $c$ is adjacent to a vertex $w\in V(t)$ and consider the quadrangulation edges $e$ and $e'$ drawn by the Schaeffer correspondence between $c$ and $t(c)$, and between $t(c)$ and $t(t(c))$, respectively. Let $e_1,\ldots,e_k$ be the edges incident to $w$ lying strictly between $e$ and $e'$, in clockwise order around $w$. If $t(c)$ is instead the corner around $\delta$, let $e_1,\ldots,e_k$ be all quadrangulation edges incident to $\delta$. 

The quadrangulation $q'=\phi(t^{v,-},\epsilon)$ can be obtained from $q=\phi(t,\epsilon)$ by ``rerouting'' the edges $e_1,\ldots,e_k$ to $v$ -- which is done by erasing their intersection with a suitably small neighbourhood of $w$ (or $\delta$) and replacing it with paths to $v$ drawn in the natural planar way
 -- then replacing the quadrangulation edge issued from $v$ with one joining $v$ to $t(t(c))$ (or to $\delta$, if $t(c)$ is the corner around $\delta$): see Figure \ref{fig:+->- static}. Notice that, if the root edge of $q=\phi(t,\epsilon)$ is one of $e_1,\ldots,e_k$, then it is ``rerouted'' to $v$ and maintains the same orientation in $q'$ (and, if it is not, then it is ``preserved''). 
\end{lemma}

\begin{figure}\centering
\begin{tikzpicture}[scale=1.5]
	\useasboundingbox (-3.8,-1.3) rectangle (0.3,3);
	\begin{pgfonlayer}{nodelayer}
		\node [style=vertex, label=left:0] (0) at (-2.25, -1) {};
		\node [style=vertex, label=left:1] (1) at (-2.25, -0) {};
		\node [style=vertex, label=2] (2) at (-3, 1) {};
		\node [style=vertex, label=left:1] (3) at (-2.25, 1) {};
		\node [style=vertex, label=left:2, fill=red, draw=pink, label=above:$v$] (4) at (-2, 1.75) {};
		\node [style=vertex, label=below:0, label=right:$w$] (5) at (-1.25, -0) {};
		\node [style=vertex, label=$\delta$, fill=red] (6) at (0, 1) {};
		\node (14) at (-1.75, 0.5) {\contour{white}{$e$}};
		\node (15) at (-0.68, 0.5) {\contour{white}{$e'$}};
		\node (16) at (-1.5, 1.5) {\contour{white}{$e_1$}};
		\node (17) at (-1.45, 1.9) {\contour{white}{$e_2$}};
		\node (18) at (-1.23, 2.25) {\contour{white}{$e_3$}};
	\end{pgfonlayer}
	\begin{pgfonlayer}{edgelayer}
	\fill[red!10] (6) circle (8pt);
		\draw [style=tree] (0) to (1);
		\draw [style=tree] (1) to (2);
		\draw [style=tree] (1) to (3);
		\draw [style=tree] (3) to (4);
		\draw [style=tree] (1) to (5);
		\draw [style=map, very thick, ->, bend right=60, looseness=1.50] (0) to (6);
		\draw [style=map, blue, very thick, in=75, out=145, looseness=11] (1) to (5);
		\draw [style=map, ultra thick] (5) to (6);
		\draw [style=map, very thick, blue, in=90, out=105, looseness=8.25] (1) to (5);
		\draw [style=map, bend left=15, looseness=1.00] (2) to (1);
		\draw [style=map, very thick, blue, in=90, out=105, looseness=3.75] (3) to (5);
		\draw [style=map, bend left=60, looseness=1.25] (4) to (3);
		\draw [style=map, ultra thick] (3) to (5);
		\draw [style=map, bend left=60, looseness=1.25] (1) to (0);
	\end{pgfonlayer}
\end{tikzpicture}
\begin{tikzpicture}[scale=1.5]
	\useasboundingbox (-3.8,-1.3) rectangle (0.3,3);
	\begin{pgfonlayer}{nodelayer}
		\node [style=vertex, label=left:0] (0) at (-2.25, -1) {};
		\node [style=vertex, label=left:1] (1) at (-2.25, -0) {};
		\node [style=vertex, label=2] (2) at (-3, 1) {};
		\node [style=vertex, label=left:1] (3) at (-2.25, 1) {};
		\node [style=vertex, label=above left:0, fill=red, draw=pink, label=above right:$v$] (4) at (-2, 1.75) {};
		\node [style=vertex, label=below:0, label=right:$w$] (5) at (-1.25, -0) {};
		\node [style=vertex, label=$\delta$, fill=red] (6) at (0, 1) {};
		\node (14) at (-1.75, 0.5) {\contour{white}{$e$}};
		\node (15) at (-0.68, 0.5) {\contour{white}{$e'$}};
	\end{pgfonlayer}
	\begin{pgfonlayer}{edgelayer}
	\fill[red!10] (6) circle (8pt);
		\draw [style=tree] (0) to (1);
		\draw [style=tree] (1) to (2);
		\draw [style=tree] (1) to (3);
		\draw [style=tree] (3) to (4);
		\draw [style=tree] (1) to (5);
		\draw [style=map, very thick, ->, bend right=60, looseness=1.50] (0) to (6);
		\draw [style=map, blue!10, very thick, dashed, in=75, out=145, looseness=11] (1) to (5); %
			\draw [style=map, very thick, blue!10, dashed, in=90, out=105, looseness=3.75] (3) to (5); %
			\draw [style=map, very thick, blue!10, dashed, in=90, out=105, looseness=8.25] (1) to (5); %
		\draw [style=map, blue, very thick, in=160, out=145, looseness=3] (1) to (4);
		\draw [style=map, ultra thick] (5) to (6);
		\draw [style=map, very thick, blue, in=180, out=105, looseness=1.5] (1) to (4);
		\draw [style=map, bend left=15, looseness=1.00] (2) to (1);
		\draw [style=map, very thick, blue, in=195, out=105, looseness=1] (3) to (4);
		\draw [style=map, very thick, dashed, red!10, bend left=60, looseness=1.25] (4) to (3);
		\draw [style=map] (4) to (6);
		\draw [style=map, ultra thick] (3) to (5);
		\draw [style=map, bend left=60, looseness=1.25] (1) to (0);
	\end{pgfonlayer}
\end{tikzpicture}
\caption{\label{fig:+->- static}The quadrangulations $\phi(t,1)$ and $\phi(t^{v,-},1)$ for a tree $t\in\LT_5$ such that $l(v)=l(p(v))+1$.}
\end{figure}
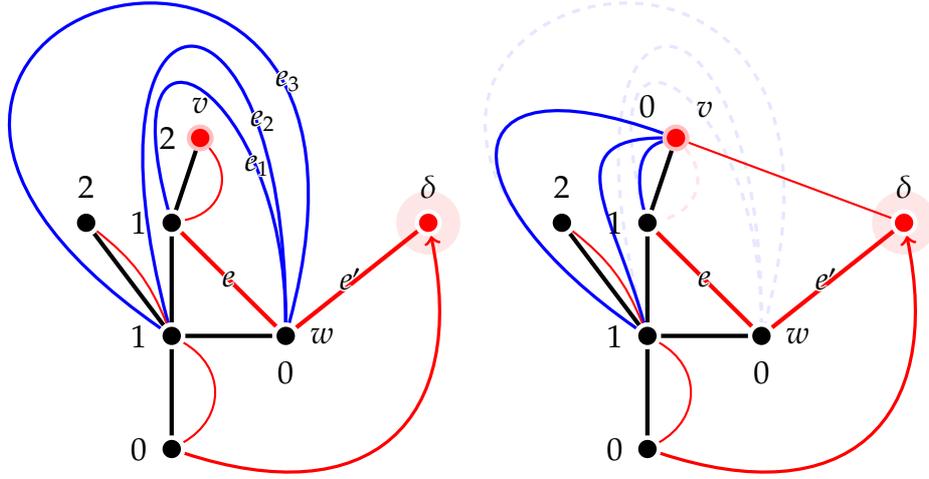

\begin{proof}
Let $c_1,\ldots,c_{2n}$ be the corners in the clockwise contour of $t$, and suppose $c_l$ is the corner around the leaf $v$. Let us first suppose that $l(c_{l+1})$ is not the minimal label in the tree (i.e. that its target is not the corner around $\delta$). Notice that decreasing the label of $v$ (by 2) does not affect any edges drawn by the Schaeffer correspondence other than
\begin{itemize}
	\item the edge drawn from $c_l$ to $c_{l+1}$, which is replaced by an edge between $c_l$ and $t(t(c_{l+1}))$, that is the first corner labelled $l(p(v))-2$ in the clockwise contour after $c_l$;
	\item all edges drawn from $c_i$ to $t(c_i)$, where $t(c_i)=t(c_{l+1})$ and $c_i$ \emph{does not} lie between $c_l$ and $t(c_{l+1})$, since the target of $c_i$ becomes $c_l$ if the label of $c_l$ is decreased by 2; indeed, those are replaced by edges joining $c_i$ to $c_l$.
\end{itemize}
The fact that all other targets remain the same should be clear: if only the label of $c_l$ is changed, the target of a corner $c\neq c_l$ may change only by becoming $c_l$ or by no longer being $c_l$. Since $c_l$ is not the target of any other corner in $t$, the edges affected are those for which $c_l$ lies between their origin corner and their target corner, labelled $l(c_{l+1})-1$, in the clockwise contour, as described above.

All that is left to show is that those edges are $e_1,\ldots,e_k$; indeed they are edges whose target is a corner of $w$ and whose origin corner comes after $t(c_{l+1})$ and before $c_l$ (or, equivalently, strictly before $c_{l+1}$), in the (cyclic) clockwise contour; that is, they lie strictly between $e=(c_{l+1},t(c_{l+1}))$ and $e'=(t(c_{l+1}),t(t(c_{l+1})))$. 

If $t(c_{l+1})$ is the corner around $\delta$ in $t$, then what happens when the label of $v$ is changed is even simpler, since $v$ becomes the unique vertex with minimal label: all corners whose target in $t$ is the corner around $\delta$ change their target to $c_l$ in $t^{v,-}$, including $c_{l-1}$ and $c_{l+1}$, while the edge $(c_l,c_{l+1})$ is replaced by one joining $v$ to $\delta$, which will have degree 1 in $\phi(t^{v,-},\epsilon)$.
\end{proof}

Before we give an actual description of a sequence of quadrangulation flips that achieves exactly the changes described by Lemma~\ref{+->- static}, we shall construct a sequence of flips that will be useful in what follows and whose only aim is to change the root edge of a quadrangulation by ``exchanging'' two edges.

\begin{lemma}[Rerooting]\label{root rotation}Consider a quadrangulation $q\in\sQ_n^\bullet$; let $e=(v,w)$ be its root edge (with either $v$ or $w$ being the origin) and let $\eta$ be the edge after $e$ in clockwise order around $v$, which we suppose distinct from $e$. Let $q'$ be the same quadrangulation (with the same distinguished vertex), rerooted in $\eta$ and with $v$ as the origin if and only if $v$ was the origin in $q$. Define the sequence of quadrangulation flips $P(q,q')=(q_i,e_i,s_i)_{i=1}^5$ so that $q_{i+1}=q_i^{e_i,s_i}$, $q_1=q$, $e_1=e_3=e_5=\eta$, $e_2=e_4=e$, $s_1=s_2=s_3=+$, $s_4=s_5=-$ (see Figure~\ref{fig:root rotation}). Then $q_6:=q_5^{e_5,s_5}=q'$.
\end{lemma}

\begin{proof}
	Consider the union of the faces of $q$ that are adjacent to $e$ or to $\eta$; unless either $e$ or $\eta$ is a double edge within a degenerate face, this is a (generalised) octagon, in the sense that its boundary has an inner contour with exactly 8 corners, which we can cyclically number as $c_0,\ldots,c_7$. We can suppose $\eta$ joins $c_0$ to $c_5$ and $e$ joins $c_0$ to $c_3$; it is then immediate to verify, as in Figure~\ref{fig:root rotation}, that the given sequence of flips ultimately results in $\eta$ joining $c_0$ to $c_3$ and $e$ joining $c_0$ to $c_5$.
	
	An analogous check can be performed for the case where $e$ or $\eta$ is a double edge, where one deals with a hexagon rather than an octagon (lower part of Figure~\ref{fig:root rotation}).
\end{proof}

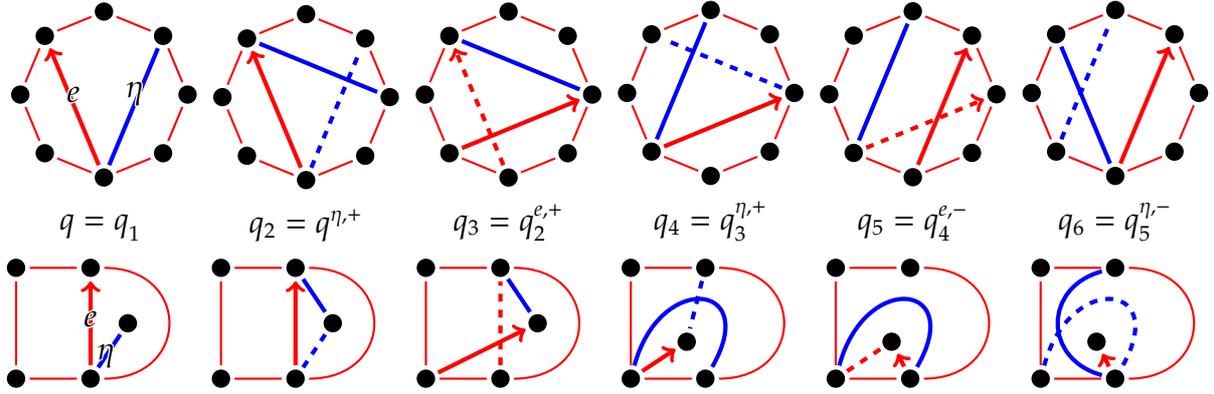
\begin{figure}\centering
\begin{tikzpicture}
\foreach \a in {0,1,...,7} {
\node[style=vertex] (\a) at (\a*45:1.1cm) {};}
\draw[style=map] (0) to (1) to (2) to (3) to (4) to (5) to (6) to (7) to (0);
\draw[ultra thick, red, ->] (6) to (3);
\draw[ultra thick, blue] (6) to (1);
\node (e1) at (-0.4cm,0) {\contour{white}{$e$}};
\node (e2) at (0.4cm,0) {\contour{white}{$\eta$}};
\node (d) at (0,-1.7cm) {$q=q_1^{\phantom{e_1}}$};
\end{tikzpicture}
\begin{tikzpicture}
\foreach \a in {0,1,...,7} {
\node[style=vertex] (\a) at (\a*45:1.1cm) {};}
\draw[style=map] (0) to (1) to (2) to (3) to (4) to (5) to (6) to (7) to (0);
\draw[ultra thick, red, ->] (6) to (3);
\draw[ultra thick, blue, dashed] (6) to (1);
\draw[ultra thick, blue] (0) to (3);
\node (d) at (0,-1.7cm) {$q_2=q^{\eta,+}$};
\end{tikzpicture}
\begin{tikzpicture}
\foreach \a in {0,1,...,7} {
\node[style=vertex] (\a) at (\a*45:1.1cm) {};}
\draw[style=map] (0) to (1) to (2) to (3) to (4) to (5) to (6) to (7) to (0);
\draw[ultra thick, red, ->, dashed] (6) to (3);
\draw[ultra thick, red, ->] (5) to (0);
\draw[ultra thick, blue] (0) to (3);
\node (d) at (0,-1.7cm) {$q_3=q_2^{e,+}$};
\end{tikzpicture}
\begin{tikzpicture}
\foreach \a in {0,1,...,7} {
\node[style=vertex] (\a) at (\a*45:1.1cm) {};}
\draw[style=map] (0) to (1) to (2) to (3) to (4) to (5) to (6) to (7) to (0);
\draw[ultra thick, red, ->] (5) to (0);
\draw[ultra thick, blue, dashed] (0) to (3);
\draw[ultra thick, blue] (5) to (2);
\node (d) at (0,-1.7cm) {$q_4=q_3^{\eta,+}$};
\end{tikzpicture}
\begin{tikzpicture}
\foreach \a in {0,1,...,7} {
\node[style=vertex] (\a) at (\a*45:1.1cm) {};}
\draw[style=map] (0) to (1) to (2) to (3) to (4) to (5) to (6) to (7) to (0);
\draw[ultra thick, red, ->, dashed] (5) to (0);
\draw[ultra thick, red, ->] (6) to (1);
\draw[ultra thick, blue] (5) to (2);
\node (d) at (0,-1.7cm) {$q_5=q_4^{e,-}$};
\end{tikzpicture}
\begin{tikzpicture}
\foreach \a in {0,1,...,7} {
\node[style=vertex] (\a) at (\a*45:1.1cm) {};}
\draw[style=map] (0) to (1) to (2) to (3) to (4) to (5) to (6) to (7) to (0);
\draw[ultra thick, red, ->] (6) to (1);
\draw[ultra thick, blue, dashed] (5) to (2);
\draw[ultra thick, blue] (6) to (3);
\node (d) at (0,-1.7cm) {$q_6=q_5^{\eta,-}$};
\end{tikzpicture}
\begin{tikzpicture}[scale=.98]
	\begin{pgfonlayer}{nodelayer}
		\node [style=vertex] (0) at (-8, -0) {};
		\node [style=vertex] (1) at (-8, 1.5) {};
		\node [style=vertex] (2) at (-7, -0) {};
		\node [style=vertex] (3) at (-7, 1.5) {};
		\node [style=vertex] (4) at (-6.5, 0.75) {};
		\node [style=vertex] (5) at (-5.25, -0) {};
		\node [style=vertex] (6) at (-4.25, 1.5) {};
		\node [style=vertex] (7) at (-5.25, 1.5) {};
		\node [style=vertex] (8) at (-4.25, -0) {};
		\node [style=vertex] (9) at (-3.75, 0.75) {};
		\node [style=vertex] (10) at (-1.5, -0) {};
		\node [style=vertex] (11) at (-1, 0.75) {};
		\node [style=vertex] (12) at (-1.5, 1.5) {};
		\node [style=vertex] (13) at (-2.5, -0) {};
		\node [style=vertex] (14) at (-2.5, 1.5) {};
		\node [style=vertex] (15) at (1.25, 1.5) {};
		\node [style=vertex] (16) at (0.25, -0) {};
		\node [style=vertex] (17) at (1.25, -0) {};
		\node [style=vertex] (18) at (0.25, 1.5) {};
		\node [style=vertex] (19) at (1, 0.5) {};
		\node [style=vertex] (20) at (3, -0) {};
		\node [style=vertex] (21) at (3, 1.5) {};
		\node [style=vertex] (22) at (4, 1.5) {};
		\node [style=vertex] (23) at (3.75, 0.5) {};
		\node [style=vertex] (24) at (4, -0) {};
		\node [style=vertex] (25) at (6.75, 1.5) {};
		\node [style=vertex] (26) at (5.75, 1.5) {};
		\node [style=vertex] (27) at (6.75, -0) {};
		\node [style=vertex] (28) at (5.75, -0) {};
		\node [style=vertex] (29) at (6.5, 0.5) {};
	\end{pgfonlayer}
	\begin{pgfonlayer}{edgelayer}
		\draw [style=map] (1) to (0);
		\draw [style=map] (0) to (2);
		\draw [style=map] (1) to (3);
		\draw [style=map, bend left=90, looseness=2.00] (3) to (2);
		\draw [ultra thick, red, ->] (2) to (3);
		\draw [ultra thick, blue] (2) to (4);
		\draw [style=map] (7) to (5);
		\draw [style=map] (5) to (8);
		\draw [style=map] (7) to (6);
		\draw [style=map, bend left=90, looseness=2.00] (6) to (8);
		\draw [ultra thick, red, ->] (8) to (6);
		\draw [ultra thick, dashed, blue] (8) to (9);
		\draw [ultra thick, blue] (6) to (9);
		\draw [style=map] (14) to (13);
		\draw [style=map] (13) to (10);
		\draw [style=map] (14) to (12);
		\draw [style=map, bend left=90, looseness=2.00] (12) to (10);
		\draw [ultra thick, dashed, red] (10) to (12);
		\draw [ultra thick, blue] (12) to (11);
		\draw [ultra thick, red, ->] (13) to (11);
		\draw [style=map] (18) to (16);
		\draw [style=map] (16) to (17);
		\draw [style=map] (18) to (15);
		\draw [style=map, bend left=90, looseness=2.00] (15) to (17);
		\draw [ultra thick, dashed, blue] (15) to (19);
		\draw [ultra thick, red, ->] (16) to (19);
		\draw [ultra thick, blue, in=60, out=75, looseness=3.25] (16) to (17);
		\draw [style=map] (21) to (20);
		\draw [style=map] (20) to (24);
		\draw [style=map] (21) to (22);
		\draw [style=map, bend left=90, looseness=2.00] (22) to (24);
		\draw [ultra thick, dashed, red] (20) to (23);
		\draw [ultra thick, blue, in=60, out=75, looseness=3.25] (20) to (24);
		\draw [ultra thick, red, ->] (24) to (23);
		\draw [style=map] (26) to (28);
		\draw [style=map] (28) to (27);
		\draw [style=map] (26) to (25);
		\draw [style=map, bend left=90, looseness=2.00] (25) to (27);
		\draw [ultra thick, dashed, blue, in=60, out=75, looseness=3.25] (28) to (27);
		\draw [ultra thick, red, ->] (27) to (29);
		\draw [ultra thick, blue, bend left=75, looseness=1.50] (27) to (25);
		\node (e1) at (-7cm,0.8cm) {\contour{white}{$e$}};
		\node (e2) at (-6.8cm,0.3cm) {\contour{white}{$\eta$}};
	\end{pgfonlayer}
\end{tikzpicture}
\caption{\label{fig:root rotation}The flip path $P(q,q')=(q_i,e_i,s_i)_{i=1}^5$ reroots $q$, rooted in $e$, in the edge $\eta$; above, the case where the union of faces adjacent to $e$ and $\eta$ has 8 corners. Below, the case where $\eta$ is a double edge within a degenerate face; the case where $e$ is a double edge is analogous.}
\end{figure}

We can now present the construction of the flip path corresponding to a colour change from $+$ to $-$ on an edge $(v,p(v))$, where $v$ is a leaf of a labelled tree $t$:

\begin{lemma}[Construction of the path, $+$ to $-$]\label{+->- construction} With the same notation as in Lemma~\ref{+->- static}, consider $t,v,\epsilon$ and the edges $e_1,\ldots, e_k$ of $\phi(t,\epsilon)$, all incident to a vertex $w$ (possibly equal to $\delta$). If the root edge of $\phi(t,\epsilon)$ does not belong to $\{e_1,\ldots,e_k\}$, then set $P_\epsilon(t,t^{v,-})=(q_i,e_i,-)_{i=1}^k$, where $q_1=\phi(t,\epsilon)$ and $q_{i+1}=q_i^{e_i,-}$. If some $e_j$ is the root of $\phi(t,\epsilon)$, then set $P_\epsilon(t,t^{v,-})$ to be the path described above, with the sequence of flips described in Lemma~\ref{root rotation} for the rotation of the root edge $e_j$ around its endpoint that is not $w$, injected right before the flip $(q_j,e_j,-)$. For simplicity, we shall not renumber the quadrangulations $q_1,\ldots,q_k$ in this case, but $q_j$ and $q_{j+1}$ will be 6 flips rather than one flip apart.

Then $q_k^{e_k,-}=\phi(t^{v,-},\epsilon)$.
\end{lemma}
\begin{proof}
	
	Consider the edges adjacent to $v$ in the quadrangulation $\phi(t^{v,-},\epsilon)$ and number them as $\eta_0,\ldots,\eta_k$, in clockwise order around $v$, so that $\eta_0$ is the (unique) edge of $\phi(t^{v,-},\epsilon)$ joining $v$ to $p(v)$. 
	
	We can show by induction that the quadrangulation $q_i$ is the quadrangulation $q_1$, where edges $e_1,\ldots,e_{i-1}$ have been replaced by edges $\eta_1,\ldots,\eta_{i-1}$ (and the natural edge identification after the flip sequence pairs $e_1$ with $\eta_1$, $e_2$ with $\eta_2$, and so on) -- see Figure~\ref{fig:P(t,t^{v,-})}. Indeed, if this is true for $q_i$, then the edge $e_i$ as seen in $q_i$ is adjacent to two faces, one in whose clockwise contour it's preceded by $\eta_{i-1}$, one in whose clockwise contour it's preceded by $e_{i+1}$ (where we set $e_{k+1}$ to be the edge between $t(c)$ and $t(t(c))$). It follows that flipping $e_i$ counterclockwise results in an edge joining the appropriate endpoint of $e_{i+1}$ to $v$, which does correspond to creating $\eta_i$. Since $\eta_0$ can already be identified with an edge of $q_1$, the final quadrangulation $q_k^{e_k,-}$ is $\phi(t^{v,-},\epsilon)$, up to rerooting.
	
	If some $e_j$ is the root edge of $q_1=\phi(t,\epsilon)$, then -- as described in Lemma~\ref{+->- static} -- $\eta_{j-1}$ is the root edge of $\phi(t^{v,-},\epsilon)$, oriented towards $v$ if and only if the root of $q_1$ is oriented towards $w$. We have shown that in $q_j$ edges $e_j$ and $\eta_{j-1}$ do belong to the same face, $e_j$ coming right after $\eta_{j-1}$ in its clockwise contour; performing the appropriate root rotation sequence of flips on $q_j$ has exactly the effect of rerooting $q_j$ in $\eta_{j-1}$ (with the desired orientation). After that is done, one can proceed with the ``normal'' flip sequence to obtain $q_k^{e_k,-}$, which is now rooted correctly.
\end{proof}

This concludes the description of our canonical flip paths corresponding to label changes of leaves: we simply set $P_\epsilon(t,t^{v,-})$, when $l(v)=l(p(v))$ in $t$, to be the concatenation of $P_\epsilon(t,t^{v,+})$ and $P_\epsilon(t^{v,+},(t^{v,+})^{v,-})$, while in general $P_\epsilon(t^{v,x},t)$ is the reverse path of $P_\epsilon(t,t^{v,x})$ (for $x\in\{+,-,=\}$).

In order to compare spectral gaps as we did in Section~\ref{chain comparison}, we need to estimate the maximum length of a flip path of the form $P_\epsilon(t,t^{v,x})$ and the number of paths involving any fixed quadrangulation edge flip; we do this via the following two lemmas.

\begin{lemma}\label{label move path length}
Given $t\in\LT_n$, a leaf $v$ of $t$ and $x\in\{+,-,=\}$, the length of the flip path $P_\epsilon(t,t^{v,x})$, for $\epsilon\in\{-1,1\}$, is at most $2n+6$.\end{lemma}

\begin{proof}
The flip sequence $P_\epsilon(t,t^{v,-})$, if we ignore the possible root rotating subsequence, does not flip ``the same edge'' twice, so it has length at most $2n$ (in fact, it has length at most the maximum degree of a vertex in $\phi(t,\epsilon)$, since all flipped edges are adjacent to the same vertex). The root rotating subsequence has length 5 and the path $P_\epsilon(t,t^{v,+})$ has length 1, hence the above estimate.
\end{proof}

\begin{figure}[t]
\centering\begin{tabularx}{.9\textwidth}{>{\centering\arraybackslash}X>{\centering\arraybackslash}X>{\centering\arraybackslash}X}
 & & \\
\begin{tikzpicture} 
	\begin{pgfonlayer}{nodelayer}
		\node [style=vertex, label=left:$a$] (0) at (-1, -0) {};
		\node [style=vertex, fill=red, draw=pink, label=above:$v$,label=right:$a+1$] (1) at (-0.5, 0.75) {};
		\node [style=vertex, label=$a$] (2) at (-0.7, 2.5) {};
		\node [style=vertex, label=$a$] (3) at (0.25, 3) {};
		\node [style=vertex, label=$a$] (4) at (1.75, 2.75) {};
		\node [style=vertex, label=below:$a-1$] (5) at (0.95, 0.25) {};
		\node [style=vertex, label=above right:$a-2$] (6) at (2.5, -0) {};
		
		\node (7) at (-0.2, 1.55) {\contour{white}{$e_1$}};
		\node (8) at (-0, 2.1) {\contour{white}{$e_2$}};
		\node (9) at (0.7, 2.25) {\contour{white}{$e_3$}};
		\node (10) at (1.45, 1.75) {\contour{white}{$e_4$}};
	\end{pgfonlayer}
	\begin{pgfonlayer}{edgelayer}
		\draw [style=tree] (0) to (1);
		\draw [style=map, bend left, looseness=1] (2) to (5); 
		\draw [style=map, bend right=15, looseness=1.00] (5) to (3); 
		\draw [style=map, very thick, ->] (4) to (5); 
		\draw [style=map, bend left=90, looseness=2.25] (0) to (5); 
		\draw [style=map, ultra thick, bend left=15, looseness=1.00] (5) to (6);
		\draw [style=map, bend left, looseness=1.00] (1) to (0);
		\draw [style=map, ultra thick] (0) to (5);
	\end{pgfonlayer}
	\pgfresetboundingbox
	\draw[draw=none, use as bounding box] (-1,-0.5) rectangle (3,3.5);
\end{tikzpicture} &
\begin{tikzpicture} 
	\begin{pgfonlayer}{nodelayer}
		\node [style=vertex, label=left:$a$] (0) at (-1, -0) {};
		\node [style=vertex, fill=red, draw=pink, label=above:$v$,label=right:$$] (1) at (-0.5, 0.75) {};
		\node [style=vertex, label=$a$] (2) at (-0.7, 2.5) {};
		\node [style=vertex, label=$a$] (3) at (0.25, 3) {};
		\node [style=vertex, label=$a$] (4) at (1.75, 2.75) {};
		\node [style=vertex, label=below:$a-1$] (5) at (0.95, 0.25) {};
		\node [style=vertex, label=above right:$a-2$] (6) at (2.5, -0) {};
		
		\node (7) at (-0.3, 1.75) {\contour{white}{$\eta_1$}};
		\node (8) at (-0, 2.1) {\contour{white}{$e_2$}};
		\node (9) at (0.7, 2.25) {\contour{white}{$e_3$}};
		\node (10) at (1.45, 1.75) {\contour{white}{$e_4$}};
	\end{pgfonlayer}
	\begin{pgfonlayer}{edgelayer}
		\draw [style=tree] (0) to (1);
		\draw [blue, dashed, ultra thick, bend left=90, looseness=2] (0) to (5); 
		\draw [style=map, bend left, looseness=1] (2) to (5); 
		\draw [style=map, bend right=15, looseness=1.00] (5) to (3); 
		\draw [style=map, very thick, ->] (4) to (5); 
		\draw [blue, ultra thick, bend left, looseness=1] (2) to (1); 
		\draw [style=map, ultra thick, bend left=15, looseness=1.00] (5) to (6);
		\draw [style=map, bend right, looseness=1.00] (1) to (0);
		\draw [style=map, ultra thick] (0) to (5);
	\end{pgfonlayer}
		\pgfresetboundingbox
	\draw[draw=none, use as bounding box] (-1,-0.5) rectangle (3,3.5);
\end{tikzpicture} &
\begin{tikzpicture} 
	\begin{pgfonlayer}{nodelayer}
		\node [style=vertex, label=left:$a$] (0) at (-1, -0) {};
		\node [style=vertex, fill=red, draw=pink, label=above:$v$,label=right:$$] (1) at (-0.5, 0.75) {};
		\node [style=vertex, label=$a$] (2) at (-0.7, 2.5) {};
		\node [style=vertex, label=$a$] (3) at (0.25, 3) {};
		\node [style=vertex, label=$a$] (4) at (1.75, 2.75) {};
		\node [style=vertex, label=below:$a-1$] (5) at (0.95, 0.25) {};
		\node [style=vertex, label=above right:$a-2$] (6) at (2.5, -0) {};
		
		\node (7) at (-0.3, 1.75) {\contour{white}{$\eta_1$}};
		\node (8) at (0.2, 2.1) {\contour{white}{$\eta_2$}};
		\node (9) at (0.7, 2.25) {\contour{white}{$e_3$}};
		\node (10) at (1.45, 1.75) {\contour{white}{$e_4$}};
	\end{pgfonlayer}
	\begin{pgfonlayer}{edgelayer}
		\draw [style=tree] (0) to (1);
		\draw [ultra thick, blue, dashed, bend left, looseness=1] (2) to (5); 
		\draw [blue, ultra thick, bend right=15, looseness=1.00] (1) to (3); 
		\draw [style=map, bend right=15, looseness=1.00] (5) to (3); 
		\draw [style=map, very thick, ->] (4) to (5); 
		\draw [style=map, bend left, looseness=1] (2) to (1); 
		\draw [style=map, ultra thick, bend left=15, looseness=1.00] (5) to (6);
		\draw [style=map, bend right, looseness=1.00] (1) to (0);
		\draw [style=map, ultra thick] (0) to (5);
	\end{pgfonlayer}
		\pgfresetboundingbox
	\draw[draw=none, use as bounding box] (-1,-0.5) rectangle (3,3.5);
\end{tikzpicture}\\
$\phi(t,\epsilon)$ & $(q_1,q_1^{e_1,-})$ & $(q_2,q_2^{e_2,-})$\\
 & & \\
\begin{tikzpicture} 
	\begin{pgfonlayer}{nodelayer}
	\draw[use as bounding box, white] (-1,-0.5) rectangle (2.5,3.5);
		\node [style=vertex, label=left:$a$] (0) at (-1, -0) {};
		\node [style=vertex, fill=red, draw=pink, label=above:$v$,label=right:$$] (1) at (-0.5, 0.75) {};
		\node [style=vertex, label=$a$] (2) at (-0.7, 2.5) {};
		\node [style=vertex, label=$a$] (3) at (0.25, 3) {};
		\node [style=vertex, label=$a$] (4) at (1.75, 2.75) {};
		\node [style=vertex, label=below:$a-1$] (5) at (0.95, 0.25) {};
		\node [style=vertex, label=above right:$a-2$] (6) at (2.5, -0) {};
		
		\node (7) at (-0.3, 1.75) {\contour{white}{$\eta_1$}};
		\node (8) at (0.2, 2.1) {\contour{white}{$\eta_2$}};
		\node (10) at (1.45, 1.75) {\contour{white}{$e_4$}};
		\node (9) at (1.2, 2.2) {\contour{white}{$\eta_3$}};
		
		\node (9) at (3.75, -0.9) {\fbox{\contour{white}{$P(q_3^{e_3,-},q_4)$}}};
	\end{pgfonlayer}
	\begin{pgfonlayer}{edgelayer}
	\draw[dashed] (3.9,3.5)--(3.9,-0.5);
		\draw [style=tree] (0) to (1);
		\draw [style=map, bend right=15, looseness=1.00] (1) to (3); 
		\draw [blue, dashed, ultra thick, bend right=15, looseness=1.00] (5) to (3); 
		\draw [style=map, very thick, ->] (4) to (5); 
		\draw [style=map, bend left, looseness=1] (2) to (1); 
		\draw [blue, ultra thick] (4) to (1); 
		\draw [style=map, ultra thick, bend left=15, looseness=1.00] (5) to (6);
		\draw [style=map, bend right, looseness=1.00] (1) to (0);
		\draw [style=map, ultra thick] (0) to (5);
	\end{pgfonlayer}
	\pgfresetboundingbox
	\draw[draw=none, use as bounding box] (-1,-0.5) rectangle (3,3.5);
\end{tikzpicture}&
\begin{tikzpicture} 
	\begin{pgfonlayer}{nodelayer}
		\node [style=vertex, label=left:$a$] (0) at (-1, -0) {};
		\node [style=vertex, fill=red, draw=pink, label=above:$v$,label=right:$$] (1) at (-0.5, 0.75) {};
		\node [style=vertex, label=$a$] (2) at (-0.7, 2.5) {};
		\node [style=vertex, label=$a$] (3) at (0.25, 3) {};
		\node [style=vertex, label=$a$] (4) at (1.75, 2.75) {};
		\node [style=vertex, label=below:$a-1$] (5) at (0.95, 0.25) {};
		\node [style=vertex, label=above right:$a-2$] (6) at (2.5, -0) {};
		
		\node (7) at (-0.3, 1.75) {\contour{white}{$\eta_1$}};
		\node (8) at (0.2, 2.1) {\contour{white}{$\eta_2$}};
		\node (9) at (1.2, 2.2) {\contour{white}{$\eta_3$}};
		\node (10) at (1.45, 0.75) {\contour{white}{$\eta_4$}};
	\end{pgfonlayer}
	\begin{pgfonlayer}{edgelayer}
		\draw [style=tree] (0) to (1);
		\draw [style=map, bend right=15, looseness=1.00] (1) to (3); 
		\draw [ultra thick, blue, dashed] (4) to (5); 
		\draw [style=map, bend left, looseness=1] (2) to (1); 
		\draw [style=map, ultra thick, bend left=15, looseness=1.00] (5) to (6);
		\draw [style=map, bend right, looseness=1.00] (1) to (0);
		\draw [style=map, ultra thick] (0) to (5);
		\draw [ultra thick, blue, looseness=0.6] (1) to (6); 
		\draw [style=map, very thick, ->] (4) to (1); 
	\end{pgfonlayer}
		\pgfresetboundingbox
	\draw[draw=none, use as bounding box] (-1,-0.5) rectangle (3,3.5);
\end{tikzpicture}&
\begin{tikzpicture}
	\begin{pgfonlayer}{nodelayer}
		\node [style=vertex, label=left:$a$] (0) at (-1, -0) {};
		\node [style=vertex, fill=red, draw=pink, label=above left:$v$,label=right:$a-1$] (1) at (-0.5, 0.75) {};
		\node [style=vertex, label=$a$] (2) at (-0.7, 2.5) {};
		\node [style=vertex, label=$a$] (3) at (0.25, 3) {};
		\node [style=vertex, label=$a$] (4) at (1.75, 2.75) {};
		\node [style=vertex, label=below:$a-1$] (5) at (0.95, 0.25) {};
		\node [style=vertex, label=above right:$a-2$] (6) at (2.5, -0) {};
		
		\node (7) at (-1, 0.45) {\contour{white}{$\eta_0$}};
		\node (7) at (-0.3, 1.75) {\contour{white}{$\eta_1$}};
		\node (8) at (0.2, 2.1) {\contour{white}{$\eta_2$}};
		\node (9) at (1.2, 2.2) {\contour{white}{$\eta_3$}};
		\node (10) at (1.45, 0.75) {\contour{white}{$\eta_4$}};
	\end{pgfonlayer}
	\begin{pgfonlayer}{edgelayer}
		\draw [style=tree] (0) to (1);
		\draw [style=map, bend left, looseness=1] (2) to (1);
		\draw [style=map, bend right=15, looseness=1.00] (1) to (3);
		\draw [style=map, very thick, ->] (4) to (1);
		\draw [style=map, looseness=0.6] (1) to (6);
		\draw [style=map, ultra thick, bend left=15, looseness=1.00] (5) to (6);
		\draw [style=map, bend right, looseness=1.00] (1) to (0);
		\draw [style=map, ultra thick] (0) to (5);
	\end{pgfonlayer}
		\pgfresetboundingbox
	\draw[draw=none, use as bounding box] (-1,-0.5) rectangle (3,3.5);
\end{tikzpicture}\\
$(q_3,q_3^{e_3,-})$ & $(q_4,q_4^{e_4,-})$ & $\phi(t^{v,-},\epsilon)$ \\
\end{tabularx}
\caption{\label{fig:P(t,t^{v,-})}The flip path $P_\epsilon(t,t^{v,-})$. Since $q_1$ is rooted in $e_4$, a root rotating sequence is inserted right before flipping $e_4$, thus correctly rooting the final quadrangulation in $\eta_3$, i.e.~the flipped image of $e_3$.}
\end{figure}
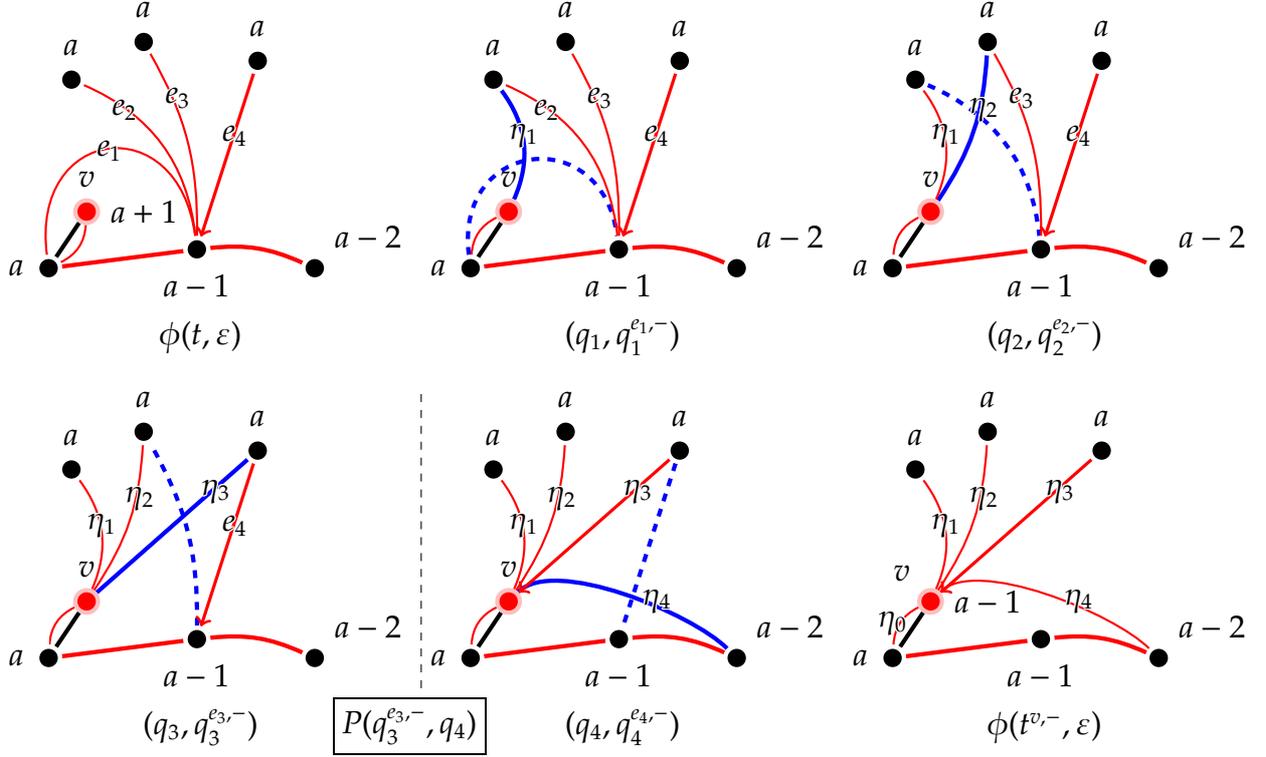

\begin{lemma}\label{label move congestion}
	Let $(q,e,s)$ be a triple with $q\in\sQ_n^\bullet$, $e\in E(q)$, $s\in\{+,-\}$; then there is a constant $C$ such that there are at most $C$ quadruples $(t,v,x,\epsilon)$, where $t\in\LT_n$, $v$ is a leaf in $t$, $x\in \{+,-,=\}$ and $\epsilon=\pm1$, for which $(q,e,s)$ appears in the flip path $P_\epsilon(t,t^{v,x})$.
\end{lemma}

\begin{proof}
Let us first consider quadruples of the form $(t,v,-,\epsilon)$.

Suppose $t,v$ are such that $l(v)-l(p(v))=1$ and $(q,e,s)$ appears in $P_\epsilon(t,t^{v,-})=(q_i,e_i,s_i)_{i=1}^k$ (that is, $q=q_l$, $e=e_l$ and $s=s_l$ for some $l\in\{1,\ldots,k\}$); we first consider the case where no root rotation sequence appears in $P_\epsilon(t,t^{v,-})$. Let $v_1$ and $v_2$ be the endpoints of $e$ in $q^{e,s}$; then $v$ corresponds to one of these two vertices in $q_1=\phi(t,\epsilon)$, since every edge $e_i$ is adjacent to $v$ in $q_i^{e_i,-}$ and subsequent quadrangulations in the flip path. In particular, since the degree of $v$ in $q_1$ is 1 and it increases by 1 with each flip in $P_\epsilon(t, t^{v,-})$, if $d_1$ and $d_2$ are the respective degrees of $v_1$ and $v_2$ in $q$, then $l=d_1$ or $l=d_2$, according to whether $v$ is $v_1$ or $v_2$. Suppose $v$ is $v_1$; let $\eta_0,\ldots,\eta_{d_1-1}$ be the edges incident to $v_1$ in $q$, in clockwise order around $v_1$, numbered so that $e$ will end up in between $\eta_{d_1-1}$ and $\eta_0$ in $q^{e,s}$. Then $q_1=\phi(t,\epsilon)=((q^{\eta_{d_1-1},+})^{\eta_2,+})^{\ldots})^{\eta_1,+}$, since $e_i$ corresponds to the edge $\eta_{d_{i-1}}$ in $q$. We therefore have only 2 possibilities for $(t,v)$.

Now suppose a root rotation sequence does appear in $P_\epsilon(t,t^{v,-})$; if it appears strictly after $(q,e,s)$, then the reasoning above is still valid. If it appears strictly before, then the root of $q$ is some $\eta_j$, and -- reasoning as before -- the quadrangulation $q_1$ is recovered by inserting the reverse of a root rotation sequence right before applying the clockwise flip of $\eta_j$ (or right at the end if $j=0$). We are left to deal with the case where $(q,e,s)$ actually belongs to a root rotation sequence. Notice that the number of possibilities for the quadrangulation $q'$ obtained at the end of the root rotation sequence is bounded by a constant independent of $n$, since the sequence only acts within (at most) three adjacent faces, two of which are adjacent to $e$. Having established some $q'$ to be the quadrangulation in question, the vertex $v$ must be one of the endpoints of the root edge $e'$ of $q'$. As before, we can now reconstruct $q_1$ by labelling $\eta_0,\ldots,\eta_{d_{i-1}}$ the edges incident to this endpoint and performing the appropriate reverse root rotation sequence, followed by clockwise flips on $\eta_{d_{i-1}},\ldots,\eta_1$.

Similarly, suppose $t,v$ are such that $l(v)=l(p(v))$ and $(q,e,s)$ appears in $P_\epsilon(t,t^{v,-})$; if $(q,e,s)$ appears in $P_\epsilon(t^{v,+},t^{v,-})$, then $\phi(t^{v,+},\epsilon)$ can be reconstructed as above, hence $t=(t^{v,+})^{v,=}$. Otherwise we have $(q,e,s)=P_\epsilon(t,t^{v,+})$, hence $q=\phi(t,\epsilon)$.

Now, since $(q,e,-)$ appears in $P_\epsilon(t,t^{v,-})$ if and only if $(q^{e,-}, e, +)$ appears in $P_\epsilon(t^{v,-}, t)$, we have thus also covered the cases where $(t,v,x)$ is such that $l(v)=l(p(v))-1$, which correspond to another 4 possibilities.

Finally, the missing cases ($x\in\{+,=\}$ and $l(v)-l(p(v))\in\{1,0\}$) are completely straightforward, since they correspond to a flip path of length 1, and therefore imply that $(t,\epsilon)=\phi^{-1}(q)$ and $v$ is the one vertex whose label is changed by the flip.

This results, indeed, in a number of possibilities for the quadruple $(t,v,x,\epsilon)$ that is bounded independently of $n$.	
\end{proof}

\subsection{The leaf translation sequences $P_\epsilon(t,t^{v,\rightarrow})$, $P_\epsilon(t,t^{v,\leftarrow})$}\label{section: leaf translation}
The other type of ``move'' we wish to emulate via quadrangulation flips is the translation of a leaf left or right in the contour of the tree. In doing this, we may suppose the leaf has the same label as its parent, and deal with all other cases by prefixing and appending flip paths of the type $P_\epsilon(t,t^{v,=})$ and $P_\epsilon(t^{v,=},t)$, which we have constructed in the previous section.

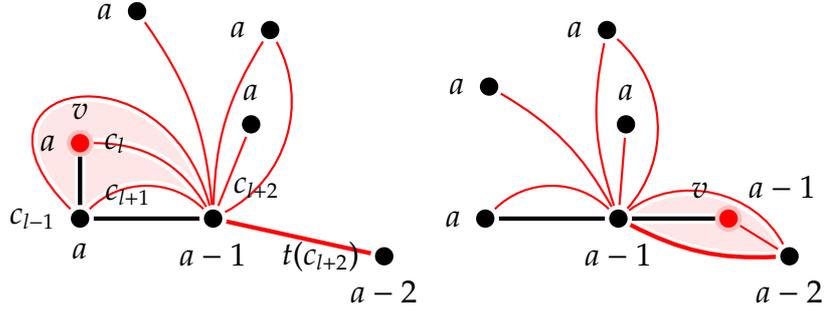
\begin{figure}\centering
\begin{tikzpicture}
	\begin{pgfonlayer}{nodelayer}
		\node [style=vertex, label=below:$a$, label=left:$c_{l-1}$, label={[shift={(18pt,-4pt)}]$c_{l+1}$}] (0) at (-2, -0) {};
		\node [style=vertex, fill=red, draw=pink, label=left:$a$, label=above:$v$,label=right:$c_l$] (1) at (-2, 1) {};
		\node [style=vertex, label=below:$a-1$, label=above right:$c_{l+2}$] (2) at (-0.25, -0) {};
		\node [style=vertex, label=left:$a$] (3) at (-1.25, 2.75) {};
		\node [style=vertex, label=left:$a$] (4) at (0.5, 2.5) {};
		\node [style=vertex, label=below:$a-2$, label=left:$t(c_{l+2})$] (5) at (2, -0.5) {};
		\node [style=vertex, label=above:$a$] (6) at (0.25, 1.25) {};
	\end{pgfonlayer}
	\begin{pgfonlayer}{edgelayer}
		\fill [red!10] (0) [in=105, out=135, looseness=3.25] to (2) [bend right=45, looseness=1.00] to (0);
			
		\draw [style=b, bend left, looseness=1.00] (1) to (2);
		\draw [style=b] (1) to (0);
		\draw [style=b, in=105, out=135, looseness=3.25] (0) to (2);
		\draw [style=b, bend left=45, looseness=1.00] (0) to (2);
		
		\draw [style=map, bend left, looseness=1.00] (1) to (2);
		\draw [style=map, bend left=15, looseness=1.00] (3) to (2);
		\draw [style=map, bend right=45, looseness=1.00] (2) to (4);
		\draw [style=map, ultra thick] (5) to (2);
		\draw [style=tree] (1) to (0);
		\draw [style=tree] (0) to (2);
		\draw [style=map, in=105, out=135, looseness=3.25] (0) to (2);
		\draw [style=map, bend left=45, looseness=1.00] (0) to (2);
		\draw [style=map, bend right=15, looseness=1.00] (4) to (2);
		\draw [style=map] (6) to (2);
	\end{pgfonlayer}
\end{tikzpicture}
\begin{tikzpicture}
	\begin{pgfonlayer}{nodelayer}
		\node [style=vertex, label=left:$a$] (0) at (-2, -0) {};
		\node [style=vertex, fill=red, draw=pink, label=above right:$a-1$, label=above left:$v$] (1) at (1.2, 0) {};
		\node [style=vertex, label=below:$a-1$] (2) at (-0.25, -0) {};
		\node [style=vertex, label=left:$a$] (3) at (-1.95, 1.75) {};
		\node [style=vertex, label=left:$a$] (4) at (-0.4, 2.5) {};
		\node [style=vertex, label=below:$a-2$] (5) at (2, -0.5) {};
		\node [style=vertex, label=above:$a$] (6) at (-0.15, 1.25) {};
	\end{pgfonlayer}
	\begin{pgfonlayer}{edgelayer}
		\fill [red!10] (2) [bend left=50, looseness=1.00] to (5.center) [bend left=15, looseness=1.00] to (2);
		
		\draw [style=b] (1) to (2);
		\draw [style=b, bend left=50, looseness=1.00] (2) to (5);
		\draw [style=b, bend left=15] (5) to (2);
		
		\draw [style=map] (1) to (5);
		\draw [style=map, bend left=15, looseness=1.00] (3) to (2);
		\draw [style=map, bend right=45, looseness=1.00] (2) to (4);
		\draw [style=map, ultra thick, bend left=15] (5) to (2);
		\draw [style=tree] (1) to (2);
		\draw [style=tree] (0) to (2);
		\draw [style=map, bend left=50, looseness=1.00] (2) to (5);
		\draw [style=map, bend left=45, looseness=1.00] (0) to (2);
		\draw [style=map, bend right=15, looseness=1.00] (4) to (2);
		\draw [style=map] (6) to (2);
	\end{pgfonlayer}
\end{tikzpicture}
\caption{\label{fig:leaf translation static}The quadrangulations $\phi(t,\epsilon)$ and $\phi(t^{v,\rightarrow},\epsilon)$, drawn in a case where $l(c_{l+2})=l(v)-1$.}
\end{figure}

The description of $\phi(t^{v,\rightarrow},\epsilon)$ in terms of $\phi(t,\epsilon)$ is rather simple and depicted in Figure~\ref{fig:leaf translation static}:
\begin{lemma}\label{leaf translation static}
	Consider $(t,\epsilon)\in\LT_n\times\{-1,1\}$ and let $v$ be a leaf of $t$ such that $l(v)=l(p(v))$. Let $c_1,\ldots,c_{2n}$ be the clockwise contour of $t$ and suppose $c_l$, with $2\leq l\leq 2n-1$, is the corner of $v$ (notice that, if we had $l=2n$, we would have $t^{v,\rightarrow}=t$). Then $\phi(t^{v,\rightarrow},\epsilon)$ can be obtained from $\phi(t,\epsilon)$ by
	\begin{itemize}
	\item identifying the two edges $(c_{l-1}, t(c_{l-1}))$ and $(c_{l+1},t(c_{l+1}))$ and erasing the double edge $(c_l,c_{l+1})$ (that is eliminating the one degenerate face of $\phi(t,\epsilon)$ which corresponds to the tree edge $(p(v),v)$);
	\item replacing the edge $(c_{l+2},t(c_{l+2}))$ by a degenerate face whose internal vertex is adjacent to the vertex of $t(c_{l+2})$.
	\end{itemize}
	Notice that it is possible that either the edge $(c_{l-1},t(c_{l-1}))$ (if $l=2$) or the edge $(c_{l+2},t(c_{l+2})))$ (if $l=2n-1$) is the root edge of $\phi(t,\epsilon)$. In the former case, the root of $\phi(t^{v,\rightarrow},\epsilon)$ is the edge obtained from identifying $(c_{l-1}, t(c_{l-1}))$ and $(c_{l+1},t(c_{l+1}))$, oriented as before; in the latter, it is the second edge of the new degenerate face in clockwise order around the vertex of $c_{l+2}$, oriented as $(c_{l+2},t(c_{l+2}))$ was.
\end{lemma}

\begin{proof}
Given $t\in\LT_n$ and $\epsilon\in\{-1,1\}$, a leaf $v$ such that $l(v)=l(p(v))$ is the internal vertex of a degenerate face in $\phi(t,\epsilon)$; removing the leaf (and the tree edge joining it to its parent) results in a tree $t'\in\LT_{n-1}$ and a quadrangulation $\phi(t',\epsilon)$ in which the face is eliminated by identifying the two edges of its boundary (into a root edge with the same orientation as before in the case where one of them was root edge in $\phi(t,\epsilon)$).
 
Notice that erasing the leaf $v$ from $t$ and from $t^{v,\rightarrow}$ yields the same tree $t'$. The quadrangulation $\phi(t^{v,\rightarrow},\epsilon)$ can thus be obtained by first performing the operation described above to build $\phi(t',\epsilon)$ and then performing it ``in reverse'' by replacing the appropriate edge (which is the one drawn from the corner of $t'$ that contains the edge joining $v$ to $p(v)$ in $t^{v,\rightarrow}$) with a degenerate face (see Figure~\ref{fig:leaf translation static}).
\end{proof}

The quadrangulation flip path $P_\epsilon(t,t^{v,\rightarrow})$ will depend on the label of the vertex $w$ of corner $c_{l+2}$ in $t$; since the cases where $l(w)=l(v)$ and $l(w)=l(v)+1$ are simpler (we can construct a path of length 1 in the first case and 3 in the second!), we refer to Figure~\ref{fig:-> over =} and Figure~\ref{fig:-> over +} for its construction, which only involves flips within two adjacent faces of $\phi(t,\epsilon)$. Notice that, furthermore, the construction preserves the edges issued from $c_{l-1}$ and $c_{l+2}$, so that the root edge is automatically the correct one in the quadrangulations $q_2$ from Figure~\ref{fig:-> over =} and $q_4$ from Figure~\ref{fig:-> over +}.

The case where $l(w)=l(v)-1$ is more complex, and will be treated in the following lemma.

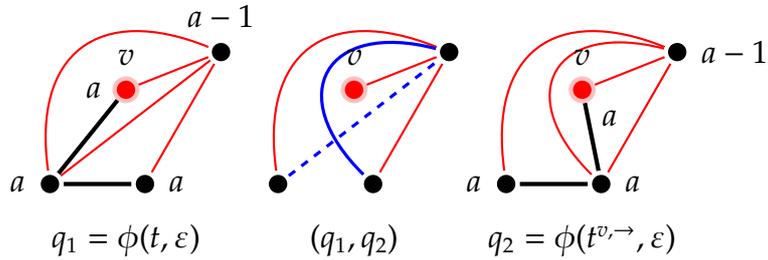
\begin{figure}\centering
\begin{tikzpicture}
	\begin{pgfonlayer}{nodelayer}
		\node [style=vertex, label=left:$a$] (0) at (-5.5, 1) {};
		\node [style=vertex, label=right:$a$] (1) at (-4.25, 1) {};
		\node [style=vertex, label=left:$a$, label=above:$v$, draw=pink, fill=red] (2) at (-4.5, 2.25) {};
		\node [style=vertex, label=above:$a-1$] (3) at (-3.25, 2.75) {};
		\node (4) at (-4.5, 0.25) {$q_1=\phi(t,\epsilon)$};
		\node (5) at (-1.5, 0.25) {$(q_1,q_2)$};
		\node [style=vertex] (6) at (-1.25, 1) {};
		\node [style=vertex] (7) at (-2.5, 1) {};
		\node [style=vertex, label=above:$v$, draw=pink, fill=red] (8) at (-1.5, 2.25) {};
		\node [style=vertex] (9) at (-0.25, 2.75) {};
		\node [style=vertex, label=left:$a$] (10) at (0.5, 1) {};
		\node [style=vertex, label=below right:$a$, label=above:$v$, draw=pink, fill=red] (11) at (1.5, 2.25) {};
		\node [style=vertex, label=right:$a$] (12) at (1.75, 1) {};
		\node [style=vertex, label=right:$a-1$] (13) at (2.75, 2.75) {};
		\node (14) at (1.5, 0.25) {$q_2=\phi(t^{v,\rightarrow},\epsilon)$};
	\end{pgfonlayer}
	\begin{pgfonlayer}{edgelayer}
		\draw [style=tree] (2) to (0);
		\draw [style=tree] (0) to (1);
		\draw [style=map, bend left=60, looseness=1.50] (0) to (3);
		\draw [style=map] (0) to (3);
		\draw [style=map] (2) to (3);
		\draw [style=map] (1) to (3);
		\draw [style=map, bend left=60, looseness=1.50] (7) to (9);
		\draw [blue, very thick, dashed] (7) to (9);
		\draw [style=map] (8) to (9);
		\draw [blue, very thick, bend left=75, looseness=2.00] (6) to (9);
		\draw [style=map] (6) to (9);
		\draw [style=tree] (10) to (12);
		\draw [style=map, bend left=60, looseness=1.50] (10) to (13);
		\draw [style=map] (11) to (13);
		\draw [style=map, bend left=75, looseness=2.00] (12) to (13);
		\draw [style=map] (12) to (13);
		\draw [style=tree] (11) to (12);
	\end{pgfonlayer}
\end{tikzpicture}
\caption{\label{fig:-> over =}The flip path $P_\epsilon(t,t^{v,\rightarrow})$ has length 1 when the edge $(v,p(v))$ is replanted onto a vertex which has the same label as $v$ and $p(v)$. Notice that the edge that needs to be flipped in order to change $\phi(t,\epsilon)$ into $\phi(t^{v,\rightarrow},\epsilon)$ \emph{cannot} be the root edge of $\phi(t,\epsilon)$, since it is not issued from the first corner of $t$.}
\end{figure}

\begin{figure}[t]\centering
\begin{tikzpicture}
	\begin{pgfonlayer}{nodelayer}
		\node [style=vertex, label=left:$a$] (0) at (-5.5, 1) {};
		\node [style=vertex, label=right:$a+1$] (1) at (-4.25, 1) {};
		\node [style=vertex, label=left:$a$, label=above:$v$, draw=pink, fill=red] (2) at (-4.75, 2.25) {};
		\node [style=vertex, label=left:$a$] (3) at (-3.25, 1.75) {};
		\node [style=vertex, label=above:$a-1$] (4) at (-3.25, 2.75) {};
		
		\node [style=vertex] (5) at (-1.25, 1) {};
		\node [style=vertex] (6) at (-0.25, 1.75) {};
		\node [style=vertex] (7) at (-0.25, 2.75) {};
		\node [style=vertex] (8) at (-2.5, 1) {};
		\node [style=vertex, label=above:$v$, draw=pink, fill=red] (9) at (-1.75, 2.25) {};
		\node [style=vertex, label=above:$v$, draw=pink, fill=red] (10) at (1.5, 2.25) {};
		\node [style=vertex] (11) at (1.75, 1) {};
		\node [style=vertex] (12) at (2.75, 2.75) {};
		\node [style=vertex] (13) at (2.75, 1.75) {};
		\node [style=vertex] (14) at (0.5, 1) {};
		\node [style=vertex] (15) at (3.5, 1) {};
		\node [style=vertex] (16) at (5.75, 1.75) {};
		\node [style=vertex] (17) at (4.75, 1) {};
		\node [style=vertex] (18) at (5.75, 2.75) {};
		\node [style=vertex, label=above:$v$, draw=pink, fill=red] (19) at (4.5, 2.25) {};
		
		\node [style=vertex, label=above:$v$, label=left:$a+1$, draw=pink, fill=red] (20) at (7.5, 2.25) {};
		\node [style=vertex, label=right:$a+1$] (21) at (7.75, 1) {};
		\node [style=vertex, label=above:$a-1$] (22) at (8.75, 2.75) {};
		\node [style=vertex, label=right:$a$] (23) at (8.75, 1.75) {};
		\node [style=vertex, label=left:$a$] (24) at (6.5, 1) {};
		
		\node (x) at (-4.5,0.25) {$\phi(t,\epsilon)=q_1$};
		\node (x) at (-1.5,0.25) {$(q_1,q_2)$};
		\node (x) at (1.5,0.25) {$(q_2,q_3)$};
		\node (x) at (4.5,0.25) {$(q_3,q_4)$};
		\node (x) at (7.5,0.25) {$q_4=\phi(t^{v,\rightarrow},\epsilon)$};
	\end{pgfonlayer}
	\begin{pgfonlayer}{edgelayer}
		\draw [style=tree] (2) to (0);
		\draw [style=tree] (0) to (1);
		\draw [style=map, bend left=60, looseness=1.50] (0) to (4);
		\draw [style=map] (0) to (4);
		\draw [style=map] (2) to (4);
		\draw [style=map] (1) to (3);
		\draw [style=map] (3) to (4);
		\draw [style=map, bend left=60, looseness=1.50] (8) to (7);
		\draw [blue, very thick, dashed] (8) to (7);
		\draw [style=map] (9) to (7);
		\draw [style=map] (5) to (6);
		\draw [style=map] (6) to (7);
		\draw [blue, very thick] (9) to (5);
		\draw [style=map, bend left=60, looseness=1.50] (14) to (12);
		\draw [blue, very thick, dashed] (10) to (12);
		\draw [blue, very thick, bend left=105, looseness=4.25] (11) to (13);
		\draw [style=map] (13) to (12);
		\draw [style=map] (10) to (11);
		\draw [style=map] (11) to (13);
		\draw [style=map, bend left=60, looseness=1.50] (15) to (18);
		\draw [style=map, bend left=105, looseness=4.25] (17) to (16);
		\draw [style=map] (16) to (18);
		\draw [blue, very thick, dashed] (19) to (17);
		\draw [style=map] (17) to (16);
		\draw [blue, very thick] (19) to (16);
		\draw [style=map, bend left=60, looseness=1.50] (24) to (22);
		\draw [style=map, bend left=105, looseness=4.25] (21) to (23);
		\draw [style=map] (23) to (22);
		\draw [style=map] (21) to (23);
		\draw [style=map] (20) to (23);
		\draw [style=tree] (24) to (21);
		\draw [style=tree] (21) to (20);
		
		\draw [style=map, bend right] (0) to (1);
		\draw [style=map, bend right] (8) to (5);
		\draw [style=map, bend right] (14) to (11);
		\draw [style=map, bend right] (15) to (17);
		\draw [style=map, bend right] (24) to (21);
	\end{pgfonlayer}
\end{tikzpicture}
\caption{\label{fig:-> over +}The flip path $P_\epsilon(t,t^{v,\rightarrow})$ when $l(w)=l(v)+1$, where $w$ is the vertex that becomes the new parent of $v$ in $t^{v,\rightarrow}$.}
\end{figure}
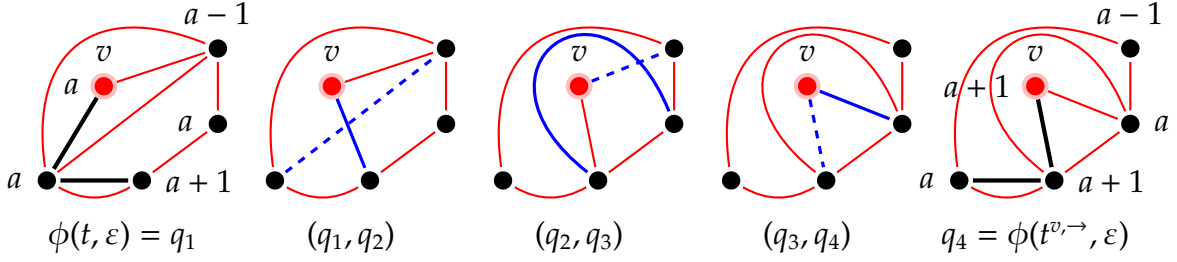

\begin{lemma}[Flip path construction, leaf translation onto smaller label]\label{leaf translation construction} Consider  a tree $t\in\LT_n$ with contour $c_1,\ldots,c_{2n}$ and let $v$ be a leaf of $t$ such that $l(v)=l(p(v))$, adjacent to a corner $c_l$ with $2\leq l\leq 2n-1$. Let $w$ be the vertex of the corner $c_{l+2}$ and further suppose that $l(w)=l(v)-1$.

For $\epsilon\in\{-1,1\}$, let $e_1,\ldots,e_{k-1}$ be the edges of $\phi(t,\epsilon)$ that are adjacent to $w$ and lie strictly between the edge $e_k=(c_l,c_{l+2})$ and the edge $(c_{l+2},t(c_{l+2}))$, in clockwise order around $w$. If none of them is the root edge of $\phi(t,\epsilon)$, then we can set $P_\epsilon(t,t^{v,\rightarrow})=(q_i,e_i,-)_{i=1}^k$, and we have $q_{k+1}:=q_k^{e_k,-}=\phi(t^{v,\rightarrow},\epsilon)$.
 
 If some $e_j$ is the root edge of $\phi(t,\epsilon)$, then we can set $P_\epsilon(t,t^{v,\rightarrow})$ to be the same flip path as above, with the root rotation flip sequence from Lemma~\ref{root rotation}, performed around the endpoint of the root edge that is not $w$, inserted right before the flip $(q_j,e_j,-)$.
 
See Figure~\ref{fig:P(t,t^{v,->})} for the construction.
\end{lemma}
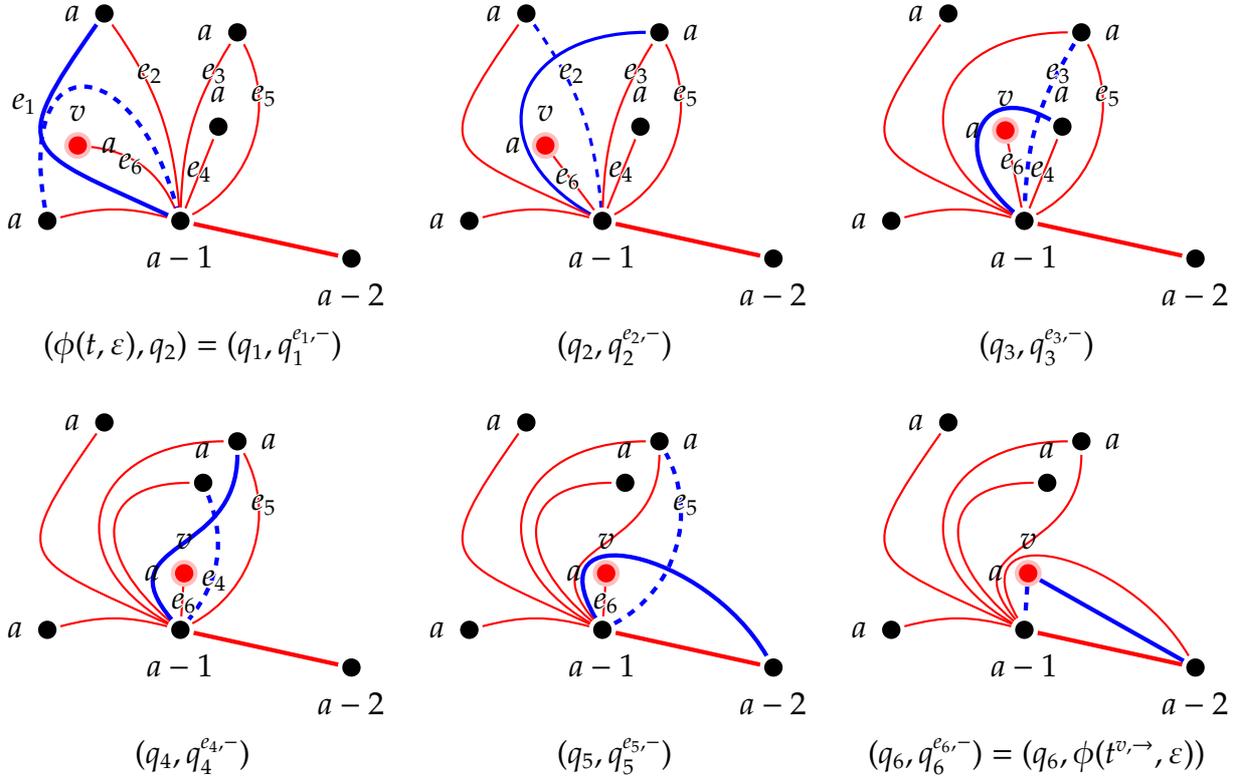
\begin{figure}[t]
\centering\begin{tabularx}{.9\textwidth}{>{\centering\arraybackslash}X>{\centering\arraybackslash}X>{\centering\arraybackslash}X}
 & & \\
\begin{tikzpicture} 
	\begin{pgfonlayer}{nodelayer}
		\node [style=vertex, label=left:$a$] (0) at (-2, -0) {};
		\node [style=vertex, fill=red, draw=pink, label=right:$a$, label=above:$v$] (1) at (-1.6, 1) {};
		\node [style=vertex, label=below:$a-1$] (2) at (-0.25, -0) {};
		\node [style=vertex, label=left:$a$] (3) at (-1.25, 2.75) {};
		\node [style=vertex, label=left:$a$] (4) at (0.5, 2.5) {};
		\node [style=vertex, label=below:$a-2$] (5) at (2, -0.5) {};
		\node [style=vertex, label=above:$a$] (6) at (0.25, 1.25) {};
	\end{pgfonlayer}
	\begin{pgfonlayer}{edgelayer}
		\draw [style=map, bend left, looseness=1.00] (1) to (2);
		\draw [style=map, bend left=15, looseness=1.00] (3) to (2);
		\draw [style=map, bend right=45, looseness=1.00] (2) to (4);
		\draw [style=map, ultra thick] (5) to (2);
		\draw [blue, ultra thick, dashed, in=105, out=100, looseness=3.25] (0) to (2);
		\draw [style=map, bend left=15, looseness=1.00] (0) to (2);
		\draw [style=map, bend right=15, looseness=1.00] (4) to (2);
		\draw [style=map] (6) to (2);
		\draw [blue, ultra thick, in=155, out=-125, looseness=2] (3) to (2);
		
		\node (e1) at (-0.9, 0.75) {\contour{white}{$e_6$}};
		\node (e1) at (-2.3, 1.55) {\contour{white}{$e_1$}};
		\node (e2) at (-0.65, 1.95) {\contour{white}{$e_2$}};
		\node (e3) at (0.2, 1.95) {\contour{white}{$e_3$}};
		\node (e4) at (0, 0.65) {\contour{white}{$e_4$}};
		\node (e5) at (0.85, 1.65) {\contour{white}{$e_5$}};
	\end{pgfonlayer}
\end{tikzpicture} &
\begin{tikzpicture} 
	\begin{pgfonlayer}{nodelayer}
		\node [style=vertex, label=left:$a$] (0) at (-2, -0) {};
		\node [style=vertex, fill=red, draw=pink, label=left:$a$, label=above:$v$] (1) at (-1, 1) {};
		\node [style=vertex, label=below:$a-1$] (2) at (-0.25, -0) {};
		\node [style=vertex, label=left:$a$] (3) at (-1.25, 2.75) {};
		\node [style=vertex, label=right:$a$] (4) at (0.5, 2.5) {};
		\node [style=vertex, label=below:$a-2$] (5) at (2, -0.5) {};
		\node [style=vertex, label=above:$a$] (6) at (0.25, 1.25) {};
	\end{pgfonlayer}
	\begin{pgfonlayer}{edgelayer}
		\draw [style=map] (1) to (2);
		\draw [blue, very thick, dashed, bend left=15, looseness=1.00] (3) to (2); 
		\draw [blue, very thick, in=180, out=150, looseness=1.80] (2) to (4); 
		\draw [style=map, bend right=45, looseness=1.00] (2) to (4);
		\draw [style=map, ultra thick] (5) to (2);
		\draw [style=map, bend left=15, looseness=1.00] (0) to (2);
		\draw [style=map, bend right=15, looseness=1.00] (4) to (2);
		\draw [style=map] (6) to (2);
		\draw [style=map, in=155, out=-125, looseness=2] (3) to (2);
		
		\node (e1) at (-0.7, 0.55) {\contour{white}{$e_6$}};
		\node (e2) at (-0.65, 1.95) {\contour{white}{$e_2$}};
		\node (e3) at (0.2, 1.95) {\contour{white}{$e_3$}};
		\node (e4) at (0, 0.65) {\contour{white}{$e_4$}};
		\node (e5) at (0.85, 1.65) {\contour{white}{$e_5$}};
	\end{pgfonlayer}
\end{tikzpicture} &
\begin{tikzpicture} 
	\begin{pgfonlayer}{nodelayer}
		\node [style=vertex, label=left:$a$] (0) at (-2, -0) {};
		\node [style=vertex, fill=red, draw=pink, label=left:$a$, label=above:$v$] (1) at (-0.5, 1.2) {};
		\node [style=vertex, label=below:$a-1$] (2) at (-0.25, -0) {};
		\node [style=vertex, label=left:$a$] (3) at (-1.25, 2.75) {};
		\node [style=vertex, label=right:$a$] (4) at (0.5, 2.5) {};
		\node [style=vertex, label=below:$a-2$] (5) at (2, -0.5) {};
		\node [style=vertex, label=above:$a$] (6) at (0.25, 1.25) {};
	\end{pgfonlayer}
	\begin{pgfonlayer}{edgelayer}
		\draw [style=map] (1) to (2);
		\draw [style=map, in=180, out=150, looseness=1.80] (2) to (4);
		\draw [style=map, bend right=45, looseness=1.00] (2) to (4);
		\draw [style=map, ultra thick] (5) to (2);
		\draw [style=map, bend left=15, looseness=1.00] (0) to (2);
		\draw [blue, ultra thick,dashed, bend right=15, looseness=1.00] (4) to (2); 
		\draw [blue, ultra thick, in=150, out=140, looseness=2.2] (2) to (6);
		\draw [style=map] (6) to (2);
		\draw [style=map, in=155, out=-125, looseness=2] (3) to (2);
		
		\node (e1) at (-0.4, 0.75) {\contour{white}{$e_6$}};
		\node (e3) at (0.2, 1.95) {\contour{white}{$e_3$}};
		\node (e4) at (0, 0.65) {\contour{white}{$e_4$}};
		\node (e5) at (0.85, 1.65) {\contour{white}{$e_5$}};
	\end{pgfonlayer}
\end{tikzpicture}\\
$(\phi(t,\epsilon),q_2)=(q_1,q_1^{e_1,-})$ & $(q_2,q_2^{e_2,-})$ & $(q_3,q_3^{e_3,-})$\\
 & & \\
 \begin{tikzpicture} 
	\begin{pgfonlayer}{nodelayer}
		\node [style=vertex, label=left:$a$] (0) at (-2, -0) {};
		\node [style=vertex, fill=red, draw=pink, label=left:$a$, label=above:$v$] (1) at (-0.2, 0.75) {};
		\node [style=vertex, label=below:$a-1$] (2) at (-0.25, -0) {};
		\node [style=vertex, label=left:$a$] (3) at (-1.25, 2.75) {};
		\node [style=vertex, label=right:$a$] (4) at (0.5, 2.5) {};
		\node [style=vertex, label=below:$a-2$] (5) at (2, -0.5) {};
		\node [style=vertex, label=above:$a$] (6) at (0.05, 1.95) {};
	\end{pgfonlayer}
	\begin{pgfonlayer}{edgelayer}
		\draw [style=map] (1) to (2);
		\draw [style=map, in=180, out=150, looseness=1.80] (2) to (4);
		\draw [style=map, bend right=45, looseness=1.00] (2) to (4);
		\draw [style=map, ultra thick] (5) to (2);
		\draw [style=map, bend left=15, looseness=1.00] (0) to (2);
		\draw [style=map, in=180, out=140, looseness=1.8] (2) to (6);
		\draw [ultra thick, blue, dashed, bend left] (6) to (2); 
		\draw [ultra thick, blue, in=-90, out=130,looseness=1.50] (2) to (4);
		\draw [style=map, in=155, out=-125, looseness=2] (3) to (2);
		
		\node (e1) at (-0.2, 0.35) {\contour{white}{$e_6$}};
		\node (e4) at (0.2, 0.65) {\contour{white}{$e_4$}};
		\node (e5) at (0.85, 1.65) {\contour{white}{$e_5$}};
	\end{pgfonlayer}
\end{tikzpicture}&
\begin{tikzpicture} 
	\begin{pgfonlayer}{nodelayer}
		\node [style=vertex, label=left:$a$] (0) at (-2, -0) {};
		\node [style=vertex, fill=red, draw=pink, label=left:$a$, label=above:$v$] (1) at (-0.2, 0.75) {};
		\node [style=vertex, label=below:$a-1$] (2) at (-0.25, -0) {};
		\node [style=vertex, label=left:$a$] (3) at (-1.25, 2.75) {};
		\node [style=vertex, label=right:$a$] (4) at (0.5, 2.5) {};
		\node [style=vertex, label=below:$a-2$] (5) at (2, -0.5) {};
		\node [style=vertex, label=above:$a$] (6) at (0.05, 1.95) {};
	\end{pgfonlayer}
	\begin{pgfonlayer}{edgelayer}
		\draw [style=map] (1) to (2);
		\draw [style=map, in=180, out=150, looseness=1.80] (2) to (4);
		\draw [ultra thick, blue, dashed, bend right=45, looseness=1.00] (2) to (4);
		\draw [ultra thick, blue, in=120, out=120, looseness=1.80] (2) to (5);
		\draw [style=map, ultra thick] (5) to (2);
		\draw [style=map, bend left=15, looseness=1.00] (0) to (2);
		\draw [style=map, in=180, out=140, looseness=1.8] (2) to (6);
		\draw [style=map, in=-90, out=130,looseness=1.50] (2) to (4);
		\draw [style=map, in=155, out=-125, looseness=2] (3) to (2);
		
		\node (e1) at (-0.2, 0.35) {\contour{white}{$e_6$}};
		\node (e5) at (0.85, 1.65) {\contour{white}{$e_5$}};
	\end{pgfonlayer}
\end{tikzpicture}&
\begin{tikzpicture} 
	\begin{pgfonlayer}{nodelayer}
		\node [style=vertex, label=left:$a$] (0) at (-2, -0) {};
		\node [style=vertex, fill=red, draw=pink, label=left:$a$, label=above:$v$] (1) at (-0.2, 0.75) {};
		\node [style=vertex, label=below:$a-1$] (2) at (-0.25, -0) {};
		\node [style=vertex, label=left:$a$] (3) at (-1.25, 2.75) {};
		\node [style=vertex, label=right:$a$] (4) at (0.5, 2.5) {};
		\node [style=vertex, label=below:$a-2$] (5) at (2, -0.5) {};
		\node [style=vertex, label=above:$a$] (6) at (0.05, 1.95) {};
	\end{pgfonlayer}
	\begin{pgfonlayer}{edgelayer}
		\draw [blue, ultra thick, dashed] (1) to (2); 
		\draw [blue, ultra thick] (1) to (5);
		\draw [style=map, in=180, out=150, looseness=1.80] (2) to (4);
		\draw [style=map, in=120, out=120, looseness=1.80] (2) to (5);
		\draw [style=map, ultra thick] (5) to (2);
		\draw [style=map, bend left=15, looseness=1.00] (0) to (2);
		\draw [style=map, in=180, out=140, looseness=1.8] (2) to (6);
		\draw [style=map, in=-90, out=130,looseness=1.50] (2) to (4);
		\draw [style=map, in=155, out=-125, looseness=2] (3) to (2);
		
	\end{pgfonlayer}
\end{tikzpicture}\\
$(q_4,q_4^{e_4,-})$ & $(q_5,q_5^{e_5,-})$ & $(q_6,q_6^{e_6,-})=(q_6,\phi(t^{v,\rightarrow},\epsilon))$\\
\end{tabularx}
\caption{\label{fig:P(t,t^{v,->})}The flip path $P_\epsilon(t,t^{v,\rightarrow})$ in the case where the edge $(v,p(v))$ is followed by a corner labelled $p(v)-1$ in the clockwise contour of $t$.}
\end{figure}
\begin{proof}
	One can show inductively that, for $i=2,\ldots,k-1$, the quadrangulation $q_i$ is obtained from $q_1$ by collapsing the face that contains the edge $e_k$ in $q_1$ and replacing the edge $e_i$ with a degenerate face whose internal vertex is adjacent to $w$; furthermore, the natural edge identification between $q_1$ and $q_i$ has $e_k$ correspond to the internal edge of the degenerate face, while $e_i$ corresponds to the `rightmost' boundary edge of the newly created degenerate face in clockwise order around $w$. Then, $q_k$ is obtained by flipping $e_{k-1}$ counterclockwise. By Lemma~\ref{leaf translation static}, the only difference between $q_k$ and $\phi(t^{v,\rightarrow},\epsilon)$ is the fact that the internal edge of the new degenerate face is incident to $w$ in $q_k$ and, potentially, the choice of the root edge (in the case where the root edge of $q_1$ is among the flipped edges $e_1,\ldots,e_{k-1}$). Flipping $e_k$ in $q_k$ -- thus obtaining $q_{k+1}$ -- is enough to fix the first issue, and yields $\phi(t^{v,\rightarrow},\epsilon)$ up to rerooting. Suppose now that the root edge of $q_1$ is some $e_j$ (with $1\leq j\leq k-1$); the edge issued from the same corner in $q_{k+1}$ is actually the flipped version of edge $e_{j-1}$ (or the edge out of $c_{l+1}$, which never gets flipped, in the case where $l=2$: for ease of notation we will call it $e_0$). In the quadrangulation $q_j$, $e_j$ (not yet flipped and still the root edge) and $e_{j-1}$ (already flipped, unless $j=1$) are consecutive in clockwise order around their endpoint that is not $w$. Performing the root rotation sequence before $(q_j,e_j,-)$ thus simply results in rerooting $q_j$ in its edge $e_{j-1}$ (with the correct orientation), which will not be flipped again and will end up being the correct root edge in $q_{k+1}$ once the rest of the flips are performed.
	\end{proof}

As mentioned before, we construct $P_\epsilon(t,t^{v,\rightarrow})$ in general as the concatenation of $P_\epsilon(t,t^{v,=})$, $P_\epsilon(t^{v,=},(t^{v,=})^{v,\rightarrow})$ and $P_\epsilon((t^{v,=})^{v,\rightarrow},((t^{v,=})^{v,\rightarrow})^{v,c})$, where $c$ is chosen so as to ``restore'' the colour of $(v,p(v))$ to the original one from $t$. We can further set $P_\epsilon(t^{v,\rightarrow},(t^{v,\rightarrow})^{v,\leftarrow})$ to be the reverse sequence of $P_\epsilon(t,t^{v,\rightarrow})$ (keeping in mind that $P_\epsilon(t,t)$ is already set to be empty). Notice that, using Lemma~\ref{label move path length} and the fact that the construction from Lemma~\ref{leaf translation construction}, excluding root rotation sequences, does not flip the same edge twice, we immediately have \begin{equation}\label{translation sequence length}\left|P_\epsilon(t^{v,d},t)\right|\leq 6n+17\end{equation}
for all $t\in\LT_n$, $v$ leaf of $t$, $d\in\{\rightarrow,\leftarrow\}$, $\epsilon\in\{-1,1\}$.

Additionally, we have the following lemma.
 
\begin{lemma}\label{leaf translation congestion}
	Let $(q,e,s)$ be a triple with $q\in\sQ^\bullet_n$, $e\in E(q)$, $s\in\{+,-\}$; then there is a constant $C$ such that there are at most $C$ quadruples $(t,v,d,\epsilon)$, where $t\in\LT_n$, $v$ is a leaf in $t$, $d\in \{\rightarrow, \leftarrow\}$ and $\epsilon=\pm1$, for which $(q,e,s)$ appears in the flip path $P_\epsilon(t,t^{v,d})$.
\end{lemma}

\begin{proof}
By lemma \ref{label move congestion}, the number of such quadruples is at most a constant times the number of those where $l(v)=l(p(v))$, so we shall restrict ourselves to the latter case; since $P_\epsilon(t,t^{v,\rightarrow})$ is the reverse of $P_\epsilon(t^{v,\rightarrow},t)$, we may also suppose $d=\rightarrow$. 

Suppose $(t,v,\rightarrow,\epsilon)$ is a quadruple such that $l(v)=l(p(v))$ and $(q,e,s)$ appears in $P_\epsilon(t,t^{v,\rightarrow})$, and let $c_1,\ldots,c_{2n}$ be the clockwise contour of $t$, with $c_l$ being the corner of $v$. If $l(c_{l+2})=l(v)$, then $q=\phi(t,\epsilon)$ and $v$ is uniquely determined from the flip $(q,e,s)$. 

If $l(c_{l+2})=l(v)+1$, then there are a few possibilities (refer again to Figure~\ref{fig:-> over +}). If $q$ has one more degree one vertex than $q^{e,s}$, then $q=\phi(t,\epsilon)$ and $v$ is that degree one vertex. If $q^{e,s}$ has one more degree one vertex than $q$, let $e'$ be the edge issued from that vertex; then $v$ is the vertex in question and $\phi(t^{v,\rightarrow},\epsilon)=(q^{e,s})^{e',-}$. Otherwise, $e$ is a double edge within a degenerate face, $q^{e,s}=\phi(t^{v,\rightarrow},\epsilon)$ and $v$ is its degree one endpoint.

Finally, if $l(c_{l+2})=l(v)-1$, then consider the edge $e$ in $q^{e,s}$. First, let's suppose that $(q,e,s)$ does not belong to a root rotating sequence. If $e$ is the interior edge of a degenerate face, then $\phi(t^{v,\rightarrow},\epsilon)=q^{e,s}$ and $v$ is its endpoint of degree 1. Otherwise, consider the two endpoints $w_1$, $w_2$ of $e$ in $q^{e,s}$; one of them must play the role of the vertex $w$ from Lemma~\ref{leaf translation construction}. If it is $w_1$, then to the right of the oriented edge $(w_1,w_2)$ in $q^{e,s}$ lies a degenerate face with an internal vertex connected to $w_1$. If $(w_1,u)$ is the first oriented edge in clockwise order around $w$ starting with $(w_1,w_2)$ such that $u$ is strictly nearer to $\delta$ than $w_1$, the construction of the flip path implies that $\phi(t^{v,\rightarrow},\epsilon)$ can be obtained from $q^{e,s}$ by collapsing the face lying right of $(w_1,w_2)$ and replacing $(w_1,u)$ with a new degenerate face whose internal edge is issued from $u$ and has $v$ as the other endpoint. An analogous argument holds for $w_2$, giving rise to only two possibilities for $(t,v,\rightarrow)$. If $(q,e,s)$ actually belongs to a root rotation sequence, then the the number of possibilities for the final quadrangulation $q'$ obtained by completing the sequence is bounded by a constant; we can then use $q'$ and its root edge in place of $q^{e,s}$ and $e$ to reconstruct the final quadrangulation $\phi(t^{v,\rightarrow},\epsilon)$  and its vertex $v$ (from which one determines $t$ as $(t^{v,\rightarrow})^{v,\leftarrow}$) in (at most) two ways.\end{proof}

\subsection{The final comparison between $\F^{n,\bullet}$ and $\VLT$}\label{section: final comparison}

\begin{proof}[Proof of Theorem~\ref{main theorem}]
The upper bound for the spectral gap of $\F^n$ is Proposition~\ref{upper bound}; we set out to prove the lower bound. 

Consider the chains $\VLT$ and $\F^{n,\bullet}$ from Section~\ref{chain comparison} and their respective spectral gaps $\widetilde{\gamma}$ and $\nu_n^\bullet$, and let $\nu_n$ be the spectral gap of $\F^n$. Also recall the Schaeffer correspondence $\phi:\LT_n\times\{-1,1\}\to\sQ_n^\bullet$ from Section~\ref{section: Schaeffer} and the flip paths $P(t)$ and $P_\epsilon(t,t^{v,x})$ constructed throughout Sections~\ref{section: rerooting}, \ref{section: colour change} and~\ref{section: leaf translation}.

Let $f:\sQ_n^\bullet\to\mathbb{R}$ be a function such that $\V_{\pi^\bullet}(f)=1$ (where $\pi^\bullet$ is the uniform measure on $\sQ_n^\bullet$) and $\mathcal{E}_{\F^{n,\bullet}}(f,f)=\nu_n^\bullet$. Then thanks to Corollary~\ref{leaf translation variant gap} we have

$$Cn^{-\frac{9}{2}}\leq\widetilde{\gamma}\leq\mathcal{E}_\VLT(f\circ\phi,f\circ\phi)=\frac12\sum_{\substack{t,t'\in\LT_n\\\epsilon,\epsilon'\in\{-1,1\}}}\left(f(\phi(t,\epsilon))-f(\phi(t',\epsilon'))\right)^2\frac{1}{2|\LT_n|}p_\VLT(t,t')$$
$$=\frac12\sum_{\substack{t\in\LT_n\\v\mbox{ leaf of }t\\x\in\{\rightarrow,\leftarrow,+,-,=\}\\\epsilon\in\{1,-1\}
}}\left(f(\phi(t,\epsilon))-f(\phi(t^{v,x},\epsilon))\right)^2\frac{1}{2|\LT_n|}\frac{1}{5(n+1)}+
\sum_{\substack{t\in\LT_n
}}\left(f(\phi(t,1))-f(\phi(t,-1))\right)^2\frac{1}{2|\LT_n|}\frac{1}{n+1}.$$

Now we may rewrite each difference within the sums above in terms of the images of subsequent quadrangulations appearing in the paths $P_\epsilon(t,t^{v,x})$ and $P(t)$, apply the Cauchy-Schwarz inequality and tweak the constants in order to recover factors of the form $p_{\F^{n,\bullet}}(q,q')$. We obtain that the expression above is at most
$$\frac{1}{2|\LT_n|}\frac{1}{n}\left(\sum_{\substack{t\in\LT_n\\v\mbox{ leaf of }t\\x\in\{\rightarrow,\leftarrow,+,-,=\}\\\epsilon\in\{1,-1\}}}
\left(\sum_{\substack{i=0,\ldots,|P_\epsilon(t,t^{v,x})|\\P_\epsilon(t,t^{v,x})=(q_i,e_i,s_i)_{i=1}^N}}
f(q_i)-f(q_i^{e_i,s_i})\right)^2+
\sum_{\substack{t\in\LT_n}}
\left(\sum_{\substack{i=0,\ldots,|P(t)|\\P(t)=(q_i,e_i,s_i)_{i=1}^N}}f(q_i)-f(q_i^{e_i,s_i})\right)^2\right)$$

$$\leq6\sum_{\substack{t\in\LT_n\\v\mbox{ leaf of }t\\x\in\{\rightarrow,\leftarrow,+,-,=\}\\\epsilon\in\{1,-1\}}}
|P_\epsilon(t,t^{v,x})|
\sum_{\substack{i=0,\ldots,|P_\epsilon(t,t^{v,x})|\\P_\epsilon(t,t^{v,x})=(q_i,e_i,s_i)_{i=1}^N}}
(f(q_i)-f(q_i^{e_i,s_i}))^2\frac{1}{|\sQ^\bullet_n|}\frac{1}{6n}$$
$$+6\sum_{\substack{t\in\LT_n}}|P(t)|
\sum_{\substack{i=0,\ldots,|P(t)|\\P(t)=(q_i,e_i,s_i)_{i=1}^N}}(f(q_i)-f(q_i^{e_i,s_i}))^2\frac{1}{|\sQ^\bullet_n|}\frac{1}{6n}.$$

Given $q\in\sQ_n^\bullet$, $e\in E(q)$, $s\in\{+,-\}$, write $C(q,e,s)$ for 
$$\left|\{t,v,x,\epsilon\st(q,e,s)\mbox{ appears in }P_\epsilon(t,t^{v,x})\}\right|+\left|\{t\st(q,e,s)\mbox{ appears in }P(t)\}\right|.$$
Thanks to Lemma~\ref{root reversal estimates}, Lemma~\ref{label move congestion} and Lemma~\ref{leaf translation congestion}, there is a constant $M$ independent of $n$ such that $C(q,e,s)\leq M$ for all $q,e,s$. Furthermore, Lemma~\ref{root reversal estimates}, Lemma~\ref{label move path length} and \eqref{translation sequence length} imply that $\max\{|P(t)|,|P_\epsilon(t,t^{v,x})|\}\leq 7n$ for all $n\geq 17$, $t\in\LT_n$, $v$ leaf of $t$, $x\in\{\rightarrow,\leftarrow,+,=,-\}$, $\epsilon\in\{-1,1\}$. From this we obtain that, for some constant $C'$,
$$Cn^{-\frac{9}{2}}\leq \frac{C'n}{2} \sum_{\substack{q\in\sQ^\bullet_n\\e\in E(q)\\s\in\{+,-\}}}C(q,e,s)(f(q)-f(q^{e,s}))^2\frac{1}{|\sQ^\bullet_n|}\frac{1}{6n}\leq
C'n M\cdot\mathcal{E}_{\F^{n,\bullet}}(f,f)=C'n M\nu_n^\bullet\leq C'M n \nu_n,$$
where the last inequality follows from Lemma~\ref{pointed comparison}. Following the chain of inequalities, we have indeed shown that $\nu_n\geq C_1n^{-\frac{11}{2}}$ for some constant $C_1$ independent of $n$.
\end{proof}

\bibliographystyle{siam}

\end{document}